\documentclass[UTF-8,reqno]{amsart}
\usepackage{enumerate}
\usepackage{mhequ}
\usepackage[margin=1.2in]{geometry}
\usepackage{amssymb,url,color, booktabs,nccmath}
\usepackage{mathrsfs}
\usepackage{enumitem}
\usepackage{graphicx}
\usepackage{tikz}
\usetikzlibrary{shapes,snakes}
\usetikzlibrary{calc}
\usetikzlibrary{decorations.shapes}

\usepackage[colorinlistoftodos,prependcaption,color=yellow,textsize=tiny,textwidth=2cm]{todonotes}
\usepackage{color}
\usepackage[colorlinks=true]{hyperref}
\hypersetup{
    linkcolor=blue,          
    citecolor=red,        
    filecolor=blue,      
    urlcolor=cyan
}

\definecolor{darkergreen}{rgb}{0.0, 0.5, 0.0}
\definecolor{darkblue}{RGB}{0,0,139}

\numberwithin{equation}{section}

\newcommand{\be}{\begin{eqnarray}}
\newcommand{\ee}{\end{eqnarray}}
\newcommand{\ce}{\begin{eqnarray*}}
\newcommand{\de}{\end{eqnarray*}}
\newtheorem{theorem}{Theorem}[section]
\newtheorem{lemma}[theorem]{Lemma}
\newtheorem{remark}[theorem]{Remark}
\newtheorem{definition}[theorem]{Definition}
\newtheorem{proposition}[theorem]{Proposition}

\newtheorem{assumption}[theorem]{Assumption}

\allowdisplaybreaks

\tikzset{
        dot/.style={circle,fill=black,inner sep=0pt, outer sep=0.7pt, minimum size=1mm},
        Phi/.style={white!40!red,thick,snake=coil,segment amplitude=0.6pt, segment length=2pt},
         Z/.style={black!40!green,thick,snake=coil,segment amplitude=0.6pt, segment length=2pt},
        C/.style={thick,black!20!blue},
          Cr/.style={thick,black!20!red},
            Cg/.style={thick,black!20!green},
       }

\usepackage{mathtools,amssymb,amsfonts,amsthm,amscd,amsmath,amsgen,upref,enumerate,nicefrac,pdfpages}

\begin{document}
\newcommand{\Q}{\mathbb{Q}}
\newcommand{\R}{\mathbb{R}}
\newcommand{\DD}{\mathbb{D}}
\newcommand{\cD}{\mathcal{D}}
\newcommand{\cG}{\mathcal{G}}
\newcommand{\TND}{\mathbb{T}^d_N}
\newcommand{\TD}{\mathbb{T}^d}
\newcommand{\I}{\mathbb{I}}
\newcommand{\Z}{\mathbb{Z}}
\newcommand{\N}{\mathbb{N}}
\newcommand{\F}{\mathcal{F}}
\newcommand{\p}{\mathbb{P}}
\newcommand{\hs}{\hspace{1cm}}
\newcommand{\XX}{\mathbb{X}}
\newcommand{\M}{\mathcal{M}}
\newcommand{\Pp}{\mathcal{P}}
\newcommand{\W}{\mathcal{W}}
\newcommand{\ED}[1]{{#1}^\epsilon_\delta}
\newcommand{\EDM}[1]{{#1}^\epsilon_{\delta,m}}
\newcommand{\EOM}[1]{{#1}^\epsilon_{0,m}}
\newcommand{\EDMg}[1]{{#1}^{\epsilon,g}_{\delta,m}}
\newcommand{\RR}{\mathcal{R}}
\newcommand{\E}{\mathbb{E}}
\newcommand{\h}{\mathcal{H}}
\newcommand{\cL}{\mathcal{L}}
\newcommand{\Ii}{\mathrm{I}}
\newcommand{\eps}{\epsilon}
\newcommand{\supp}{\mathrm{supp}}
\newcommand{\law}{\mathrm{Law}}
\newcommand{\leb}{\lambda}
\newcommand{\id}{\mathrm{id}}
\newcommand{\pr}{\mathrm{pr}}
\newcommand{\Ss}{\mathfrak{S}}
\newcommand{\inter}{\mathrm{int\,}}
\newcommand{\dhcomment}[1]{\textbf{\textcolor{red}{#1}}}
\newcommand{\PP}{\mathbb{P}}

\newcommand{\Ent}{\textrm{Ent}}

\newcommand{\norm}[1]{\| #1 \|}
\newcommand{\abs}[1]{|#1|}
\newcommand{\ud}[1]{\, \mathrm{d} #1}
\newcommand{\dx}{\ud{x}}
\newcommand{\dy}{\ud{y}}
\newcommand{\dxi}{\ud \xi}
\newcommand{\deta}{\ud{\eta}}
\newcommand{\dr}{\ud{r}}
\newcommand{\drp}{\ud{r'}}
\newcommand{\dxp}{\ud{x'}}
\newcommand{\dxip}{\ud{\xi'}}
\newcommand{\dyp}{\ud{y'}}
\newcommand{\ds}{\ud{s}}
\newcommand{\dt}{\ud{t}}
\newcommand{\dz}{\ud{z}}
\newcommand{\dd}{\ud}
\newcommand{\C}{\textrm{C}}
\newcommand{\ve}{\varepsilon}
\newcommand{\sgn}{\textrm{sgn}}
\newcommand{\mcS}{\mathcal{S}}
\newcommand{\Supp}{\textrm{Supp}}

\renewcommand{\d}{\delta}

\title[LDP for the ZRP in the whole space]{\Large Matching Large Deviation Bounds of the Zero-Range Process in the whole space}

\author{Benjamin Fehrman}
\author{Benjamin Gess}
\author{Daniel Heydecker}

\address{B. Fehrman:  Department of Mathematics, Louisiana State University, United States.}
\email{fehrman@math.lsu.edu}
\address{B. Gess: Institut f\"ur Mathematik, Technische Universit\"at Berlin, 10623 Berlin, Germany \\ \& Max Planck Institute for Mathematics in the Sciences, 04103 Leipzig, Germany.}
\email{Benjamin.Gess@mis.mpg.de}
\address{D. Heydecker: Department of Mathematics, Imperial College London, London, SW7 2AZ.}
\email{d.heydecker@imperial.ac.uk}

\subjclass[2010]{}


\begin{abstract}
We consider the large deviations of the hydrodynamic rescaling of the zero-range process on $\mathbb{Z}^d$ in any dimension $d\ge 1$. Under mild and canonical hypotheses on the local jump rate, we obtain matching upper and lower bounds, thus resolving the problem opened by \cite{KL99}. On the probabilistic side, we extend the superexponential estimate to any dimension, and prove the superexponential concentration on paths with finite entropy dissipation. In addition, we extend the theory of the parabolic-hyperbolic skeleton equation to the whole space, and remove global convexity/concavity assumptions on the nonlinearity.

\end{abstract}

\maketitle

\setcounter{tocdepth}{1}
\tableofcontents

\section{Introduction \& Main Result} \label{sec: SOR} We consider the large deviations of the zero-range process (ZRP) on the infinite lattice $\mathbb{Z}^d$. Given a local jump rate\footnote{In the literature, the jump rate is often written $g$. We have deviated from this convention in order to avoid potential confusion with an $L^2_{t,x}$-control, which has also been denoted $g$ in previous papers \cite{FehGes19}.} $\lambda:\mathbb{N}\to [0,\infty)$, the zero-range process is informally the Markov process on configurations $X:=\mathbb{N}^{\mathbb{Z}^d}$ where, at a jump rate $2d\lambda(\eta(x))$, a particle jumps from the lattice site $x$ to a randomly chosen neighbour $y$. This is made precise by the generator \begin{equation}
	\label{eq: generator} LF(\eta):=\sum_{x,y\in \mathbb{Z}^d, x\sim y} \lambda(\eta(x))(F(\eta+1_y-1_x)-F(\eta))
\end{equation} whenever $F:X\to \mathbb{R}$ is any function depending on only finitely many sites, and $1_x, 1_y$ are the configurations with single particles at $x,y$.  \\ \\ The zero-range process is one of the classical models of statistical mechanics, and under mild assumptions it is well-known \cite{BKL95}, \cite{KL99} that the hydrodynamic rescaling of the empirical density field
 $$ \pi^N_t:=N^{-d}\sum_{x\in \mathbb{Z}^d} \eta_{N^2t}(x)\delta_{x/N} $$ 
converges to a solution to a nonlinear diffusion equation $$ \partial_t \rho=\Delta \varphi(\rho) $$ for a nonlinearity $\varphi(\rho)$ given in terms of the equilibrium measures, see \eqref{eq: def varphi} below. 

\medskip\noindent
The goal of the current work is to describe the behaviour beyond the law of large numbers, and to obtain matching upper and lower bounds for the large deviations of $\pi^N_t$. 
 While an in depth treatment has been obtained by Benois-Kipnis-Landim \cite{BKL95} and Kipnis-Landim \cite{KL99}, two important points remain unresolved,  of which one requires probabilistic, and the other one analytic argumentation. These two key difficulties are described in the next paragraphs, and the main goal of this work is to resolve them in order to reach the full large deviation principle displayed in Theorem \ref{thrm: main result} below.
 
 \medskip\noindent
 The first key difficulty appears in the proof of the superexponential estimate, see Theorem \ref{thrm: supex}, in both arbitrary dimension and infinite volume. In trying to follow the classical proof \cite{KOV89,BKL95,KL99}, the step which becomes problematic is the {\em two-block estimate}, where one must estimate in mean the difference of averages on large microscopic boxes for lattice points at a small macroscopic distance, under a measure admitting a density $f_N$ with respect to the invariant measure, in terms of a Dirichlet form $\mathfrak{d}(f_N)$. The proof in dimension $d=1$ \cite{KOV89,BKL95} can be seen to resemble the deduction of continuity from $H^1$-Sobolev regularity, and therefore cannot be directly modified to $d\ge 2$. To make the connection precise, we show in Proposition \ref{prop: no 2 block} that the respective estimate used in \cite{BKL95} for $d=1$ is false in $d\ge 2$ as a result of the failure of the Sobolev embedding $H^1(\R^d)\not \hookrightarrow C(\R^d)$. In the course of the proof we unveil a connection between the two-block estimate and the macroscopic entropy dissipation functional \eqref{eq: entropy diss def} which arises in the analysis of the limiting fluctuations, see equations \eqref{eq: connect dirichlet form to fischer information 1}, \eqref{eq: connect dirichlet form to fischer information 2}. As a result, the {\em only} quantities which can be controlled by Dirichlet form estimates are the probabilistic analogues of those controlled by the PDE analysis in \cite{FehGes19}: This principle appears to be quite general, and also applies to lattice systems not of zero-range type.

\medskip\noindent
  Existing works in the literature which treat general dimension \cite{KL99, Qu.Re.Va1999} avoid this issue in the torus $\TD, d\ge 1$ through an averaging procedure, which tacitly benefits from the geometry by averaging over a {\em compact, transitively acting symmetry group}; in the setting of the  full space, compactness is lost, and the proof cannot be repeated. \\  In order to resolve this issue, we introduce a further approximation, which is agnostic to both dimension and the global domain geometry. In Lemma \ref{lemma: typical f}, we show that the variational problem appearing in the whole superexponential estimate gains {\em regularity in the probability space from the regularity of test functions} in an $N$-uniform way. More precisely, thanks to testing against a regular test function $H\in C^1_c([0,T]\times \R^d)$, a supremum over all probability density functions $f_N$ may be restricted to those $f_N$ enjoying a uniform bound on the contributions of any one edge $\mathfrak{d}_{x,y}(f_N)$ to the Dirichlet form, up to an error which vanishes, uniformly in $N$, as the cutoff is removed. Extending the analogy to the Sobolev embedding, imposing such a cutoff corresponds to working with $W^{1,\infty}$-type regularity rather than $H^1$, and this additional restriction provides the missing regularity necessary to close the two-block estimate in any dimension and in the full space.
 
 \medskip\noindent
 The second key observation is that the existing bounds found in the literature \cite{BKL95,KL99} do not determine a {\em full} large deviation principle, in that the lower bound is restricted to a set of sufficiently regular fluctuations, or equivalently that different rate functions appear in the upper and lower bounds: We refer to the definitions preceding Theorem \ref{thrm: partial LDP}. This problem is discussed in \cite{FehGes19,heydecker2023large, gess2025porousmediumequationlarge}, and examples have been found in other models by the third author \cite{heydecker2023large} where such a candidate rate function does not capture the full large deviation principle. In the context of the zero-range process, in finite volume, and under the assumption that $\varphi^{1/2}$ is convex or concave, this problem has recently been solved by the first two named authors \cite{FehGes19}, by proving the well-posedness of weak solutions to the \emph{skeleton equation} \begin{equation}
	\label{eq: sk} \partial_t \rho=\Delta \varphi(\rho)-\nabla\cdot(\varphi^{1/2}(\rho)g), \qquad g\in (L^2_{t,x})^d. 
\end{equation}  
These results have subsequently been applied to prove a full large deviations principle for a zero-range process with rescaled particle-size about the porous medium equation by two of the authors \cite{gess2025porousmediumequationlarge}. However, when applied to the zero-range process in the classical hydrodynamical scaling, the validation that the macroscopic diffusivity $\varphi^{1/2}$ is convex or concave leads to restrictive assumptions on the local junp rate $\lambda$. The second goal is thus to extend the theory of \eqref{eq: sk} to cover infinite volume, and to ensure that the only hypotheses needed on $\varphi^{1/2}$ can be derived from the mild, canonical hypotheses \eqref{eq: Lipschitz}, \eqref{eq: spectral gap} below on the microscopic jump rate $\lambda$. \\
This analysis culminates in Theorem \ref{thm:envelope}, which resolves the difficulty of matching bounds. Following \cite{FehGes19}, the uniqueness required to extend the lower bound must be proven at the level of renormalised solutions, which informally satisfy the equation \eqref{eq: sk} after cutting off large values, see Definition \ref{skel_sol_def_rel}. Meanwhile, the upper bound remains expressed in the facially weaker theory of weak solutions of \eqref{eq: sk}, and the goal of obtaining matching bounds therefore demands to show that these two solution theories coincide. It is here that the properties of the nonlinearity $\varphi$ play a pivotal role: In passing from a weak solution $\rho$ to the renormalised framework, commutator errors appear by estimating $\varphi(\rho)^\epsilon \approx \varphi(\rho^\epsilon)$, where $\cdot^\epsilon$ denotes mollification on scale $\epsilon$. The required estimation of these commutators developed in \cite{FehGes19} relies on the strong hypothesis that $\varphi^{1/2}$ is either convex or concave, a condition that does not seem to be implied by any natural assumptions on the jump rate.\\
It is for this reason that we introduce a novel `defective concavity' property in Lemma~\ref{eq: defective concavity}, which may be deduced directly from the canonical hypotheses \eqref{eq: Lipschitz}, \eqref{eq: spectral gap} on $\lambda$.  The notion of `defective concavity' allows to substantially extend the applicability of \cite[Theorem~14]{FehGes19} to include all nonlinearities $\varphi$ that arise from jump rates satisfying \eqref{eq: Lipschitz}, \eqref{eq: spectral gap}.\\
Further, several new techniques are needed in order to extend the results of \cite{FehGes19} to infinite volume.  The solution theory is itself based on the kinetic form of the equation, which is a $L^1$-based theory.  However, in the whole space, the fluctuations obtained as limits of the zero-range process are only locally, and not globally, $L^1$-integrable, as a result of the infinite volume. For this reason, the proof of uniqueness for \eqref{eq: sk} in Theorem~\ref{thm_rel_unique} below requires to handle new cutoff errors at infinity through a careful use of interpolation and Sobolev inequalities.  In addition, the characterisation of the large deviations rate function in Theorem~\ref{thm:envelope} below requires an additional spatial approximation to remain in the space of smooth fluctuations given in Definition~\ref{smooth_fluctuation} below.  All of this adds to the intrinsic difficulties of treating equation \eqref{eq: sk}, for which a uniqueness theory based on $L^p$-analysis, $p>1$, is impossible due to the energy criticality (see, for example, \cite[Section~2.1]{FehGes19}), and must instead rely on the relative entropy estimate of Proposition~\ref{rel_prop0} below.  More details on these analytic issues can be found in Sections~\ref{sec_unique}, \ref{section_equivalence}, and \ref{sec_lsc_envelope} below. 

\medskip\noindent
We next give some definitions in order to give a precise statement of the main result in Theorem \ref{thrm: main result} below.    \\ \paragraph{\emph{Hypotheses on the local jump rate $\lambda$}} We first give the conditions on the local jump rate $\lambda$. We assume that \begin{enumerate}[label=\textbf{(A\arabic*)}] 
	\item $\lambda(k)>\lambda(0)=0$ for all $k>0$, and for some finite $c<\infty$, \begin{equation}\label{eq: Lipschitz} \lambda(k')-\lambda(k)\le c(k'-k) \hspace{1cm} \text{ for all } k, k' \in \mathbb{N}\text{  with  } k'\ge k.\end{equation} 
	\item (Spectral Gap Inequality) The function $\lambda$ is nondecreasing, and there exists $k\in \mathbb{N}$ and $\delta>0$ such that, for all $n\in \mathbb{N}$, \begin{equation}\label{eq: spectral gap} \lambda(n+k)\ge \lambda(n)+\delta.\end{equation} 
\end{enumerate} We study the large deviations on a finite time interval $[0,T]$, considering a Skorokhod space $\mathbb{D}$ and its topology defined in Section \ref{sec: prelim}. Our (candidate) rate function is given as follows. \begin{definition}[Rate Function] \label{def: rate function} For every $\pi \in \mathbb{D}$, for $Z$ as in \eqref{eq: partition} below, we set the relative entropy to be \begin{equation} \label{eq: relative entropy} \mathcal{H}_\Phi(\pi_0|\gamma):=\int_{\mathbb{R}^d}\Big\{\rho_0(x)\log \Big(\frac{\varphi(\rho_0(x))}{\varphi(\gamma)}\Big)-\log \Big(\frac{Z(\varphi(\rho_0(x)))}{Z(\varphi(\gamma))}\Big)\Big\}\end{equation} if $\pi_0$ admits a density $\rho_0$, and $\infty$ else. We now set \begin{equation}\label{lsc_01111'}
\mathcal{I}(\pi)=\mathcal{H}_\Phi(\pi_0|\gamma)+\frac{1}{2}\inf\Big\{\norm{g}^2_{L^2}\colon \partial_t\rho = \Delta\varphi(\rho)-\nabla\cdot(\varphi^\frac{1}{2}(\rho)g)\Big\}
\end{equation} if $\pi_t$ admits a density $\rho_t$ for all times, satisfying $\|\nabla \varphi^{1/2}(\rho)\|_{L^2_{t,x}}^2<\infty$, and where the infimum is taken over all controls $g\in (L^2_{t,x})^d$ such that $\rho$ solves the \emph{skeleton equation} in the sense of Definition \ref{def_sol_re}. If no such density exists, then $\mathcal{I}(\pi):=\infty$.
\end{definition} This rate function is different from the one appearing in \cite{BKL95} as a result of imposing the entropy dissipation condition $\|\nabla \varphi^{1/2}(\rho)\|_{L^2_{t,x}}^2<\infty$; we refer to the overview below Theorem \ref{thrm: partial LDP} for a precise discussion of the difference.  We may now give the main result. \begin{theorem} \label{thrm: main result}
	Let $\lambda$ be a local jump rate satisfying (A1-2), and let $\pi^N=(\pi^N_t)_{0\le t\le T}$ be the hydrodynamic rescaling of a zero-range process $(\eta_t, t\ge 0)$ defined on a probability space $(\Omega, \mathcal{F}, \mathbb{P})$, under which $\eta_0$ is sampled from an invariant measure $\nu_\gamma$, defined in \eqref{eq: nurho}, for some $\gamma>0$. Then for any open subset $\mathcal{U}\subset \DD$, \begin{equation}
		\liminf_N N^{-d}\log \mathbb{P}\left(\pi^N \in \mathcal{U}\right)\ge -\inf\left\{\mathcal{I}(\pi):\pi \in \mathcal{U}\right\}
	\end{equation} and for any closed set $\mathcal{A}\subset \DD$, \begin{equation}
		\limsup_N N^{-d}\log \mathbb{P}\left(\pi^N \in \mathcal{A}\right)\le -\inf\left\{\mathcal{I}(\pi): \pi \in \mathcal{A}\right\}.
	\end{equation}
\end{theorem} This paper is divided into two parts, corresponding to the twofold aims discussed above, as follows. Sections \ref{sec: prelim} - \ref{sec: entropy dissipation} contain the probabilistic content, generalising \cite{BKL95,KL99} to give a large deviation principle with non-matching upper and lower bounds in Theorem \ref{thrm: partial LDP}. Since the general method is well-established, we focus only on the technical aspects: the superexponential estimate Theorem \ref{thrm: supex} is proven in Section \ref{sec: supex}, and an entropy dissipation estimate Lemma \ref{lemma: finite entropy dissipation} is proven in Section \ref{sec: entropy dissipation}. The second part, in Sections \ref{sec_rel_skel_exist} - \ref{sec_lsc_envelope}, develops a theory of \eqref{eq: sk} in infinite volume, for a general nonlinearity $\Phi$ replacing $\varphi$. This leads up to Theorem \ref{thm:envelope}, which, under hypotheses covering $\varphi$ under only (A1-A2), shows that $\mathcal{I}$ coincides with its lower semicontinuous envelope on a restricted set. Together, Theorems \ref{thrm: partial LDP} and \ref{thm:envelope} thus establish the main result Theorem \ref{thrm: main result}. \section{Preliminaries} \label{sec: prelim} \subsection{Preliminaries on the Zero-Range Process}
We first give some preliminary facts about the zero-range process. Assuming only (A1) and that $\lambda(n)\to \infty$, the partition function given by  \begin{equation} \label{eq: partition} Z(\varphi):=\sum_{k=0}^\infty \frac{\varphi^k}{\prod_{1\le l\le k} \lambda(l)} \end{equation} is finite for all $0\le \varphi<\infty$. Under the same assumption, for every density $0\le 
\rho<\infty$, there exists a unique invariant measure ${\nu}_\rho$ on $X$, which is the product measure under which each value $\eta(x), x\in \Z^d$ has the distribution \begin{equation}
	\label{eq: nurho} {\nu}_\rho(\eta\in X: \eta(x)=k)=\frac{\varphi^k(\rho)}{Z(\varphi(\rho))\prod_{1\le l\le k}\lambda(l)}
\end{equation} where $\varphi(\rho)$ is (uniquely) chosen so that $\mathbb{E}_{\nu_\rho}[\eta(0)]=\rho$. The nonlinearity $\varphi$ may be recovered by \begin{equation} \label{eq: def varphi} \varphi(\rho)=\mathbb{E}_{\nu_\rho}[\lambda(\eta(0))]\end{equation} and it is a well-known piece of folklore that, under \eqref{eq: spectral gap}, the nonlinearity $\varphi$ is uniformly elliptic: for some $0<a\le A<\infty$, for all $\xi\in [0,\infty)$, \begin{equation}
	\label{eq: DH hyp on Phi'} a\le \varphi'(\xi)\le A.\end{equation} These measures satisfy the change-of-variables formula \begin{equation} \label{eq: COV} \varphi(\rho)\int_{X} F(\eta+1_x)\nu_\rho(d\eta)=\int_{X}\lambda(\eta(x))F(\eta)\nu_\rho(d\eta).\end{equation} We have the further consequence of the finiteness of the partition function, which will play a role similar to \cite[(A4)]{benois1995large}. \begin{lemma}\label{lemma: exp int}  There exists a convex function $w:[0,\infty)\to [0,\infty)$ such that $w(x)/x\to \infty$ as $x\to \infty$, and such that, for any $\rho\in[0,\infty)$, there exists $\theta>0$ such that \begin{equation}
	\label{eq: a4} \mathbb{E}_{\nu_\rho}\Big[e^{\theta w(\eta(0))}\Big]<\infty.
\end{equation}\end{lemma} In the rest of the paper, we will reserve the notation $\varphi$ for the function thus defined, which characterises the hydrodynamic limit and large deviations. We distinguish this from the notation $\Phi$ in \ref{sec_rel_skel_exist} - \ref{sec_lsc_envelope}, which we use to denote a more general nonlinearity. We similarly write $\cD(v)$ for the entropy dissipation functional for the particular nonlinearity $\varphi$, given by \begin{equation}\label{eq: entropy diss def} \cD(v):=\|\nabla \varphi^{1/2}(v)\|_{L^2_x}^2; \qquad v\in L^1_{\rm loc}(\mathbb{R}^d) \end{equation} which we allow to take the value $\infty$ if $\varphi^{1/2}(v)\not \in \dot{H}^1(\R^d)$. \subsection{Topological Definitions} We now introduce some topological spaces, which are deferred from the statement of Theorem \ref{thrm: main result}. We write $\mathcal{M}$ for the space of locally finite, nonnegative Radom measures on $\mathbb{R}^d$, equipped with a metric $d$ inducing the vague topology of convergence against test functions $f\in C_c(\mathbb{R}^d)$. We will take the path space \begin{equation} \label{eq: def skorokhod} \DD:=\DD([0,T],(\mathcal{M}, d))\end{equation} to be the Skorokhod space on a fixed time interval $[0,T]$. We will also write $\mathcal{C}:=\textrm{C}([0,T];(\mathcal{M}(\R^d),d))$ for those trajectories which are further continuous under the metric $d$.\\ \paragraph{\emph{Densities}} In the statement of Definition \ref{def: rate function}, we carefully distinguished between a trajectory $\pi \in \DD$ and its density $v$, should one exist. In order to ease statements in the sequel, we will ignore this difference. We therefore identify $L^1_{{\rm loc},+}:=L^1_{\rm loc}(\mathbb{R}^d, [0,\infty))$ with the absolutely continuous measures in $\mathcal{M}$ by identifying every $v$ with the measure $v(x)dx$ of which it is the density. We write $\DD_{\rm a.c.}$ to denote the Skorokhod space $\DD([0,T], (L^1_{{\rm loc},+},d))$, which is taken by the identification to those paths $\pi \in \DD$ which are absolutely continuous with respect to Lebesgue measure for all $0\le t\le T$. In this way, we may write $I(\rho)$ for functions $\rho_t(x)$. We use the symbol $\rho$ for a generic element of $\DD_{\rm a.c.}$, and $\pi$ for a generic element of $\DD$. When not specified otherwise, $L^1_{{\rm loc},+}$ will be understood to have the subspace topology induced by $d$; we will also speak of the strong Fr\'echet topology for the locally convex topology induced by the seminorms $v\mapsto \int_{|x|\le R} |v(x)|dx, R\in (0,\infty)$.\section{Superexponential Estimate in Infinite Volume} \label{sec: supex}
The goal of the following two sections is to give the probabilistic arguments which lead to large deviation upper and lower bounds, with (a priori) non-matching rate functions. For the upper bound, set, for $H\in C^{1,3}_{\rm c}([0,T]\times\mathbb{R}^d)$ and $\rho \in \DD_{\rm a.c.}$, \begin{equation}\begin{split}
 J(H, \rho)&:=\langle H_T,\rho_T\rangle - \langle H_0, \rho_0\rangle - \int_0^T\langle \partial_t H_t, \rho_t\rangle dt \\& \hspace{1cm}-\frac12 \int_0^T \int_{\mathbb{R}^d}\varphi(\rho_t(x))\left(\Delta H_t(x)+|\nabla H_t(x)|^2\right) dt dx   \end{split}\end{equation} if $\rho$ has finite entropy dissipation $\int_0^T \cD(\rho_t)dt<\infty$, and $\infty$ if the last integral is infinite, or if $\pi \in \DD\setminus \DD_{\rm a.c.}$. The initial upper bound will be given in terms of \begin{equation}
 	\label{eq: Ior} \mathcal{I}^{\rm up}(\pi):=\mathcal{H}_\Phi(\pi_0|\gamma)+\sup_{H\in C^{1,3}_{\rm c}([0,T]\times \mathbb{R}^d)} J(H, \pi).
 \end{equation}  For the lower bound, set $\mathcal{S}$ to be the set of $\rho\in \DD_{\rm a.c.}$ such that $\rho_0-\gamma \in C_c(\mathbb{R}^d)$, such that $\rho_0$ is strictly positive, and such that, for some compactly supported $H\in C^{1,3}_c([0,T]\times\mathbb{R}^d)$, $\rho_t$ solves the Fokker-Planck equation \begin{equation}
 	\partial_t \rho_t = \frac12\Delta \varphi(\rho_t)-\nabla\cdot(\varphi(\rho_t)\nabla H).
 \end{equation} Define the lower semicontinuous envelope \begin{equation}\label{eq: Ilo}
 	\mathcal{I}^{\rm lo}(\pi):=\inf \Big\{\liminf_n\Big\{\mathcal{H}(\pi^n_0|\gamma)+\frac12\int_0^T \int_{\mathbb{R}^d}\varphi(\rho^n)|\nabla H^n|^2 \Big\}: \rho^n \in \mathcal{S}, \rho^n \to \pi\Big\}
 \end{equation} where $H^n$ are the controls associated to the smooth approximating trajectories $\rho^n \in \mathcal{S}$. With these defined, the theorem is as follows. \begin{theorem}\label{thrm: partial LDP} Let $(\Omega, \mathcal{F}, \mathbb{P})$ be a probability space on which is defined a zero-range process $(\eta_t, t\ge 0)$ with a jump rate $\lambda$ satisfying (A1) and $\lambda(n)\to \infty$, and suppose that $\eta_0 \sim \nu_\gamma$, for some $\gamma\in (0,\infty)$. Let $\pi^N=(\pi^N_t)_{0\le t\le T}$ be the hydrodynamic rescaling. Then, for any closed set $\mathcal{A}\subset \DD$, \begin{equation}
		\limsup_N N^{-d}\log \mathbb{P}\left(\pi^N \in \mathcal{A}\right)\le -\inf\left\{\mathcal{I}^{\rm up}(\pi): \pi \in \mathcal{A}\right\}
	\end{equation} and for any open set $\mathcal{U}\subset \DD$, \begin{equation}
		\liminf_N N^{-d}\log \mathbb{P}\left(\pi^N \in \mathcal{U}\right)\ge -\inf\left\{\mathcal{I}^{\rm lo}(\pi):\pi \in \mathcal{U}\right\}.
	\end{equation} 
\end{theorem} As discussed in the introduction, the above result is the main probabilistic input for the central Theorem \ref{thrm: main result}, but leaves open the question of whether the two rate functions coincide. In the absence of global convexity, this question remained open from the original works \cite{BKL95,KL99} until it was resolved by the first two named authors \cite{FehGes19} in finite volume under certain assumptions on $\Phi$; let us also refer to work by the third author \cite{heydecker2023large} for an example where the (candidate) rate function differs from the lower semicontinuous envelope. For the current work, Theorem \ref{thm:envelope} shows that, only assuming (A1-2), both rate functions $\mathcal{I}^{\rm up}, \mathcal{I}^{\rm lo}$ coincide with the rate function $\mathcal{I}$ defined at \eqref{lsc_01111'}, from which Theorem \ref{thrm: main result} follows.
\paragraph{\textbf{Overview of the Proof of Theorem \ref{thrm: partial LDP}}.} The central pattern of deriving such bounds is identical to \cite[Theorem 1]{benois1995large}, and we only focus on the two technical elements. In this section, we show how the proof of the \emph{superexponential estimate} \cite[Theorem 2]{benois1995large} must be adapted to infinite volume in $d\ge 2$. This is the content of Theorem \ref{thrm: supex} below, and allows a proof of Theorem \ref{thrm: partial LDP} when the definition of $\mathcal{I}^{\rm up}$ is weakened, to only assume that each $\pi_t$ admits a locally integrable density $\rho_t$, without the entropy dissipation bound. This is addressed by Lemma \ref{lemma: finite entropy dissipation}, from which one may upgrade the upper bound to the one stated above.  We start with some definitions, where we recall that $X=\N^{\Z^d}$ is the state space.\begin{definition}
	A cylinder function $\Psi:X\to [0,\infty)$ is a function which depends only on $(\eta(x),x\in A)$ for a finite set $A$. We say that $\Psi$ is \emph{of linear growth} if, for some constant $C$ and all $\eta\in X$, \begin{equation}
		\label{eq: lin growth} |\Psi(\eta)|\le C\Big(1+\sum_{x\in A}\eta(x)\Big). 
	\end{equation}
\end{definition} For such a function, we define the function of a real variable \begin{equation}
	\label{eq: tildepsi} \widetilde{\Psi}(\rho):=\mathbb{E}_{\nu_\rho}\left[\Psi(\eta)\right]
\end{equation} where $\nu_\rho$ is the equilibrium measure (\ref{eq: nurho}). For $x\in \mathbb{Z}^d$, we set $\tau_x:X\to X$ to be the configuration $\tau_x \eta(y):=\eta(y-x)$, and define the translates of the cylinder function $\tau_x\Psi:=\Psi\circ \tau_x$. Finally, for $N\in \mathbb{N}, \epsilon>0$, we define the averaging $\bar{\eta}^{N\epsilon}$ at the macroscopic scale $\epsilon$ by $$ \bar{\eta}^{N\epsilon}(x):=(2\lfloor N\epsilon \rfloor +1)^{-d}\sum_{y: |y-x|_\infty \le \lfloor N\epsilon \rfloor } \eta(y).$$ With this notation fixed, the first technical result is as follows.
\begin{theorem}\label{thrm: supex}
	Let $H\in C^\infty_c([0,T]\times\mathbb{R}^d)$, and let $\Psi$ be a cylinder function of linear growth for which the function $\widetilde{\Psi}:[0,\infty)\to [0,\infty)$ is Lipschitz. For $\epsilon>0$ and $\eta\in X$, define \begin{equation}
		\label{eq: local error} V(H,\Psi, N,\epsilon)(t,\eta):=\frac{1}{N^d}\sum_{x\in \mathbb{Z}^d} H\big(t, \frac{x}{N}\big)\big[\tau_x\Psi(\eta)-\widetilde{\Psi}(\bar{\eta}^{N\epsilon}(x))\big].
	\end{equation}  Then \begin{equation}
		\label{eq: supex conclusion} \limsup_{\delta\to 0}\limsup_{\epsilon\to \infty} \limsup_{N\to\infty} N^{-d}\log \mathbb{P}\Big(\int_0^t\left|V(H,\psi,N,\epsilon)(t, \eta_{N^2t})\right| dt>\delta\Big)=-\infty. 
	\end{equation}
\end{theorem} While the technical condition we need on $\Psi$ is implicit, in that it is written in terms of the averaged function $\widetilde{\Psi}$, this still leaves a wide class of functions, including $\Psi$ which are Lipschitz in $(\eta(x), x\in A)$. In applying Theorem \ref{thrm: supex} to derive Theorem \ref{thrm: partial LDP}, the only important application will be $\Psi=\lambda(\eta(0))$, which is justified by observing that this is guaranteed by \eqref{eq: DH hyp on Phi'}.

The proof of the superexponential estimate closely follows the analagous arguments in \cite{benois1995large,kipnis2013scaling}: we will now sketch the relevant arguments, discuss why a new technical tool is needed in $d\ge 2$ in infinite volume, which is provided by Lemma \ref{lemma: typical f}, and how this will allow us to modify the arguments of \cite{BKL95} to conclude the theorem. \\ \\ First, we fix a smooth, integrable function $a:\mathbb{R}^d\to (0,\infty)$ with $\|\nabla \log a\|_\infty<\infty$, and for $w$ as in Lemma \ref{lemma: exp int}, we argue as in \cite[Lemma 4.1]{benois1995large} to find \begin{equation} \label{eq: density}
	\limsup_{z\to \infty}\limsup_N N^{-d}\log \mathbb{P}\Big(\frac{1}{N^d}\int_0^T \sum_{x \in \Z^d} a\big(\frac{x}{N}\big) w(\eta_{N^2t}(x)) dt >z\Big) =-\infty.
\end{equation} This allows us to deduce the theorem from showing the same statement for $V(H,\Psi,N,\epsilon,\beta)$ for any $\beta>0$, for the new functional \begin{equation} \label{eq: new functional} V(H,\Psi,N,\epsilon,\beta)(t,\eta):=V(H,\Psi,N,\epsilon)(t,\eta)-\frac{\beta}{N^d}\sum_{x \in \Z^d}a\big(\frac{x}{N}\big)w(\eta(x)). \end{equation} It is standard, using the Feynman-Kac formula (see, for example, the proof of \cite[Theorem 2.1]{KOV89}), to reduce the new statement to proving, for any $\delta>0$, \begin{equation}
	\label{eq: FK prototype} \limsup_{\epsilon\to 0}\limsup_{N\to \infty} \sup_f\left\{\mathbb{E}_{f\nu_\rho}[V(H,\Psi,N, \epsilon,\beta)(t,\eta)]-\delta N^{2-d}\mathfrak{d}(f)\right\}\le 0
\end{equation} where the supremum runs over probability density functions $f$ with respect to $\nu_\rho$, and $\mathfrak{d}$ is the Dirichlet form \begin{equation}\label{eq: dirichlet form}
	\mathfrak{d}(f)=\sum_{x\sim y}\mathfrak{d}_{x,y}(f):=\sum_{x\sim y}\int_{X} \lambda(\eta(x))\Big[\sqrt{f(\eta+1_y-1_x)}-\sqrt{f(\eta)}\Big]^2 \nu_\rho(d\eta).
\end{equation} For future convenience, we have introduced notation for the contribution $\mathfrak{d}_{x,y}(f)$ for any bond $(x,y)$. It is an elementary remark that $V(H,\Psi,N,\epsilon,\beta)\le C_\beta$ is bounded, using the linear growth of $\Psi$ and the superlinear growth of $w$, and hence the supremum may be restricted to those $f$ with $\mathfrak{d}(f)\le C_\beta \delta^{-1} N^{d-2}$. \\ \\  
The standard proof of the analogous statements to \eqref{eq: FK prototype} in \cite[Theorem 2]{benois1995large}, \cite[Lemma 1.10]{kipnis2013scaling} proceeds by decomposing the error into a `one-block estimate' \cite[Lemma 4.2]{benois1995large},\cite[Lemma 3.1]{kipnis2013scaling}, which shows that the same error vanishes on a large microscopic scale, and the `two-block estimate', considering two boxes of microscopic side length $\ell \to \infty$ at a small macroscopic distance $\epsilon\to 0$. The one-block estimate generalises with only cosmetic modifications to the present setting of infinite volume, and the only difficulty in proving Theorem \ref{thrm: supex} arises in the two-block estimate.  Precisely, the argument in \cite[Lemma 4.3]{benois1995large} shows in dimension $d=1$ that, for all $\beta, \delta>0$, and writing $\limsup_{\ell,\ve,N \to \infty}$ for the triple-limit $\limsup_{\ell\to \infty} \limsup_{\epsilon \to 0} \limsup_{N\to \infty}$,
\begin{align}\label{eq: prototype two block est}
& \nonumber \limsup_{\ell,\ve,N \to \infty} \hspace{0.1cm}\sup_{x,y\in \Z^d: |x-y|\le N\epsilon}\sup_f \left\{\mathbb{E}_{f\nu_\rho}[|\bar{\eta}^\ell(x)-\bar{\eta}^\ell(y)|-\beta \left(w(\bar{\eta}^\ell(x))+w(\bar{\eta}^\ell(y))\right)]-\delta N^{2-d}\mathfrak{d}(f)\right\}
\\ & \le 0,
\end{align}
which implies that
\begin{align}\label{eq: prototype two block est'}
\nonumber &\limsup_{\ell,\ve,N \to \infty}\hspace{0.1cm} \sup_{x\in \Z^d}\sup_f \Big\{\mathbb{E}_{f\nu_\rho}\Big[\frac{1}{(2\epsilon N+1)^d}\sum_{\ell <|y|\le \epsilon N}|\bar{\eta}^{N\epsilon}(x)-\bar{\eta}^\ell(x+y)|
\\ & \qquad \qquad  \qquad \qquad \qquad -\beta \left(w(\bar{\eta}^\ell(x))+w(\bar{\eta}^\ell(x+y))\right)\Big]-\delta N^{2-d}\mathfrak{d}(f)\Big\} \le 0.
\end{align}
The argument for the the previous two equations relies on joining any two microscopic boxes $x+[-\ell, \ell], y+[-\ell, \ell]$ at macroscopic distance $\epsilon$ by a path $x_0=x, x_1, x_2, \dots, x_L=y$ of length $L\sim N\epsilon$. A reduced Dirichlet form $\mathfrak{d}_{2,\ell}$, on probability density functions on $\mathbb{N}^{[-\ell, \ell]^2}$ is then estimated on the restriction $f^{2,\ell}_{x,y}$ of $f$ to the microscopic boxes by summing the contributions $\mathfrak{d}_{x_k, x_{k+1}}(f)$ along the path to yield \begin{equation} \label{eq: 2blockdirichlet}
		\mathfrak{d}_{2,\ell}(f^{2,\ell}_{x,y})\le (N\epsilon+3)\mathfrak{d}(f).
	\end{equation} This is the point at which the dimension $d=1$ is crucial to derive (\ref{eq: prototype two block est}). Using the restriction $\mathfrak{d}(f)\le C_\beta \delta^{-1}(N\epsilon+3)N^{d-2}=C_\beta \delta^{-1}(\epsilon+3N^{-1})$, we obtain a bound on $\mathfrak{d}_{2, \ell}(f^{2,\ell}_{x,y})$ which vanishes in the double limit where first $N\to \infty$ and then $\epsilon\to 0$, which allows one to replace the supremum over $f$ in the limit by a supremum over $f^{2,\ell}$ satisfying $\mathfrak{d}_{2,\ell}(f^{2,\ell})=0$. On the other hand, in higher dimension $d\ge 2$, one only obtains a bound $\mathfrak{d}_{2,\ell}(f^{2,\ell}_{x,y})\le C_\beta \delta^{-1} (N\epsilon+3)N^{d-2}$, which diverges in the same limit. Indeed, to illustrate the non-triviality of adapting this argument to higher dimensions, we will give the following proposition. \begin{proposition}[No 2-Block Estimate Without Averaging] \label{prop: no 2 block} Let $d\ge 2$. For sufficiently small $\beta>0$, the limit in \eqref{eq: prototype two block est} is strictly positive, no matter how large $\delta >0$ is. In dimension $d\ge 3$, the same applies to both limits \eqref{eq: prototype two block est}, \eqref{eq: prototype two block est'}. \end{proposition} The failure of the proof is a question of how many `atypical' edges there can be for a given profile $f$, in the sense that the contributions to the Dirichlet form are much larger than average, i.e. $\mathfrak{d}_{x,y}(f)\ge zN^{-d}\mathfrak{d}(f)$, for some large $z$. A simple Markov inequality shows that, for fixed $f$, the proportion of edges in $[-MN, MN]^d$ for which this holds is at most $CM^{-d}z^{-1}$ and vanishes as $z\to \infty$, uniformly in $f, N$. However, it may be the case that the small proportion of `atypical' edges prevents us joining two boxes $x+[-\ell, \ell]^d, y+[-\ell, \ell]^d$ by a path whose length is of the correct order, illustrated in Figure \ref{fig: atypical edges}. This is exactly analagous to the failure of embedding theorems $H^1(\Omega)\to C(\Omega)$ in dimensions $d\ge 2$. \begin{figure}[htbp]
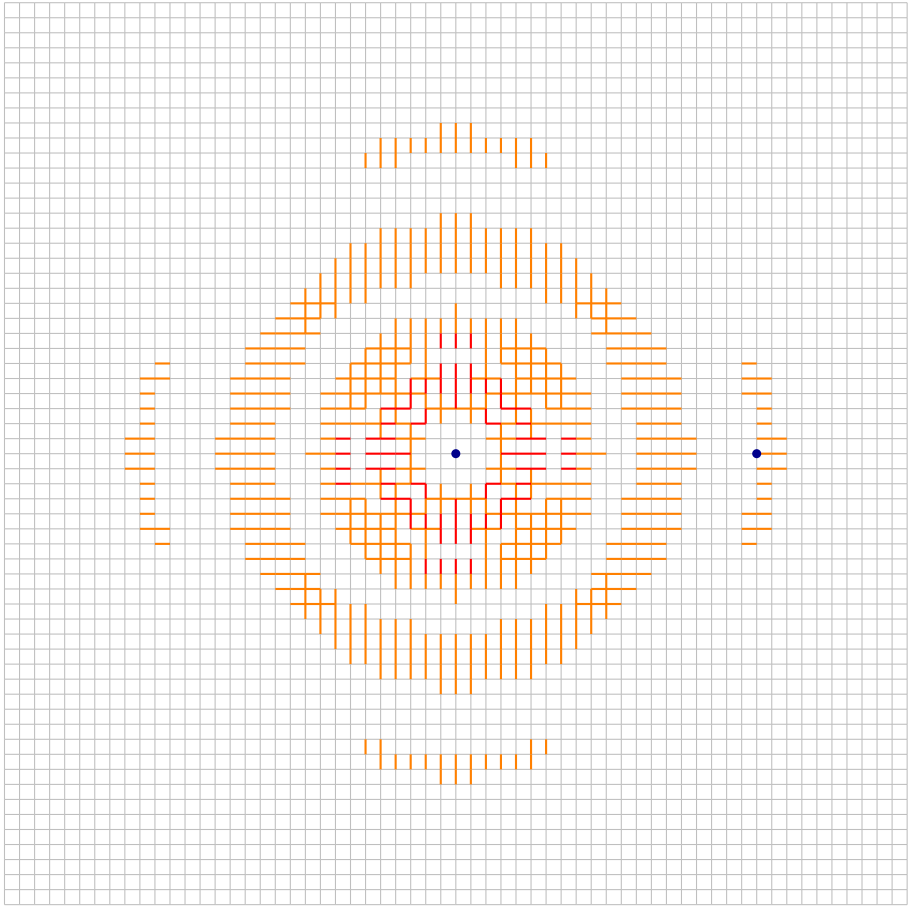
 \label{fig: atypical edges}
  \centering

  \caption{The geometric nature of the obstruction. Edges $\{x,y\}$ are highlighted in red, respectively orange, for which $\mathfrak{d}_{x,y}(f)$ lies in the 1$\%$ percentile, respectively top 10$\%$, for $f$ similar to the examples constructed in the proof of Proposition \ref{prop: no 2 block}. No path between $x=(0,0)$ and $y=(20,0)$ can avoid the red edges, nor a significant number of orange edges.}
  \label{fig:highlighted-edges}
\end{figure}

We will prove Proposition \ref{prop: no 2 block} at the end of this section by taking $f$ to be a slowly-varying local equilibrium for a profile $u$ with finite Fischer information, but which is neither continuous nor, in dimension $d\ge 3$, of vanishing mean oscillation.
    
    \noindent
    This problem is resolved for higher dimensions $d\ge 2$ in \cite[Lemma 3.2]{kipnis2013scaling}, but relying now on finite volume. The supremum in (\ref{eq: prototype two block est'}) is replaced by averaging over $x \in \mathbb{T}^d_N$ and taking the supremum over $|y|\le N\epsilon$, which makes the observable inside the expectation invariant under shifts: it therefore follows that the supremum is obtained at $f$ satisfying $\mathfrak{d}_{x,y}(f) \le CN^{-d}\mathfrak{d}(f)\le C_\beta \delta^{-1} N^{-2}$ for all neighbouring pairs $x,y$, and so \emph{all} edges are typical in the sense of the previous discussion. Once such a uniform bound is established, the same argument yields  \begin{equation} \label{eq: 2blockdirichlet'}
		\mathfrak{d}_{2,\ell}(f^{2,\ell}_{x,y})\le ((N\epsilon)^2+C\ell^d)(CN^{-d}\mathfrak{d}(f)) \le C_\beta \delta^{-1}(\epsilon^2+C\ell^dN^{-2})
	\end{equation} which again vanishes in the limit $N\to \infty, \epsilon\to 0$ as required. 
    
    \noindent
    In conclusion, this leaves as an open problem the two block estimate on the full space $\R^d$ for $d \ge 2$.
    
    \noindent
    The novel argument we give for (\ref{eq: FK prototype}) is as follows. In infinite volume, the same argument of global averaging is not available. Instead, we average only \emph{locally}. As a consequence of the regularity of the test function $H$, we will show that the original supremum (\ref{eq: FK prototype}) can be restricted to $f$ with a global `typicality' estimate $\mathfrak{d}_{x,y}(f)\le zN^{-2}$, up to a small error $\Gamma(z,H, \beta)$, which is uniform in $f,N$ and vanishes as $z\to \infty$. For fixed $z$, the same argument leading to (\ref{eq: 2blockdirichlet'}) yields an upper bound $Cz(\epsilon^2+C\ell^d N^{-2})$, and the same argument yields the two-block estimate. The one-block estimate follows by the same argument as \cite[Lemma 4.2]{benois1995large} in any dimension, and the only remaining term is the error $\Gamma(z,H,\beta)$, which vanishes as $z\to \infty$. The theorem will therefore follow from the same arguments as \cite{benois1995large,kipnis2013scaling}, which we have sketched above, once we prove the following result. 
    
    \begin{lemma}[Restriction to uniformly typical $f$]\label{lemma: typical f} Fix $H\in C^\infty_c([0,T]\times \mathbb{R}^d)$ and $\Psi$ as in Theorem \ref{thrm: supex}, $N<\infty, \epsilon>0, \delta>0$ and $\beta>0$. Then, for any $z<\infty$, there exists $\Gamma(z,H,\beta, \delta)$, which vanishes as $z\to \infty$, such that the supremum in \eqref{eq: FK prototype} is bounded by \begin{equation} \label{eq: regular f conclusion}
		\begin{split}
			& \limsup_{N\to \infty} \sup_f\left\{\mathbb{E}_{f\nu_\rho}[V(H,\Psi,N, \epsilon, \beta)(t,\eta)]-\delta N^{2-d}\mathfrak{d}(f)\right\} \\ &  \le \limsup_{N\to \infty} \sup_f\left\{\mathbb{E}_{f\nu_\rho}[V(H,\Psi,N, \epsilon,\beta_{z,\delta})(t,\eta)]-\delta N^{2-d}\mathfrak{d}(f): \mathfrak{d}_{x,y}(f)\le zN^{-2} \text{ for all }x\sim y\right\} \\ & \hspace{3cm} +\Gamma(z,H,\beta, \delta)
		\end{split}
	\end{equation} for some $0<\beta_{z,\delta}<\beta$ depending (only) on $z, \delta, \beta$.
		
	\end{lemma} \begin{proof}[Proof of Lemma \ref{lemma: typical f}] Fix $H, \Psi, N, \epsilon, \beta$ as in the statement, and fix $f$; as already remarked under \eqref{eq: dirichlet form}, we may harmlessly assume that $\mathfrak{d}(f)\le C_\beta \delta^{-1} N^{d-2}$. Let us fix $z<\infty$, and choose $\theta:=(C_\beta/\delta z)^{1/d}$ and $N_0(z, \delta)$ large enough that, for all $N\ge N_0$, $\frac{N}{2\lfloor N\theta\rfloor+1} \le \theta^{-1}$. For $N\ge N_0$, let $H_{N,\theta}$ be the function given by discrete convolution \begin{equation}
		H_{N,\theta}(z,t):=\frac{1}{(2\lfloor N\theta\rfloor +1)^d}\sum_{y\in \Z^d: |y|_\infty\le \lfloor N\theta\rfloor}H\big(z-\frac{y}{N},t\big).
	\end{equation} First, observe that $\|H_{N,\theta}-H\|_\infty\le \theta\|\nabla H\|_\infty$, uniformly in $N$. We next consider the error in substituting $H_{N,\theta}$ for $H$ in the optimisation problem \eqref{eq: FK prototype}, observing that $V(\cdot,\Psi,N,\epsilon, \beta)$ is an affine function. For any $\eta\in X$, the contribution from $V(H_{N,\theta}-H,\Psi, N\epsilon)$ is bounded above by $$ V(H_{N,\theta}-H,\Psi, N\epsilon) \le CM^d\left(1+\bar{\eta}^{MN}(0)\right)\|H_{N,\theta}-H\|_\infty$$ for $C$ depending only on $\Psi$ and where $M$ is the (macroscopic) size of the support of $H$. Meanwhile, the convexity and superlinear growth of $w$ mean that the entropy term of \eqref{eq: new functional} is at least $$ \frac{\beta}{N^d}\sum_{x\in \Z^d} a\big(\frac{x}{N}\big)w(\eta(x))\ge a_{\star,M}M^d\beta w(\bar{\eta}^{MN}(0)) $$ where $a_{\star,M}:=\inf_{|x|\le M}a$. Combining the two, the difference is at most\begin{equation} \begin{split}
&|V(H_{N,\theta},\Psi,N,\epsilon,\beta)(t,\eta)-V(H,\Psi,N,\epsilon,\beta)(t,\eta)| \\[1ex]& \hspace{2cm} \le  CM^d\|H-H_{N,\theta}\|_\infty\left(1+\bar{\eta}^{MN}(0)\right)-\beta a_{\star,M}M^d w\left(\bar{\eta}^{MN}(0)\right) \\[1ex] & \hspace{2cm} \le CM^d\theta\|\nabla H\|_\infty+\beta a_{\star,M} M^d w^\star\Big(\frac{C\theta\|\nabla H\|_\infty}{\beta a_{\star,M}}\Big) =: \Gamma(z,H,\beta, \delta)
	\end{split} \end{equation} where $w^\star$ denotes the Legendre transform of the convex function $w$, and where we recall that $\theta$ is a function of $z, \delta$. For any function $w$ as in Lemma \ref{lemma: exp int}, $w^\star$ is continuous at $0$ and $w^\star(0)=0$, whence this error converges to $0$ as $z\to \infty$. Moreover, the previously displayed estimate is uniform in $\eta\in X$, so we may integrate against any competitor $f\nu_\rho$ to find \begin{equation} \label{eq: H to HNT}
		\mathbb{E}_{f\nu_\rho}[V(H,\Psi,N,\epsilon,\beta)(t,\eta)]\le \mathbb{E}_{f\nu_\rho}[V(H_{N,\theta},\Psi,N,\epsilon,\beta)(t,\eta)]+\Gamma(z,H,\beta, \delta).
	\end{equation} We now examine the expectation on the right-hand side. Using the definition of $H_{N,\theta}$, we write
\begin{align}\label{eq: VHNT}
& \mathbb{E}_{f\nu_\rho}[V(H_{N,\theta},\Psi,N,\epsilon)]
 \\ \nonumber &=\frac{1}{N^d(2\lfloor N\theta\rfloor +1)^d}\sum_{x\in \Z^d}\sum_{|y|_\infty\le N\theta}H\big(\frac{x}{N}\big)\mathbb{E}_{f\nu_\rho}\left[\left((\tau_{x+y}\Psi)(\eta)-\widetilde{\Psi}(\bar{\eta}^{N\epsilon}(x+y)\right)\right]
\\ \nonumber & =\frac{1}{N^d(2\lfloor N\theta\rfloor +1)^d}\sum_{x\in \Z^d}\sum_{|y|_\infty\le N\theta} H\big(\frac{x}{N}\big)\mathbb{E}_{\tau_y f\nu_\rho}\left[\left((\tau_{x}\Psi)(\eta)-\widetilde{\Psi}(\bar{\eta}^{N\epsilon}(x)\right)\right]
\\ \nonumber &  = \mathbb{E}_{\bar{f}_{N\theta}\nu_\rho}\left[V(H,\Psi,N,\epsilon)\right],
\end{align}
where $\bar{f}_{N,\theta}$ is given by averaging $\tau_y f$ over $y\in \Z^d$ with $|y|_\infty \le \lfloor N\theta\rfloor$. In the negative term involving $w$, we observe that, for any $y$ with $ |y|_\infty\le \lfloor N\theta \rfloor$, \begin{equation}\begin{split}
		\mathbb{E}_{\tau_yf\nu_\rho}\Big[\frac{\beta}{N^d}\sum_{x\in \Z^d}a\big(\frac{x}{N}\big)w\left(\eta(x)\right)\Big]&=\mathbb{E}_{f\nu_\rho}\Big[\frac{\beta}{N^d}\sum_{x\in \Z^d}a\big(\frac{x-y}{N}\big)w(\eta(x))\Big] \\& \le \mathbb{E}_{f\nu_\rho}\Big[\frac{\beta e^{\theta c}}{N^d}\sum_{x\in \Z^d}a\big(\frac{x}{N}\big)w(\eta(x))\Big]\end{split}
	\end{equation} where $c<\infty$ is the Lipschitz constant of the composition $\log a$. Averaging over $y$, it follows that \begin{equation} \label{eq: negative term}
		\mathbb{E}_{f\nu_\rho}\Big[\frac{\beta}{N^d}\sum_{x\in \Z^d} a\big(\frac{x}{N}\big)w(\eta(x))\Big] \ge \mathbb{E}_{\bar{f}_{N,\theta}\nu_\rho}\Big[\frac{\beta e^{-\theta c}}{N^d}\sum_{x\in \Z^d}a\big(\frac{x}{N}\big)w(\eta(x))\Big].
	\end{equation}
		Finally, by convexity of the Dirichlet form, $\mathfrak{d}(\bar{f}_{N,\theta})\le \mathfrak{d}(f)$. Combining this with \eqref{eq: H to HNT}, \eqref{eq: VHNT}, \eqref{eq: negative term}, we find that \begin{align} \label{eq: f vs fnt}
			& \mathbb{E}_{f\nu_\rho}\left[V(H,\Psi,N,\epsilon,\beta)\right]-\delta N^{2-d}\mathfrak{d}(f) \\ \nonumber & \le 
			\mathbb{E}_{\bar{f}_{N,\theta}\nu_\rho}\left[V(H,\Psi,N,\epsilon,\beta e^{-\theta c})\right]-\delta N^{2-d}\mathfrak{d}(\bar{f}_{N,\theta}) +\Gamma(z,H,\beta, \delta).
		\end{align} Finally, let us show that $\bar{f}_{N,\theta}$ has the desired typicality property. For any bond $x,y$, we use the convexity of $\mathfrak{d}_{x,y}$ to see that \begin{equation} \begin{split} \mathfrak{d}_{x,y}(\bar{f}_{N,\theta}) & \le \frac{1}{(2\lfloor N\theta\rfloor +1)^d}\sum_{|z|\le \lfloor N\theta\rfloor} \mathfrak{d}_{x,y}(\tau_z f) \\& \le \frac{1}{(2\lfloor N\theta\rfloor +1)^d}\sum_{|z|\le \lfloor N\theta\rfloor} \mathfrak{d}_{x+z,y+z}(f) \le \frac{1}{(2\lfloor N\theta\rfloor +1)^d}\mathfrak{d}(f).\end{split}\end{equation} Since it was assumed at the beginning that $\mathfrak{d}(f)\le C_\beta \delta^{-1} N^{2-d}$, it follows that $$ \mathfrak{d}_{x,y}(\bar{f}_{N,\theta})\le C_\beta \delta^{-1} N^{-2}\theta^{-d} \le zN^{-2}$$ for all $N\ge N_0$, using the definition of $\theta=\theta(z, \delta), N_0=N_0(z, \delta)$ at the beginning. We may therefore replace the right-hand side of (\ref{eq: f vs fnt}) by the right-hand side of the conclusion (\ref{eq: regular f conclusion}), and since $f$ was arbitrary up to the bound on the Dirichlet form, we optimise over $f$ to obtain the conclusion (\ref{eq: regular f conclusion}).
	\end{proof}
	
	We now give the proof of Proposition \ref{prop: no 2 block}, deferred from earlier, which asserts that the two-block estimate cannot be proven in dimension $d\ge 2$ without the use of additional considerations such as the averaging procedure \cite{KL99} or considerations using the regularity of the test function. For consistency with the rest of the paper, we will write the proof in the whole-space case for $\Z^d$, but the same considerations would apply in the periodic setting if one does not first perform an averaging operation. \begin{proof}[Proof of Proposition \ref{prop: no 2 block}] We prove first the assertion regarding \eqref{eq: prototype two block est} in dimension $d\ge 2$. The assertion regarding \eqref{eq: prototype two block est'} in dimension $d\ge 3$ is similar, and is discussed at the end. Since $\delta$ is, in contrast to its role in the previous arguments, allowed to be arbitrarily {\em large}, we will denote it instead $M$. Throughout, $\gamma$ is the parameter of the equilibrium measure $\nu_\gamma$.\\ \\ Fix $\epsilon, M>0$, and consider the profile \begin{equation}
	\label{eq: bad profile in d=2} u(x):=\gamma\Big(1+\frac12\sin\big(\left|\log (\psi(|x|))\right|^{1/4}\big)\Big)
\end{equation} for some $\psi:[0,\infty)\to (0,1)$ satisfying $\psi'\in C_c((0,\infty))$, $0\le \psi'\le 1$, $\psi(r)\ge r/2$ on $\text{supp}(\psi')$, and $\psi(\infty)=e^{-(n\pi)^4}$ for some $n\in \mathbb{N}$, to be chosen. Since $\frac\gamma2\le u\le \frac{3\gamma}{2}$, the Fischer information $\|\nabla u^{1/2}\|_{L^2_x}^2$ is comparable to the homogeneous $\dot{H}^1_x$ norm $\|\nabla u\|_{L^2_x}^2$, which can be found by direct computation to be bounded above by \begin{equation}
	\label{eq: h1 of bad profile in d=2} \|\nabla u\|_{L^2_x}^2 \le \frac{C}{\sqrt{|\log \psi(\infty)|}}.
\end{equation} This bound does not depend on the lower cutoff $\psi(0)$, and indeed can be made arbitrarily small by shrinking $\psi(\infty)\to 0$. Under the assumption ({\bf A1}), the nonlinearity $\varphi$ associated to the zero-range process satisfies $\varphi(u)\le Cu$ for all $u$, and $\varphi(u)\ge cu$ on the bounded set $[\frac\gamma2,\frac{3\gamma}{2}]$, meaning that the above displayed inequality holds for the entropy dissipation and \begin{equation}
	\label{eq: entropy dissipation bad profilde d=2} \mathcal{D}(u)\le \frac{C}{\sqrt{|\log \psi(\infty)|}}
\end{equation} for some constant $C=C(\gamma, \varphi)$ independent of the choice of $\psi$. Now, given the parameters $\epsilon, M$, we can construct $\psi$, still satisfying the assumptions above, so that $\psi(r)=r$ on some subinterval $[\delta',\delta]\subset (0,\epsilon], |\log \delta'|^{1/4}=|\log\delta|^{1/4}+2\pi$ and with the upper cutoff $\psi(\infty)$ chosen so that $\mathcal{D}(u)< \frac{\gamma}{4M}$. Due to the cutoff away from $x=0$, we obtain in this way a smooth profile $u$.\\ \\ We now take, as a competitor in \eqref{eq: prototype two block est'}, the density $f_N=\frac{d\nu^N_u}{d\nu_\gamma}$ of a slowly-varying local equilibrium $\nu^N_u$ characterised by $\mathbb{E}_{\nu^N_u}\eta(x)=u(x/N)$; such a density exists thanks to the compact support of $u-\gamma$ and the strict positivity and boundedness of $u$, and is given explicitly by \begin{equation}
	\label{eq: competitor f} f(\eta):=\prod_{|x|_\infty \le MN} \frac{\varphi(u(x/N))^k Z(\varphi(\gamma))}{\varphi(\gamma)^k Z(\varphi(u(x/N))}
\end{equation} for some sufficiently large $M$ that $u=\gamma$ when $|x|_\infty\ge M$. Let us now estimate the Dirichlet form $\mathfrak{d}(f_N)$: for any edge $\{x,y\}$ of $\mathbb{Z}^d$, we may use the integration by parts formula \eqref{eq: COV} to see that \begin{equation}
	\mathfrak{d}_{x,y}(f_N)=\varphi(\gamma)\mathbb{E}_{\nu_\gamma}\Big[\Big(\sqrt{f_N(\eta+1_x)}-\sqrt{f_N(\eta+1_y)}\Big)^2\Big].
\end{equation} The formula \eqref{eq: competitor f} produces the simplification $f_N(\eta+1_x)=\frac{\varphi(u(x/N))}{\varphi(\gamma)}f_N(\eta)$, and using $\mathbb{E}f_N(\eta)=1$ produces the discrete gradient \begin{equation} \label{eq: connect dirichlet form to fischer information 1}
N^2\mathfrak{d}_{x,y}(f_N)=\Big(\frac{\sqrt{\varphi(u(x/N))}-\sqrt{\varphi(u(y/N))}}{1/N}\Big)^2.\end{equation} Summing over $x,y\in \mathbb{Z}^d$, using the smoothness of $u$ and $\varphi$ on bounded regions produces \begin{equation} \label{eq: connect dirichlet form to fischer information 2}
	\limsup_N \mathfrak{d}(f_N)=\mathcal{D}(u)
\end{equation} which is controlled by the parameter choices after \eqref{eq: entropy dissipation bad profilde d=2}, and in particular the Dirichlet form is at most $M^{-1}$ for all sufficiently large $N$. Fix some $\theta\in (0,1)$ to be chosen later; for all sufficiently large $\ell\in \mathbb{N}$, the weak law of large numbers implies that $\bar{\eta}^\ell(x)>\frac{3\gamma}{2}-2\theta$ with $\nu^N_u$-probability exceeding $1-\theta$ whenever $u(x'/N)>\frac{3\gamma}{2}-\theta$ for all $x'$ with $|x'-x|<\ell$, and similarly $\bar{\eta}^\ell(y)<\frac\gamma2+2\theta$ with $\nu^N_u$-probability exceeding $1-\theta$ whenever $u(y'/N)<\frac\gamma2+\theta$ for all $x'$ with $|x'-x|<\ell$. For any such $x, y$ and $\ell$ large enough, depending only on $\theta$, we get \begin{equation} \label{eq: choice of l}
	\mathbb{E}_{f_N \nu_\gamma}\left[|\bar{\eta}^\ell(x)-\bar{\eta}^\ell(y)|\right]>(\gamma-4\theta)(1-2\theta)
\end{equation} and the right-hand side can be fixed to be larger than $\frac{3\gamma}{4}$ by tuning $\theta$. It is easy to deduce from \eqref{eq: a4} that $z \mapsto \mathbb{E}_{\nu_z}[w(\eta(0))]$ is locally bounded on $z\in [0,\infty)$, and together with the convexity of $w$ and boundedness of $u$ uniformly in $N$, \begin{equation}\label{eq: choice of lambda} \mathbb{E}\left[\beta(w(\bar{\eta}^\ell(x))+w(\bar{\eta}^\ell(y)))\right]\le C\beta<\frac\gamma4 \end{equation} as soon as $\beta>0$ is small enough. \\ \\ We now substitute all of this into the problem \eqref{eq: prototype two block est}, with $\delta$ replaced by $M$. First, fix $\ell$ large enough as discussed before \eqref{eq: choice of l}, $M<\infty$, and $\beta>0$ according to \eqref{eq: choice of lambda}, and fix arbitrarily small $\epsilon>0$. The construction above allows us to choose $u$, subject to all of the requirements above and with $\mathcal{D}(u)\le \frac{\gamma}{4M}$, such that there exist $x,y$ with $|x-y|<\epsilon$ but $u(x)=\frac{3\gamma}{2}, u(y)=\frac\gamma2$; for this $u$, take $f_N$ as above. Thanks to continuity, and because the choice of $\ell$ was independent of $N$, for all sufficiently large $N$ there exist lattice points $x, y\in \mathbb{Z}^d$ with $|x-y|_\infty\le N\epsilon$ such that, whenever $|x-x'|_\infty\le \ell, |y-y'|_\infty\le \ell$, then $u(x'/N)>\frac{3\gamma}{2}-\theta, u(y'/N)<\frac\gamma2+\theta$. Thanks to the choice of $\ell$, \eqref{eq: choice of l}, \eqref{eq: choice of lambda} give \begin{equation}
	\mathbb{E}_{f_N\nu_\gamma}\left[|\bar{\eta}^\ell(x)-\bar{\eta}^\ell(y)|-\beta(w(\bar{\eta}^\ell(x))+w(\bar{\eta}^\ell(y)))\right] >\frac{\gamma}{2}
\end{equation} while $MN^{2-d} \mathfrak{d}(f_N)\le \frac{\gamma}{4}$. We thus conclude that, for fixed $\epsilon$ and all sufficiently large $\ell$, \begin{equation}
	\label{eq: 2 block counterexample end} \limsup_{N \to \infty} \hspace{0.1cm}\sup_{x,y\in \Z^d: |x-y|\le N\epsilon}\sup_f \left\{\mathbb{E}_{f\nu_\gamma}[|\bar{\eta}^\ell(x)-\bar{\eta}^\ell(y)|-\beta \left(w(\bar{\eta}^\ell(x))+w(\bar{\eta}^\ell(y))\right)]-M N^{2-d}\mathfrak{d}(f)\right\} \ge \frac{\gamma}{2}
\end{equation} and the proof of the claim regarding \eqref{eq: prototype two block est} is complete. The proof of the claim regarding \eqref{eq: prototype two block est'} is identical, starting instead from profiles of the form \begin{equation}
	\label{eq: bad profile d=3} u(x)=\frac\gamma2\left(2+\sin\left(\psi(|x|)^{-\alpha}\right)\right); \qquad 0<\alpha<\frac{d-2}{2}.
\end{equation} \end{proof}

\section{Finiteness of Entropy Dissipation}  \label{sec: entropy dissipation} We now prove that the rate function may be taken to be infinite on paths $\pi \in \DD$ which do not satisfy the regularity requirement in Definition \ref{def: rate function}, namely, that $\pi=\rho\in \DD_{\rm a.c.}$ and the density $\rho$ satisfies $$ \int_0^T \cD(\rho_s)ds = \left\|\nabla \varphi(\rho)\right\|_{L^2_{t,x}}^2<\infty. $$ We start with the existence of a density. \begin{lemma} 
	\label{lemma: eliminate dminusdac} Assume the notation of Theorem \ref{thrm: partial LDP}. Then $\pi^N$ are exponentially tight in the topology of $\DD$: for any $z<\infty$, there exists a compact set $\mathcal{K}_z\subset \DD$ such that \begin{equation}
		\label{eq: ET} \limsup_{N\to \infty} N^{-d}\log \PP\left(\pi^N\not \in \mathcal{K}_z\right)\le -z.
	\end{equation} Moreover, for every $\pi \in \DD\setminus \DD_{\rm a.c.}$ and every $z<\infty$, there exists an open set $\mathcal{U}\ni \pi$ such that \begin{equation}
		\label{eq: eliminate dminusdac} \limsup_N N^{-d}\log \PP\left(\pi^N\in \mathcal{U}\right)\le -z.
	\end{equation}
\end{lemma} This allows us to consistently set $I(\pi)=\infty$ on such paths, as promised in Section \ref{sec: prelim}. This appears to be well-known; indeed, it is implicit in the rate function of \cite{BKL95}, but since we have not found a proof in the literature, the details are sketched for completeness. 

\begin{proof}[Sketch Proof of Lemma \ref{lemma: eliminate dminusdac}] The exponential tightness is the same as in \cite{benois1996large}, using the Lipschitz bound on the jump rates. The same argument also produces, for any $\varphi\in C^\infty_c(\R^d)$ and $\epsilon>0$, \begin{equation}
	\limsup_{\delta\downarrow 0}\limsup_{N\to \infty}N^{-d}\log \PP\left(\exists s, t\le T: |s-t|\le \delta, |\langle \varphi, \pi^N_s-\pi^N_t\rangle|>\epsilon\right)=-\infty
\end{equation} which proves the same statement \eqref{eq: eliminate dminusdac} when $\pi \not \in \mathcal{C}$.\\ \\   We now turn to the absolute continuity condition. For the same function $w$ as in Lemma \ref{lemma: exp int}, and all $R>0$, let us set $$ H_{w,R}(\mu):=\int_{|x|\le R}w\big(\frac{d\mu}{dx}\big) dx $$ which we understand to be infinite if $\mu(\cdot \cap \{|x|\le R\})\not \in L^1_{{\rm loc},+}$. Since $w$ is convex, each $H_{w,R}$ is convex and lower semicontinuous, and admits the dual representation $$ H_{w,R}(\mu)=\sup_{f\in C^\infty(B_R)} \Big(\langle f, \mu\rangle - \int_{\mathbb{R}^d}w^\star(f(x))dx\Big) $$where the supremum runs over smooth functions supported on $\{|x|<R\}$, valid whether or not $\mu \in L^1_{{\rm loc},+}$, and where $w^\star$ is the Legendre transform of the convex function $w$. We now fix $\pi \in \DD\setminus \DD_{\rm a.c.}$; if $\pi \not \in \mathcal{C}$, then the conclusion \eqref{eq: eliminate dminusdac} is already proven in the course of the tightness estimate, so we need only to deal with the case $\pi \in \mathcal{C}$. In this case, there exists $t\le T$ and for all $z<\infty$ there exists $\varphi \in C^\infty_c(\mathbb{R}^d)$ such that $$ \langle \varphi, \pi_t\rangle-\int_{\mathbb{R}^d} w^\star(\varphi(x))dx>z. $$ Thanks to the continuity assumption, the same is true, with $z+1$ in place of $z+2$, for all $s$ sufficiently close to $t$, and we may find a function $\psi \in C([0,T], [0,\infty))$ with $\int_0^T \psi ds=1$ such that $$ \int_0^T \psi(s)\Big(\langle \varphi, \pi_s\rangle ds-\int_{\mathbb{R}^d} w^\star(\varphi(x))dx\Big)>z+1. $$ The functional defined by the left-hand side is now readily seen to be continuous with respect to the topology of $\DD$, so the set $\mathcal{U}$ of $\pi'$ where the above inequality holds is an open set containing $\pi$. For the empirical measure $\pi^N$ of the particle system, using convex duality again $$ \langle \varphi, \pi^N_s\rangle -\int_{\mathbb{R}^d}w^\star(\varphi(x))dx \le N^{-d} \sum_{|x|_\infty \le RN} w(\eta_{N^2 t}(x))-\epsilon_N $$ for some error $\epsilon_N \to 0$, which depends on the continuity of $w^\star$ on the (compact) range of $\varphi$ and $\sup_{|x-y|\le d/N}|\varphi(y)-\varphi(x)|$, and where $R$ is such that $\text{supp}(\varphi)\subset[-R,R]^d$. We now restrict to $N$ sufficiently large that $\epsilon_N<1$, for which  $$ \{\pi^N\in \mathcal{U}\}\subset \Big\{\int_0^T \psi(s)N^{-d}\sum_{|x|_\infty \le RN}w(\eta_{N^2 s}(x))ds >z\Big\}.$$ For the same $\theta>0$ as in Lemma \ref{lemma: exp int}, we use Jensen's inequality to obtain\begin{equation*}
	\begin{split}
		\mathbb{E}\Big[\exp\Big(\theta \int_0^T \psi(s)\sum_{|x|_\infty \le RN}w(\eta_{N^2 s}) ds\Big) \Big] &  \le 	\mathbb{E}\Big[\int_0^T \theta \psi(s) \exp\Big( \sum_{|x|_\infty \le RN}w(\eta_{N^2 s}(x)) \Big) ds \Big] \\ & = 	\mathbb{E}\Big[\exp\Big(\theta \sum_{|x|_\infty \le RN}w(\eta_0(x)) \Big) ds \Big]
	\end{split}
\end{equation*} where the final line follows using stationarity of the process and the normalisation $\int \psi ds=1$. Using Lemma \ref{lemma: exp int}, the right-hand side is at most $e^{CN^d}$, for some fixed $C$, and a Chebychev estimate produces $$N^{-d} \log \PP\left(\pi^N\in \mathcal{U}\right)\le C-\theta z. $$ Since the right-hand side can be made arbitrarily negative, we are done.  \end{proof} The main result of this section, which completes the proof of Theorem \ref{thrm: partial LDP}, is as follows.
	\begin{lemma} \label{lemma: finite entropy dissipation} Let $\rho\in \mathbb{D}_{\rm a.c.}$ be such $\int_0^T \mathcal{D}(\rho_s)ds=\infty.$ Then for every $z<\infty$ there exists an open set $\mathcal{U}\ni u$ such that \begin{equation}
		\limsup_N N^{-d}\log \mathbb{P}\left(\pi^N \in \mathcal{U}\right)\le -z. 
	\end{equation}  \end{lemma}  Our proof is based on the following representation of the functional $v\mapsto \|\nabla \varphi^{1/2}(v)\|_{L^2_x}^2 \in [0,\infty]$, which may be proven with some elementary Hilbert space theory.
\begin{lemma}[Variational Representation of $\cD$]\label{lemma: variational rep} For $v\in L^1_{{\rm loc},+}$, we have the equality 
\begin{equation}
	\label{eq: variational rep 1} \frac14\cD(v)=\sup\Big\{\int_{\mathbb{R}^d}\varphi(v(x))(-\Delta H(x)-|\nabla H|^2) dx: H \in C^2_c(\mathbb{R}^d)\Big\}
\end{equation} allowing the case $\infty=\infty$. Consequently, for any $\rho\in \DD_{\rm a.c.}$, it holds that \begin{equation}
	\label{eq: variational rep 2} \frac14\int_0^T\cD(\rho_s)ds=\sup\Big\{\int_0^T\int_{\mathbb{R}^d}\varphi(\rho_s(x))(-\Delta H(s,x)-|\nabla H(s,x)|^2) ds dx: H \in C^\infty_c([0,T]\times \mathbb{R}^d)\Big\}.
\end{equation}
	\end{lemma}  \begin{proof}[Proof of Lemma \ref{lemma: finite entropy dissipation}]
		Fix $z$. Thanks to Lemma \ref{lemma: variational rep}, we may choose $H\in C^\infty_c([0,T]\times \mathbb{R}^d)$ such that \begin{equation}
			\int_0^T \int_{\mathbb{R}^d} \varphi(\rho_s(x))(-\Delta H(s,x)-|\nabla H(s,x)|^2)ds dx>z+3.
		\end{equation} With this choice of $H$ fixed, we set $G:=-\Delta H-|\nabla H|^2$ and may use Theorem \ref{thrm: supex} to find $\epsilon>0$ such that \begin{equation} \label{eq: choose e 1}
			\limsup_{N\to \infty} N^{-d}\log \mathbb{P}\Big(\Big|\int_0^T V(G,\lambda,N,\epsilon)(t, \eta_{N^2t})ds\Big|>1\Big)\le -z
		\end{equation} and so that, making $\epsilon>0$ smaller if necessary, \begin{equation} \label{eq: choose e 2}
			\int_0^T \int_{\mathbb{R}^d} \varphi(\rho^\epsilon_s(x))G(s,x) ds dx >z+2
		\end{equation} where $\rho^\epsilon$ denotes the convolution in space with $\alpha_\epsilon:=(2\epsilon)^{-d}1_{[-\epsilon, \epsilon)^d}$.  We now take $\mathcal{U}$ to be the set \begin{equation}
			\mathcal{U}:=\Big\{v\in \mathbb{D}: \int_0^T \int_{\mathbb{R}^d}\varphi(v^\epsilon_s(x))G(s,x)ds dx>z+2\Big\}
		\end{equation} which contains $\rho$ by the choice of $\epsilon$ in (\ref{eq: choose e 2}), and can readily seen to be open in the topology of $\mathbb{D}$. For $\beta>0$ to be chosen later and $M$ large enough that $\text{supp}(H)\subset \subset [-M,M]^d\times[0,T]$, we use the the definition of $\varphi$ to divide the event $\{\pi^N \in \mathcal{U}\}$ into cases: {\bf{either}} the event in (\ref{eq: choose e 1}) occurs, {\bf or} \begin{equation}
			\frac\beta{N^d}\int_0^T \sum_{|x|_\infty \le MN} w(\eta_{N^2t}(x))dt>1
		\end{equation} {\bf or} \begin{equation}\label{eq: bad entropy event} \frac{1}{N^d}\int_0^T \sum_{x\in \mathbb{Z}^d} \Big\{\lambda(\eta_{N^2t}(x))\Big(-\Delta H\big(t,\frac{x}N\big)-\Big|\nabla H\big(t,\frac{x}{N}\big)\Big|^2\Big)-\beta w(\eta_{N^2t}(x))1_{|x|_\infty\le MN}\Big\} dt >z.
			\end{equation} We now {\bf{claim}} that, for any $\beta>0$
\begin{align} \label{eq: claim that}
\nonumber & \limsup_N \frac{1}{N^{d}} \log \mathbb{E}\exp-\int_0^T \sum_{x\in \mathbb{Z}^d} \lambda(\eta_{N^2t}(x))\big(\Delta H\big(t,\frac{x}N\big)+|\nabla H\big(t,\frac{x}{N}\big)|^2\big)+\beta w(\eta_{N^2t}(x))1_{|x|_\infty\le MN}
\\ &  \le 0,
\end{align}
which implies a bound of order $e^{-zN^d}$ on the final event. The proof concludes by choosing $\beta>0$ small enough that the event in (\ref{eq: bad entropy event}) has probability $e^{-zN^d}$ as in the proof of Theorem \ref{thrm: supex}. \bigskip \\ We now prove the claim. First, we note that there exists $a<\infty$ such that the whole sum is negative whenever $\eta\in X$ is any configuration with $N^{-d}\sum_{x\in B_{MN}}\eta(x)>a$. Whenever the reverse inequality holds, we may use the smoothness of $H$ to replace both the Laplacian and the gradient by the discrete Laplacian and discrete gradient, plus an error which vanishes, uniformly in $\eta$, as $N\to \infty$. Using the Feynman-Kac formula, we may therefore bound the left-hand side, up to an error vanishing in $N$, by \begin{equation}\begin{split} \label{eq: introduce FK}
			\sup_f \sup_{t\le T}& \Big\{\mathbb{E}_{f\nu_\rho}\Big[N^{2-d} \sum_{x\sim y} \Xi(H(t,\cdot), N,\eta, x,y)-N^{2-d}\mathfrak{d}(f)\Big]\Big\}\end{split}\end{equation} where, as before, the outermost supremum runs over probability density functions $f$ with respect to $\nu_\rho$, the sum now runs over all unordered pairs of neighbours, and \begin{equation}\begin{split} \Xi(H,N,\eta,x,y):&=(\lambda(\eta(y)-\lambda(\eta(x))\left(H\left(\frac{y}{N}\right)-H\left(\frac{x}{N}\right)\right)\\&\hspace{1cm}-\frac12(\lambda(\eta(y)+\lambda(\eta(x))\left(H\left(\frac{y}{N}\right)-H\left(\frac{x}{N}\right)\right)^2. \end{split}
			\end{equation} Moving factors involving $H$ outside the expectation, for any $f$, \begin{equation}
				\begin{split}
					\mathbb{E}_{f\nu_\rho}[\Xi(H,N,\eta,x,y)]&=\mathbb{E}_{f\nu_\rho}\left[\lambda(\eta(y)-\lambda(\eta(x)\right]\left(H\left(\frac{y}{N}\right)-H\left(\frac{x}{N}\right)\right)\\& \hspace{1cm} -\frac12\mathbb{E}_{f\nu_\rho}\left[\lambda(\eta(y)+\lambda(\eta(x)\right]\left(H\left(\frac{y}{N}\right)-H\left(\frac{x}{N}\right)\right)^2 \\[1ex] & \le \frac{\left(\mathbb{E}_{f\nu_\rho}\left[\lambda(\eta(y))-\lambda(\eta(x))\right]\right)^2}{2\mathbb{E}_{f\nu_\rho}\left[\lambda(\eta(y))+\lambda(\eta(x))\right]}.
				\end{split}
			\end{equation} Using the change-of-variables formula (\ref{eq: COV}) in both numerator and denominator and using Cauchy-Schwarz in the probability space, the final expression is at most \begin{equation}\begin{split}
				\varphi(\rho)\frac{(\mathbb{E}_{\nu_\rho}[f(\eta+1_x)-f(\eta+1_y)])^2}{2\mathbb{E}_{\nu_\rho}[f(\eta+1_x)+f(\eta+1_y)]} &\le \varphi(\rho)\mathbb{E}_{\nu_\rho}\Big[\Big(\sqrt{f(\eta+1_x)}-\sqrt{f(\eta+1_y)}\Big)^2\Big] \\ &=\mathbb{E}_{\nu_\rho}\Big[\lambda(\eta(x))\Big(\sqrt{f(\eta+1_y-1_x)}-\sqrt{f(\eta)}\Big)^2\Big].
			\end{split}\end{equation}The final expression is the contribution to $\mathfrak{d}(f)$ from the edge $\{x,y\}$, and in particular, the supremum in (\ref{eq: introduce FK}) is nonpositive. The proof of the claim is complete, and hence so is the lemma.\end{proof}

\section{Weak Solutions of the Skeleton Equation}\label{sec_rel_skel_exist} {In order to develop a theory of the skeleton equation to close the large deviation principle in Theorems \ref{thrm: main result} - \ref{thrm: partial LDP}, we must consider initial data that fall outside the global $L^p$-framework, for any $p\in[1,\infty)$. Indeed, in \emph{any} limit obtained from large deviations in equilibrium, the total mass of $\rho$ will be infinite for all times. It is for this reason that we instead} develop a solution theory based on the following notion of relative entropy with respect to a constant fixed density $\gamma\in(0,\infty)${, which appears as the static large deviations rate function and a priori remains finite on fluctuations}.  Precisely, for every $\gamma\in(0,\infty)$, we let $\Psi_{\Phi,\gamma}$ be the unique function satisfying
\[\Psi_{\Phi,\gamma}(\gamma)=0\;\;\textrm{and}\;\;\Psi_{\Phi,\gamma}'(\xi) = \log\Big(\frac{\Phi(\xi)}{\Phi(\gamma)}\Big).\]
It follows from the definition that {$\Psi_{\Phi, \gamma}$ is convex} and $\Psi_{\Phi,\gamma}(\xi)\geq 0$ for every $\xi\in[0,\infty)$ with $\Phi(\xi)=0$ if and only if $\xi=\gamma$ whenever $\Phi$ is strictly increasing.  In the case that $\Phi(\xi)=\xi$ we have that
%
\[\Psi_{\Phi,\gamma}(\xi) = \xi\log\Big(\frac{\Phi(\xi)}{\Phi(\gamma)}\Big) - (\xi-\gamma),\]
which defines the relative entropy with respect to the constant density $\gamma$.  The corresponding relative entropy space is then
\[\Ent_{\Phi,\gamma} =\Big\{\rho\colon\R^d\rightarrow\R\;\;\textrm{nonnegative and measurable with}\;\;\int_{\R^d}\Psi_{\Phi,\gamma}(\rho)<\infty\Big\}.\]
 If we define $\tilde{\Psi}_{\Phi,\gamma}(\xi) = \Psi_{\Phi,\gamma}(\xi)$ if $\xi\geq \gamma$ and $\tilde{\Psi}_{\Phi,\gamma}(\xi) = -\Psi_{\Phi,\gamma}(\xi)$ if $\xi\in[0,\gamma]$, it is straightforward to check that $\Ent_{\Phi,\gamma}$ is a complete, separable metric space with respect to the metric

 \[d(f,g) = \int_{\R^d}\abs{\tilde{\Psi}_{\Phi,\gamma}(f)-\tilde{\Psi}_{\Phi,\gamma}(g)},\]
 where the completeness and separability are respectively consequences of the fact that $\tilde{\Psi}_{\Phi,\gamma}$ is strictly increasing and convex.  We will similarly define the following $L^p$-spaces shifted by $\gamma$,
 \[L^p_\gamma(\R^d) = \Big\{\rho\colon\R^d\rightarrow\R\;\;\textrm{nonnegative and measureable}\;\;\textrm{with}\;\;\int_{\R^d}\abs{\rho-\gamma}^p<\infty\Big\}.\]
Observe that the spaces $L^p_\gamma$ are canonically isomorphic to the space of $L^p$ functions bounded from below by $-\gamma$, and that $L^2_\gamma$ comes equipped with the induced inner product and norm.

We will first prove the existence of suitable weak solutions to the regularised equation
\begin{equation}\label{rel_1}\partial_t\rho = \Delta\Phi(\rho) +\eta \Delta \rho - \nabla\cdot (\sigma(\rho)g)\;\;\textrm{in}\;\;\R^d\times(0,T)\;\;\textrm{with}\;\;\rho(x,0)=\rho_0(x),\end{equation}
for $\sigma$ a smooth and bounded approximation of $\Phi^\frac{1}{2}$.  Note here that we do not expect to see preservation of the $L^1_\gamma$-norm due to the fact that, unless $g=0$, the constant density $\gamma$ is not a solution.  We similarly define $H^1_\gamma(\R^d)$ to be the space of $L^2_\gamma(\R^d)$-functions with $L^2(\R^d)$-integrable weak gradients.

\begin{assumption}\label{ap_assume}  Let $\Phi,\sigma \in \C([0,\infty)\cap \C^1_{\textrm{loc}}((0,\infty))$ satisfy the following assumptions.  
\begin{enumerate}
\item \emph{Vanishing at zero}:  we have that $\Phi(0)=0$ and $\sigma(0)=0$ and that
\[\Phi^\frac{1}{2}(\xi)\log(\Phi(\xi))\;\;\textrm{is locally bounded on $[0,\infty)$.}\]
\item \emph{Strictly increasing}:  we have that $\Phi'(\xi)>0$ for every $\xi\in(0,\infty)$.
\item \emph{Local integrability}:  we have that
\[\log(\Phi(\xi))\;\;\textrm{and}\;\;\Phi'(\xi)\log(\Phi(\xi)),\]
are locally integrable on $[0,\infty)$.
\item \emph{Growth at zero and infinity}:  there exists $m\in[1,\infty)$ and $c\in(0,\infty)$ such that
\[\Phi(\xi)\leq c(\xi+\xi^m)\;\;\textrm{and}\;\;\abs{\sigma(\xi)}\leq c\Phi^\frac{1}{2}(\xi)\;\;\textrm{for every}\;\;\xi\in[0,\infty).\]
\item Assume either that there exists $c\in(0,\infty)$ and $\theta\in[0,\nicefrac{1}{2}]$ such that
\begin{equation}\label{5_00} \frac{\Phi'(\xi)}{\Phi^\frac{1}{2}(\xi)}\leq c\xi^{-\theta}\;\;\textrm{for every}\;\;\xi\in(0,\infty),\end{equation}
or that there exists $c\in(0,\infty)$ and $q\in[1,\infty)$ such that
\begin{equation}\label{5_0}\abs{\xi-\xi'}^q\leq c\abs{\Phi^\frac{1}{2}(\xi)-\Phi^\frac{1}{2}(\xi')}^2\;\;\textrm{for every}\;\;\xi,\xi'\in[0,\infty).\end{equation}
\end{enumerate}
\end{assumption}

 \begin{definition}\label{def_sol_re}  Let $\gamma\in(0,\infty)$, let $\eta\in[0,1)$, let $\rho_0\in \Ent_{\Phi,\gamma}(\R^d)$, and let $g\in(L^2_{t,x})^d$.  A solution of \eqref{rel_1} is a continuous $L^1_{\textrm{loc}}(\R^d)$-valued function $\rho$ that satisfies the following two properties.
\begin{enumerate}
\item{The relative entropy estimate}:  we have that
\[\int_0^T\int_{\R^d}\abs{\nabla\Phi^\frac{1}{2}(\rho)}^2<\infty.\]
\item{The equation}:   for every $\psi\in\C^\infty_c(\R^d)$ and $t\in[0,T]$,

\begin{align*}
\int_{\R^d}\rho(x,t)\psi(x,t) &  = \int_{\R^d}\rho_0(x)\psi(x) -\int_0^t\int_{\R^d}\nabla\Phi(\rho)\cdot\nabla\psi
\\ & \quad -\eta\int_0^t\int_{\R^d}\nabla\rho\cdot\nabla\psi+\int_0^t\int_{\R^d}\sigma(\rho)g\cdot\nabla\psi.
\end{align*}

\end{enumerate}
\end{definition}

\begin{remark} In Definition~\ref{def_sol_re}, it is a consequence of the identity $\nabla\phi = 2\Phi^\frac{1}{2}\nabla\Phi^\frac{1}{2}$ and the interpolation estimate of Proposition~\ref{prop_rel_int} below that $\nabla\Phi^\frac{1}{2}$ is $L^1$-integrable in space.  The local $L^1$-integrability of $\nabla\rho$ is only used for approximating solutions, and follows from the energy estimate of Proposition~\ref{rel_prop0} below for some $\eta\in(0,\infty)$, for a bounded nonlinearity $\sigma$.   \end{remark}

\begin{proposition}\label{rel_prop0}  Let $\eta\in(0,1)$, let $\gamma\in(0,\infty)$, let $\rho_0\in \Ent_{\Phi,\gamma}(\R^d)\cap L^2_{\gamma}(\R^d)$, and let $g\in (L^2_{t,x})^d$.  Under Assumption~\ref{ap_assume} and the additional conditions that $\sigma$ and $\Phi'$ are  bounded and Lipschitz continuous and that $\abs{\sigma(\xi)}\leq c\Phi^\frac{1}{2}(\xi)$ for some $c\in(0,\infty)$, there exists a weak solution $\rho\in L^2_tH^1_{\gamma,x}$ of \eqref{rel_1}.  Furthermore, the solution satisfies the following three estimates.
\begin{enumerate}
\item{The energy estimate}:   for $\Theta_\Phi(\xi) = \int_0^\xi (\Phi'(\xi'))^\frac{1}{2}\dxip$, for some $c\in(0,\infty)$,
\[\sup\nolimits_{t\in[0,T]}\norm{\rho(\cdot,t)}^2_{L^2_{\gamma,x}}+\norm{\nabla\Theta_\Phi(\rho)}^2_{L^2_{t,x}}+\eta\norm{\nabla\rho}^2_{L^2_{t,x}} \leq \norm{\rho_0}^2_{L^2_{\gamma,x}}+c\norm{\sigma}^2_{L^\infty}\norm{g}^2_{L^2_{t,x}}.\]
\item{Regularity in time}:  for some $c\in(0,\infty)$,
\[\norm{\rho}_{H^1_tH^{-1}_{\gamma,x}}\leq \norm{\rho_0}_{L^2_{\gamma,x}}+(\eta+\norm{\Phi'}_{L^\infty})\norm{\nabla\rho}_{L^2_{t,x}}+\norm{\sigma}_{L^\infty}\norm{g}_{L^2_{t,x}}.\]
\item{The relative entropy estimate}:  for some $c\in(0,\infty)$,
\[\sup_{t\in[0,T]}\int_{\R^d}\Psi_{\Phi,\gamma}(\rho(x,t))+\int_0^T\int_{\R^d}\abs{\nabla\Phi^\frac{1}{2}(\rho)}^2\leq c\Big(\int_{\R^d}\Psi_{\Phi,\gamma}(\rho_0)+\int_0^T\int_{\R^d}\abs{g}^2\Big).\]
\end{enumerate}

\end{proposition}

\begin{proof}  The energy estimate is obtained by testing the equation with $\rho$, after introducing a standard regularization by convolution, and the regularity in time is a direct consequence of the equation.   For the final estimate, it follows from a straightforward approximation argument by convolution using the definition of $\Psi_{\Phi,\gamma}$ that, for every $t\in[0,T]$,
\begin{align*}
& \int_{\R^d}\Psi_{\Phi,\gamma}(\rho(x,t))+\eta\int_0^t\int_{\R^d}\frac{\Phi'(\rho)}{\Phi(\rho)}\abs{\nabla\rho}^2+\int_0^t\int_{\R^d}\frac{\Phi'(\rho)^2}{\Phi(\rho)}\abs{\nabla\rho}^2
\\ & = \int_{\R^d}\Psi_{\Phi,\gamma}(\rho_0)+\int_0^t\int_{\R^d}\frac{\sigma(\rho)\Phi'(\rho)}{\Phi(\rho)}g\cdot\nabla\rho,
\end{align*}
and therefore, using H\"older's inequality, Young's inequality, and the assumption that $\sigma(\rho)\leq c\Phi^\frac{1}{2}(\rho)$, we have for some $c\in(0,\infty)$ that
\begin{align*}
& \int_{\R^d}\Psi_{\Phi,\gamma}(\rho(x,t))+\eta\int_0^t\int_{\R^d}\frac{\Phi'(\rho)}{\Phi(\rho)}\abs{\nabla\rho}^2+\int_0^t\int_{\R^d}\frac{\Phi'(\rho)^2}{\Phi(\rho)}\abs{\nabla\rho}^2
\\ & \leq c\Big(\int_{\R^d}\Psi_{\Phi,\gamma}(\rho_0)+\int_0^t\int_{\R^d}\abs{g}^2\Big),
\end{align*}
which completes the proof of the final estimate.  The proof of existence then follows from a Galerkin approximation, the energy estimate above, the boundedness of $\sigma$, the compact embedding of $H^1(B_R)$ into $L^2(B_R)$ and the continuous embedding of $L^1(B_R)$ into $H^{-1}(B_R)$, for every $R\in(0,\infty)$, and the Aubin--Lions--Simon lemma \cite{Aubin,pLions,Simon}.  \end{proof}

\begin{proposition}\label{prop_rel_int}  Let $\Phi$ satisfy Assumption~\ref{ap_assume}, let $d\geq 3$, and let $\frac{1}{2_*}=\frac{1}{2}-\frac{1}{d}$.  Furthermore, let $\gamma\in(0,\infty)$, let $\Phi\in\C^1_{\textrm{loc}}((0,\infty))\cap\C([0,\infty))$ satisfy Assumption~\ref{ap_assume}, and let $\rho$ be a nonnegative measurable function satisfying
\[\sup_{t\in[0,T]}\int_{\R^d}\Psi_{\Phi,\gamma}(\rho)+\int_0^T\int_{\R^d}\abs{\nabla\Phi^\frac{1}{2}(\rho)}^2<\infty.\]
Then, for some $c\in(0,\infty)$,
\[\int_0^T\Big(\int_{\R^d}(\Phi^\frac{1}{2}(\rho)-\Phi^\frac{1}{2}(\gamma))^{2_*}\Big)^\frac{2}{2_*}\leq c\int_0^T\int_{\R^d}\abs{\nabla\Phi^\frac{1}{2}(\rho)}^2.\]
\end{proposition}
\begin{proof}  The proof is a straightforward consequence of the Gagliardo--Nirenberg--Sobolev inequality, H\"older's inequality, and the fact that functions of the form $\psi+\gamma$ for $\psi\in\C^\infty_c(\R^d)$ are dense in the space $\Ent_{\Phi,\gamma}(\R^d)$.  \end{proof}

We will now construct a solution of the skeleton equation
\begin{equation}\label{rel_eq}\partial_t\rho = \Delta\Phi(\rho)-\nabla\cdot(\Phi^\frac{1}{2}(\rho)g),\end{equation}
with initial data in $\Ent_{\Phi,\gamma}(\R^d)$.  We focus on the space $\Ent_{\Phi,\gamma}(\R^d)$ due to its relevance to large deviations below.  The argument can also be applied to initial data in $L^1_{\gamma}(\R^d)$ where, in the latter case, we would obtain local $H^1$-regularity of the solution in compactly supported neighborhoods of $\gamma\in(0,\infty)$.

\begin{proposition}\label{rel_exist}  Let $T\in(0,\infty)$, $\gamma\in(0,\infty)$, and let $\Phi\in\textrm{C}([0,\infty))\cap\textrm{C}^1_{\textrm{loc}}((0,\infty))$ satisfy Assumptions~\ref{ap_assume}.  Then, for every $\rho_0\in \Ent_{\Phi,\gamma}(\R^d)$ and $g\in (L^2_{t,x})^d$ there exists a nonnegative weak solution to \eqref{rel_eq} in the sense of Definition~\ref{def_sol_re} with $\eta=0$ that satisfies the relative entropy estimate of Proposition~\ref{rel_prop0}.\end{proposition}

\begin{proof}  The proof is a small modification of \cite[Proposition~20]{FehGes19} based on the relative entropy estimate of Propostion~\ref{rel_prop0}.  In this case, strong compactness is obtained in $L^1(B_N\times[0,T])$ for every $N\in\N$, and the solution is obtained as the limit along a diagonal subsequence. \end{proof}

 \section{Renormalised Kinetic Solutions of the Skeleton Equation}\label{renormalised_sol}  In this section, we establish the well-posedness of \emph{renormalised kinetic solutions} to the skeleton equation
\begin{equation} \label{intro_1} \partial_t \rho = \nabla\cdot(\Phi'(\rho)\nabla\rho)-\nabla\cdot(\Phi^\frac{1}{2}(\rho)g) = 2\nabla\cdot (\Phi^\frac{1}{2}(\rho)\nabla\Phi^\frac{1}{2}(\rho))-\nabla\cdot(\Phi^\frac{1}{2}(\rho)g),\end{equation}
set on the whole space with initial data with finite relative entropy with respect to some constant density $\gamma\in(0,\infty)$.  A detailed overview of the derivation of the kinetic equation can be found, for instance, in \cite{CP03} and \cite[Section~2]{FehGes19}, and here we only briefly recall the main details.

\subsection{Renormalised kinetic solutions}  In this section, we derive the kinetic formulation of \eqref{intro_1}.  After introducing an additional velocity variable $\xi\in\R$, the kinetic function $\chi$ of a solution $\rho$ is defined, for $\overline{\chi}\colon\R^2\rightarrow\{-1,0,1\}$ satisfying $\overline{\chi}(s,\xi)=\mathbf{1}_{\{0<\xi<s\}}-\mathbf{1}_{\{s<\xi<0\}}$, by
\[\chi(x,\xi,t) = \overline{\chi}(\rho(x,t),\xi) = \mathbf{1}_{\{0<\xi<\rho(x,t)\}},\]
where the final equality follows from the nonnegativity of the solution.  It then follows formally from the distributional identities
\[\nabla_x\chi(x,\xi,t)=\delta_0(\xi-\rho(x,t))\nabla \rho(x,t)\;\;\textrm{and}\;\;\partial_\xi\chi(x,\xi,t)=\delta_0(\xi)-\delta_0(\xi-\rho(x,t)),\]
that the kinetic function satisfies the equation
\[\partial_t \chi =\Phi'(\xi)\Delta_x\chi -\big(\partial_\xi \Phi^\frac{1}{2}(\xi)\big)g(x,t)\nabla_x\chi+ \big(\nabla_xg(x,t)\big)\Phi^\frac{1}{2}(\xi)\partial_\xi \chi+\partial_\xi p,\]
where $\chi(\cdot,0)= \overline{\chi}(\rho_0)$ and where $p$ is a locally finite, nonnegative measure on $\R^d\times\mathbb{R}\times(0,T)$ that satisfies
\begin{equation}\label{parabolic_defect_measure}\delta_0(\xi-\rho) \Phi'(\xi)\abs{\nabla\rho}^2\leq p.\end{equation}
However, due to the low regularity of the control $g$, we write the equation in its conservative form
\[\partial_t \chi =\Phi'(\xi)\Delta_x\chi -\partial_\xi\big(g(x,t)\Phi^\frac{1}{2}(\xi)\nabla_x\chi\big)+\nabla_x\cdot\big(g(x,t)\Phi^\frac{1}{2}(\xi)\partial_\xi \chi\big)+\partial_\xi p,\]
which, with the estimates of Propositions~\ref{rel_prop0} and \ref{prop_rel_int}, form the basis of the solution theory.

\begin{definition}\label{skel_sol_def_rel}  Let $T,\gamma\in(0,\infty)$, let $\Phi\in \C[0,\infty)\cap \C^1_{\textrm{loc}}((0,\infty))$, and let $\rho_0\in\Ent_{\Phi,\gamma}(\R^d)$.  A nonnegative function $\rho\in L^1([0,T],L^1_{\textrm{loc}}(\R^d))$ is a renormalised kinetic solution of \eqref{intro_1} with initial data $\rho_0$ if $\rho$ satisfies the following property.
\begin{enumerate}
\item \emph{The relative entropy estimate}:  we have that
\[\int_0^T\int_{\R^d}\abs{\nabla\Phi^\frac{1}{2}(\rho)}^2<\infty.\]
\end{enumerate}
Furthermore, there exists a nonnegative, locally finite measure $p$ on $\R^d\times (0,\infty) \times[0,T]$ that satisfies the following three properties.
\begin{enumerate}
\item \emph{Regularity}:  in the sense of measures,
\[\delta_0(\xi-\rho)\Phi'(\rho)\abs{\nabla\rho}^2\leq p\;\;\textrm{on}\;\;\R^d\times (0,\infty)\times[0,T].\]
\item \emph{Vanishing at infinity}:  we have that
\[\lim_{M\rightarrow\infty} p(\R^d\times[M,M+1]\times[0,T])=0.\]
\item \emph{The equation}: there exists a subset $\mathcal{N}\subseteq(0,T]$ of Lebesgue measure zero such that, for every $t\in[0,T]\setminus \mathcal{N}$, for every $\psi\in \C^\infty_c(\R^d\times(0,\infty))\cap C(\R^d\times[0,\infty))$ with $\psi(x,0)=0$, the kinetic function $\chi$ of $\rho$ satisfies that
\begin{align*}
& \int_{\R^d}\int_\mathbb{R} \chi(x,\xi,t)\psi(x,\xi)\dx\dxi =  -\int_0^t\int_{\R^d} \Phi'(\rho)\nabla\rho\cdot(\nabla_x\psi)(x,\rho)\dx\dr \nonumber
\\ & \quad + \int_0^t\int_{\R^d} \Phi^\frac{1}{2}(\rho)\nabla\rho\cdot g(x,t)(\partial_\xi\psi)(x,\rho)\dx\dr + \int_0^t\int_{\R^d} \Phi^\frac{1}{2}(\rho)g(x,t)\cdot (\nabla_x\psi)(x,\rho) \dx\dr\nonumber
\\ & \quad + \int_\R\int_{\R^d}\overline{\chi}(\rho_0(x),\xi)\psi(x,\xi)\dx\dxi- \int_0^t\int_\R\int_{\R^d}p\partial_\xi \psi \dx\dxi\dr.
\end{align*}
\end{enumerate}
\end{definition}

\begin{remark}  Observe that Definition~\ref{skel_sol_def_rel} does not impose any continuity on the solution in time.  However, using the local integrability implied by the relative entropy estimate, we prove in Proposition~\ref{rel_L1-continuity} below that the solutions admit continuous representatives in $\mathcal{C}$, and in fact are continuous in the strong Fr\'echet topology of $L^1_{{\rm loc},+}$. \end{remark}

\subsection{Uniqueness of renormalised kinetic solutions}\label{sec_unique}

In this section, we will prove the uniqueness of renormalised kinetic solutions in the sense of Definition~\ref{skel_sol_def_rel} for nonlinearities $\Phi$ that satisfy Assumption~\ref{as_unique} below.  The proof of uniqueness in Theorem~\ref{thm_rel_unique} is significantly complicated by the fact that the various terms in the equation and the parabolic defect measure are neither globally integrable nor decaying at infinity with respect to the velocity variable $\xi\in\R$.   We treat these terms by first introducing a cutoff in the velocity variable, which we remove using \eqref{skel_continuity_3} and the estimates of Proposition~\ref{prop_rel_unique} below.  Notice as well that, in comparison to \cite[Theorem~8]{FehGes19}, we do not obtain an $L^1$-contraction for the equation due to the fact that even the difference of two solutions need not be $L^1$-integrable.  For this reason, it is necessary to use Sobolev embeddings and the estimates of Proposition~\ref{prop_rel_int} to control spatial cutoff errors at infinity.

\begin{assumption}\label{as_unique}  Let $\Phi\in\C([0,\infty))\cap\C^1_{\textrm{loc}}((0,\infty))$ satisfy $\Phi(0)=0$ and $\Phi'(\xi)>0$ for every $\xi\in(0,\infty)$.  Assume that
\[\Phi'\;\;\textrm{is locally $\nicefrac{1}{2}$-H\"older continuous on $(0,\infty)$,}\]
and that there exists $c\in(0,\infty)$ such that, for every $M\in(0,\infty)$,
\begin{equation}\label{skel_continuity_3}\sup_{0\leq \xi\leq M}\abs{\frac{\Phi(\xi)}{\Phi'(\xi)}}\leq cM.\end{equation}
\end{assumption}

\begin{proposition}\label{prop_rel_unique}  Let $T,\gamma\in(0,\infty)$, let $\Phi\in\textrm{C}([0,\infty))\cap\textrm{C}^1_{\textrm{loc}}((0,\infty))$ satisfy Assumption~\ref{as_unique}, let $g\in (L^2_{t,x})^d$, and let $\rho_0\in \Ent_{\Phi,\gamma}(\R^d)$.  Then, if $\rho$ is a solution of \eqref{intro_1} with control $g$, initial data $\rho_0$, and kinetic measure $p$, we have for every $M\in(\gamma,\infty)$ that, for some $c\in(0,\infty)$,
\[\int_{R^d}\Psi_{\Phi,\gamma}(\rho(x,T)) +\int_0^T\int_\R\int_{\R^d}\frac{\Phi'(\xi)}{\Phi(\xi)} p\leq c\Big(\int_{\R^d}\Psi_{\Phi,\gamma}(\rho_0)+\int_0^T\int_{\R^d}\abs{g}^2\Big),\]
which implies, in particular, that
\[\int_0^T\int_\R\int_{\R^d}\frac{1}{\xi} p\leq c\Big(\int_{\R^d}\Psi_{\Phi,\gamma}(\rho_0)+\int_0^T\int_{\R^d}\abs{g}^2\Big).\]
\end{proposition}

\begin{proof}  After introducing smooth cutoff functions in space and velocity, and considering smooth and bounded approximations of $\psi(\xi)=\Psi'_{\Phi,\gamma}(\xi)$,  it is possible to justify testing the equation with $\Psi'_{\Phi,\gamma}$.  For this, we obtain that
\[\int_{R^d}\Psi_{\Phi,\gamma}(\rho(x,T)) +\int_0^T\int_\R\int_{\R^d}\frac{\Phi'(\xi)}{\Phi(\xi)} p = \int_{\R^d}\Psi_{\Phi,\gamma}(\rho_0)+2\int_0^T\int_{\R^d}g\cdot \nabla\Phi^\frac{1}{2}(\rho).\]
Since the regularity of the kinetic measure proves that
\[\frac{\Phi'(\xi)}{\Phi(\xi)}p\geq \delta_0(\xi-\rho)\frac{\Phi'(\xi)^2}{\Phi(\xi)}\abs{\nabla\rho}^2=\delta_0(\xi-\rho) 4\abs{\nabla\Phi^\frac{1}{2}(\rho)}^2, \]
it follows from H\"older's inequality and Young's inequality that, for some $c\in(0,\infty)$,
\[\int_{R^d}\Psi_{\Phi,\gamma}(\rho(x,T)) +\int_0^T\int_\R\int_{\R^d}\frac{\Phi'(\xi)}{\Phi(\xi)} p\leq c\Big(\int_{\R^d}\Psi_{\Phi,\gamma}(\rho_0)+\int_0^T\int_{\R^d}\abs{g}^2\Big).\]
The final claim then follows from the assumption that $\frac{\Phi(\xi)}{\Phi'(\xi)}\leq c\xi$ so that
\[\int_{R^d}\Psi_{\Phi,\gamma}(\rho(x,T)) +\int_0^T\int_\R\int_{\R^d}\frac{1}{\xi} p\leq c\Big(\int_{\R^d}\Psi_{\Phi,\gamma}(\rho_0)+\int_0^T\int_{\R^d}\abs{g}^2\Big),\]
which completes the proof.  \end{proof}

\begin{theorem}\label{thm_rel_unique} Let $T,\gamma\in(0,\infty)$, let $\Phi\in\textrm{C}([0,\infty))\cap\textrm{C}^1_{\textrm{loc}}((0,\infty))$ satisfy Assumption~\ref{as_unique}, let $g\in (L^2_{t,x})^d$, and let $\rho_0\in\Ent_{\Phi,\gamma}(\R^d)$.  Then, a solution of \eqref{intro_1} in the sense of Definition~\ref{skel_sol_def_rel} with control $g$ and with initial data $\rho_0$ is unique.
\end{theorem}

\begin{proof}  For the kinetic function $\chi^i$ of $\rho^i$, we will write $\chi^i_t(x,\xi)=\chi(x,\xi,t)$ and we will make similar conventions for the defect measure $p^i_t$ and all other time-dependent functions or measures appearing in the proof.  For every $\ve\in(0,1)$ let $\kappa^\ve$ be a standard convolution kernel on $\R^d$ of scale $\ve$, for every $\d\in(0,1)$ let $\kappa^\d$ be a standard convolution kernel on $\mathbb{R}$ of scale $\d$, and for every $\varepsilon,\delta\in(0,1)$ let $\kappa^{\ve,\d}(x,y,\xi,\eta)=\kappa^\ve(x-y)\kappa^\d(\xi-\eta)$.  Let $\chi^{i,\ve,\d}_t=(\chi^i_t*\kappa^{\ve,\d})$ and let $\sgn^\delta(\eta) = (\sgn*\kappa^\d)(\eta)=(\sgn*\kappa^{\ve,\d})(y,\eta)$, where in what follows the convolution kernel $\kappa^{\ve,\d}$ will play the role of the test function in Definition~\ref{skel_sol_def_rel}.  It is also necessary to introduce cutoffs in the spatial and velocity variables.  For this, for every $M\in[2,\infty)$, let $\zeta^M:[0,\infty)\rightarrow[0,1]$ satisfy $\zeta^M=0$ on $[0,M^{-1}]\cup[2M,\infty)$ and $\zeta^M = 1$ on $[2M^{-1},M]$ and let $\zeta^M$ linearly interpolate between the values $\{0,1\}$ on $[M^{-1},2M^{-1}]$ and $[M,2M]$.  And for every $R\in[1,\infty)$ let $\varphi_R$ satisfy $\varphi_R=1$ on $\overline{B}_R$, $\varphi_R=0$ on $B_{2R}^c$, and $\abs{\nabla \varphi_R}+R\abs{\nabla^2\varphi_R}\leq \nicefrac{c}{R}$ for some $c\in(0,\infty)$ independent of $R$.  We will now begin with the proof of uniqueness.

In what follows, when necessary we will write $(x,\xi)\in\R^d\times\mathbb{R}$ for the variables inside the convolutions, and we will write $(y,\eta)\in\R^d\times\mathbb{R}$ for the outer integration variables.  The fact that the kinetic function is $\{0,1\}$-valued proves that
\begin{align}\label{su_0}
&   \int_\mathbb{R}\int_{\R^d}\abs{\chi^1_t-\chi^2_t}^2\zeta^M(\eta)\varphi_R(y)\dy\deta \nonumber
\\ & =\lim_{\varepsilon,\delta\rightarrow 0}\int_\mathbb{R}\int_{\R^d}\abs{\chi^{1,\varepsilon,\delta}_t-\chi^{2,\varepsilon,\delta}_t}^2\zeta^M(\eta)\varphi_R(y)\dy\deta\nonumber
\\ & = \lim_{\varepsilon,\delta\rightarrow 0}\int_\mathbb{R}\int_{\R^d}\left(\chi^{1,\ve,\delta}_t\sgn^\delta(\eta)+\chi^{2,\varepsilon,\delta}_t\sgn^\delta(\eta)-2\chi^{1,\ve,\delta}_t\chi^{2,\ve,\delta}_t\right)\zeta^M(\eta)\varphi_R(y)\dy\deta. 
\end{align}
Let $\overline{\kappa}^{\ve,\d}_i$ be defined by
\[\overline{\kappa}^{\ve,\d}_{i,t}(x,y,\eta)=\kappa^{\ve,\d}(x,y,\rho^i(x,t),\eta),\]
and observe from the equation satisfied by $\rho^i$ that, as distributions on $\R^d\times\mathbb{R}\times(0,T)$,
\begin{align*}
\partial_t\chi^{i,\ve,\delta}_t= & -(\Phi'(\rho^i)\nabla\rho^i*\nabla_x\overline{\kappa}_{i,t}^{\ve,\d})-(p^i_t*\partial_\xi\kappa^{\ve,\d})
\\ & +\Big(\Phi^\frac{1}{2}(\rho^i)g\cdot \nabla\rho^i*\partial_\xi \overline{\kappa}_i^{\ve,\d}\Big)+ (\Phi^\frac{1}{2}(\rho^i)g(x,t)*\nabla_x\overline{\kappa}_i^{\ve,\d}).
\end{align*}
It then follows from the symmetry of the convolution kernel that
\begin{align}\label{su_7}
 \partial_t\chi_t^{i,\ve,\delta} & = \nabla_y\cdot (\Phi'(\rho^i)\nabla\rho^i*\overline{\kappa}_{i,t}^{\ve,\d}) + \partial_\eta(p^i_t*\kappa^{\ve,\d}) \nonumber
\\ & \quad -\partial_\eta\big(\Phi^\frac{1}{2}(\rho^i)g\cdot \nabla\rho^i*\partial_\xi \overline{\kappa}_i^{\ve,\d}\big) - \nabla_y\cdot \big(\Phi^\frac{1}{2}(\rho^i)g(x,t)*\overline{\kappa}_i^{\ve,\d}\big).
\end{align}
We define
\begin{equation}\label{su_0007}\mathbb{I}^{\ve,\d,M,R}_t=\int_\mathbb{R}\int_{\R^d}\left(\chi^{1,\ve,\delta}_t\sgn^\delta(\eta)+\chi^{2,\varepsilon,\delta}_t\sgn^\delta(\eta)-2\chi^{1,\ve,\delta}_t\chi^{2,\ve,\delta}_t\right)\zeta^M(\eta)\varphi_R(y)\dy\deta.\end{equation}
We will analyze the terms involving the $\sgn$ function and the mixed term separately.  For every $i\in\{1,2\}$ let
\begin{equation}\label{su_07}\mathbb{I}^{\ve,\d,M}_{t,i,\sgn}=\int_\mathbb{R}\int_{\R^d}\chi^{i,\ve,\delta}_t\sgn^\delta(\eta)\zeta^M(\eta)\varphi_R(y)\;\;\textrm{and}\;\;\mathbb{I}^{\ve,\d,M}_{t,\textrm{mix}}=\int_{\mathbb{R}}\int_{\R^d}\chi^{1,\ve,\delta}_t\chi^{2,\ve,\delta}_t\zeta^M(\eta)\varphi_R(y).\end{equation}
In what follows, we will let $t\in[0,T]$, $M\in(0,\infty)$, $\ve\in(0,1)$, and $\d\in(0,(2M)^{-1})$ be arbitrary.  In particular, this choice of $\d$ guarantees that
\begin{equation}\label{exp_50} \textrm{$\rho^i>(2M)^{-1}$ on the support of $\overline{\kappa}^{\ve,\d}_{i,t}\zeta^M$,}\end{equation}
and that, by the triangle inequality,
\begin{equation}\label{exp_51}\kappa^{\ve,\d}(x,y,\xi,\eta)\kappa^{\ve,\d}(x',y,0,\eta)\zeta_M(\eta) = 0.\end{equation}
In the argument, we will first let $\ve\rightarrow 0$, then $\d\rightarrow 0$, then $R\rightarrow\infty$, and then $M\rightarrow\infty$ which is consistent with this choice.

\textbf{The sign terms.}  We will first analyze the $\sgn $ terms \eqref{su_07}, and we will first consider the case $i=1$.  It follows from \eqref{su_7} that, as distributions on $(0,T)$,
\begin{align*}
& \partial_t \mathbb{I}^{\ve,\d,M}_{t,1,\sgn} = -\int_{\mathbb{R}}\int_{\R^d}(p^1_t*\kappa^{\ve,\d})\partial_\eta\Big(\sgn^\d(\eta)\zeta^{M}(\eta)\Big)\varphi_R
\\ &  + \int_\mathbb{R}\int_{\R^d}\big(\Phi^\frac{1}{2}(\rho^1)g\cdot \nabla\rho^1* \overline{\kappa}_{1,t}^{\ve,\d}\big)\partial_\eta\Big(\sgn^\d(\eta)\zeta^M(\eta)\Big)\varphi_R
\\ & -\int_\R\int_{\R^d}(\Phi'(\rho^1)\nabla\rho^1*\overline{\kappa}_{1,t}^{\ve,\d})\cdot\nabla\varphi_R(\sgn^\d(\eta)\zeta^M(\eta))
\\ & +\int_\R\int_{\R^d}\big(\Phi^\frac{1}{2}(\rho^1)g*\overline{\kappa}_{1,t}^{\ve,\d}\big)\cdot\nabla\varphi_R(\sgn^\d(\eta)\zeta^M(\eta)).
\end{align*}
The symmetry of the convolution kernel, the distributional equality $\partial_\xi\sgn=2\delta_0$ and \eqref{exp_51} prove that
\begin{align}\label{su_11}
\partial_t \mathbb{I}^{\ve,\d,M}_{t,1,\sgn}=&- \int_{\mathbb{R}}\int_{\R^d}(p^1_t*\kappa^{\ve,\d})\sgn^\delta(\eta)\partial_\eta\zeta^M\varphi_R \nonumber
\\ & + \int_{\mathbb{R}}\int_{\R^d}\big(\Phi^\frac{1}{2}(\rho^1)g\cdot \nabla\rho^1* \overline{\kappa}_{1,t}^{\ve,\d}\big)\sgn^\d(\eta)\partial_\eta\zeta^M\varphi_R\nonumber
\\ & -\int_\R\int_{\R^d}(\Phi'(\rho^1)\nabla\rho^1*\overline{\kappa}_{1,t}^{\ve,\d})\cdot\nabla\varphi_R(\sgn^\d(\eta)\zeta^M(\eta)) \nonumber 
\\  & +\int_\R\int_{\R^d}\big(\Phi^\frac{1}{2}(\rho^1)g*\overline{\kappa}_{1,t}^{\ve,\d}\big)\cdot\nabla\varphi_R(\sgn^\d(\eta)\zeta^M(\eta)).
\end{align}
The case $i=2$ is identical and this completes the initial analysis of the $\sgn$ terms.

\textbf{The mixed term.}  As distributions on $(0,T)$,
\begin{align}\label{su_011}
\partial_t\mathbb{I}^{\ve,\d,M}_{t,\textrm{mix}} = & \int_{\mathbb{R}}\int_{\R^d}\partial_t\chi^{1,\ve,\d}_t\chi^{2,\ve,\d}_t\zeta^M(\eta)\varphi_R(y)\dy\deta+\int_{\mathbb{R}}\int_{\R^d}\chi^{1,\ve,\d}_t\partial_t\chi^{2,\ve,\d}_t\zeta^M(\eta)\varphi_R(y)\dy\deta\nonumber
\\ = & \partial_t\mathbb{I}^{\ve,\d,M}_{t,1,\textrm{mix}}+\partial_t\mathbb{I}^{\ve,\d,M}_{t,2,\textrm{mix}}.
\end{align}
Since we have the distributional equality $\partial_{\xi}\chi^2=\delta_0(\xi)-\delta_0(\xi-\rho^2)$, it follows from \eqref{su_7} that
\begin{align}\label{su_18}
& \partial_t\mathbb{I}^{\ve,\d,M}_{t,1,\textrm{mix}}  =   -\int_\mathbb{R}\int_{\R^d}(\Phi'(\rho^1)\nabla\rho^1*\overline{\kappa}^{\ve,\d}_{1,t})\cdot (\nabla\rho^2*\overline{\kappa}^{\ve,\d}_{2,t})\zeta^M\varphi_R \nonumber
\\ &  \quad +\int_{\mathbb{R}}\int_{(\R^d)^2}(p^1_t*\kappa^{\ve,\d}_1)\overline{\kappa}^{\ve,\d}_{2,t}\zeta^M\varphi_R  -\int_\mathbb{R}\int_{(\R^d)^2}\big(\Phi^\frac{1}{2}(\rho^1)g\cdot\nabla\rho^1*\overline{\kappa}_{1,t}^{\ve,\d}\big)\overline{\kappa}^{\ve,\d}_{2,t}\zeta^M\varphi_R \nonumber
\\ & \quad +\int_\mathbb{R}\int_{\R^d}\big(\Phi^\frac{1}{2}(\rho^1)g*\overline{\kappa}_{1,t}^{\ve,\d}\big)\cdot\big(\nabla\rho^2*\overline{\kappa}^{\ve,\d}_{2,t}\big)\zeta^M\varphi_R \nonumber
\\ & \quad +\int_\mathbb{R}\int_{\R^d}\big(\big(\Phi^\frac{1}{2}(\rho^1)g\cdot \nabla\rho^1* \overline{\kappa}_{1,t}^{\ve,\d}\big)-(p^1_t*\kappa^{\ve,\d}_1)\big)\chi^{2,\ve,\d}_t\partial_\eta\zeta^M\varphi_R \nonumber
\\ & \quad  -\int_{\R}\int_{\R^d}(\Phi'(\rho^1)\nabla\rho^1*\overline{\kappa}_{1,t}^{\ve,\d})\cdot\nabla\varphi_R(\chi^{2,\ve,\d}_t\zeta^M) +\int_\R\int_{\R^d}\big(\Phi^\frac{1}{2}(\rho^1)g*\overline{\kappa}_{1,t}^{\ve,\d}\big)\cdot\nabla\varphi_R(\chi^{2,\ve,\d}_t\zeta^M).
\end{align}
We obtain an identical formula for $\mathbb{I}^{\ve,\d,M}_{t,2,\textrm{mix}}$ after swapping the roles of $i\in\{1,2\}$, and this completes the initial analysis of the mixed term.

\textbf{The full derivative.}  We decompose the derivative of \eqref{su_0007} defined by \eqref{su_07},  \eqref{su_11}, \eqref{su_011}, and \eqref{su_18} into the terms
\begin{equation}\label{su_24}\partial_t\mathbb{I}^{\ve,\d,M}_t = \partial_t\mathbb{I}^{\ve,\d,M}_{t,\textrm{par}}+\partial_t\mathbb{I}^{\ve,\d,M}_{t,\textrm{hyp}}+ \partial_t\mathbb{I}^{\ve,\d,M}_{t,\textrm{con}}+\partial_t\mathbb{I}^{\ve,\d,M}_{t,\textrm{cut}},\end{equation}
for the the parabolic term
\begin{align}\label{su_21}
& \partial_t\mathbb{I}^{\ve,\d,M}_{t,\textrm{par}}  =  -2\int_{\mathbb{R}}\int_{(\R^d)^2}\big((p^1_t*\kappa^{\ve,\d})\overline{\kappa}^{\ve,\d}_{2,t}+(p^2_t*\kappa^{\ve,\d})\overline{\kappa}^{\ve,\d}_{1,t}\big)\zeta^M\varphi_R \nonumber
\\  & + 2\int_\mathbb{R}\int_{(\R^d)^3}(\Phi'(\rho^1)+\Phi'(\rho^2))\nabla\rho^1\cdot \nabla\rho^2\overline{\kappa}^{\ve,\d}_{1,t}\overline{\kappa}^{\ve,\d}_{2,t}\zeta^M\varphi_R,
\end{align}
for the hyperbolic term
\begin{align}\label{su_22}
\partial_t\mathbb{I}^{\ve,\d,M}_{t,\textrm{hyp}} & =  2\int_{\mathbb{R}}\int_{(\R^d)^3}g(x,t)\cdot\nabla\rho^1\big(\Phi^\frac{1}{2}(\rho^1)-\Phi^\frac{1}{2}(\rho^2)\big)\overline{\kappa}^{\ve,\d}_{1,t}\overline{\kappa}^{\ve,\d}_{2,t}\zeta^M\varphi_R \nonumber
\\ & +2\int_{\mathbb{R}}\int_{(\R^d)^3}g(x',t)\cdot\nabla\rho^2\big(\Phi^\frac{1}{2}(\rho^2)-\Phi^\frac{1}{2}(\rho^1)\big)\overline{\kappa}^{\ve,\d}_{1,t}\overline{\kappa}^{\ve,\d}_{2,t}\zeta^M\varphi_R,
\end{align}
for the term involving the control
\begin{align}\label{su_022}
\partial_t\mathbb{I}^{\ve,\d,M}_{t,\textrm{con}} & =  2\int_\mathbb{R}\int_{(\R^d)^3}(g(x,t)-g(x',t))\cdot\nabla\rho^1\Phi^\frac{1}{2}(\rho^2)\overline{\kappa}^{\ve,\d}_{1,t}\overline{\kappa}^{\ve,\d}_{2,t}\zeta^M\varphi_R\nonumber
\\ & +2\int_\mathbb{R}\int_{(\R^d)^3}(g(x',t)-g(x,t))\cdot\nabla\rho^2\Phi^\frac{1}{2}(\rho^1)\overline{\kappa}^{\ve,\d}_{1,t}\overline{\kappa}^{\ve,\d}_{2,t}\zeta^M\varphi_R,
\end{align}
and for the term defined by the two cutoffs
\begin{align}\label{su_23}
& \partial_t\mathbb{I}^{\ve,\d,M}_{t,\textrm{cut}} =  \int_\mathbb{R}\int_{\R^d}\big(\Phi^\frac{1}{2}(\rho^1)g\cdot \nabla\rho^1* \overline{\kappa}_{1,t}^{\ve,\d}\big)\big(\sgn^\delta-2\chi^{2,\ve,\d}_t\big)\partial_\eta\zeta^M\varphi_R\nonumber
\\ & + \int_\mathbb{R}\int_{\R^d}\big(\Phi^\frac{1}{2}(\rho^2)g\cdot \nabla\rho^2* \overline{\kappa}_{2,t}^{\ve,\d}\big)\big(\sgn^\delta-2\chi^{1,\ve,\d}_t\big)\partial_\eta\zeta^M\varphi_R\nonumber
\\ & - \int_{\mathbb{R}}\int_{\R^d}\big((p^1_t*\kappa^{\ve,\d})\big(\sgn^\d-2\chi^{2,\ve,\d}_t\big)+(p^2_t*\kappa^{\ve,\d})\big(\sgn^\d-2\chi^{1,\ve,\d}_t\big)\big)\partial_\eta\zeta^M\varphi_R\nonumber
\\ & -\int_\R\int_{\R^d}(\Phi'(\rho^1)\nabla\rho^1*\overline{\kappa}_{1,t}^{\ve,\d})\cdot\nabla\varphi_R(\sgn^\d-2\chi^{2,\ve,\d}_t)\zeta^M \nonumber 
\\  & +\int_\R\int_{\R^d}\big(\Phi^\frac{1}{2}(\rho^1)g*\overline{\kappa}_{1,t}^{\ve,\d}\big)\cdot\nabla\varphi_R(\sgn^\d-2\chi^{2,\ve,\d}_t)\zeta^M\nonumber
\\ & -\int_\R\int_{\R^d}(\Phi'(\rho^2)\nabla\rho^2*\overline{\kappa}_{2,t}^{\ve,\d})\cdot\nabla\varphi_R(\sgn^\d-2\chi^{1,\ve,\d}_t)\zeta^M \nonumber 
\\  & +\int_\R\int_{\R^d}\big(\Phi^\frac{1}{2}(\rho^2)g*\overline{\kappa}_{2,t}^{\ve,\d}\big)\cdot\nabla\varphi_R(\sgn^\d-2\chi^{1,\ve,\d}_t)\zeta^M.
\end{align}
Each of the four terms on the righthand side of \eqref{su_24} will be handled separately.

\textbf{The parabolic terms.}  After adding and subtracting $2(\Phi'(\rho^1)\Phi'(\rho^2))^\frac{1}{2}$ and using that
\[\Phi'(\rho^1)+\Phi'(\rho^2)-2(\Phi'(\rho^1)\Phi'(\rho^2))^\frac{1}{2}=\big((\Phi'(\rho^1))^\frac{1}{2}-(\Phi'(\rho^2))^\frac{1}{2}\big)^2,\]
the parabolic term \eqref{su_21} satisfies
\begin{align}\label{su_25}
& \partial_t\mathbb{I}^{\ve,\d,M}_{t,\textrm{par}} = 4\int_{\mathbb{R}}\int_{(\R^d)^3}(\Phi'(\rho^1)\Phi'(\rho^2))^\frac{1}{2}\nabla\rho^1\cdot \nabla\rho^2\overline{\kappa}^{\ve,\d}_{1,t}\overline{\kappa}^{\ve,\d}_{2,t}\zeta^M\varphi_R\nonumber
\\ &+2\int_{\mathbb{R}}\int_{\R^d}\big((\Phi'(\rho^1))^\frac{1}{2}-(\Phi'(\rho^2))^\frac{1}{2}\big)^2\nabla\rho^1\cdot \nabla\rho^2\overline{\kappa}^{\ve,\d}_{1,t}\overline{\kappa}^{\ve,\d}_{2,t}\zeta^M\varphi_R\nonumber
\\ &  -2\int_{\mathbb{R}}\int_{\R^d}\left(p^1_t\kappa^{\ve,\d}_1\overline{\kappa}^{\ve,\d}_{2,t}+p^2_t\kappa^{\ve,\d}_2\overline{\kappa}^{\ve,\d}_{1,t}\right)\zeta^M\varphi_R.
\end{align}
The definition of the parabolic defect measures, H\"older's inequality, and Young's inequality prove that
\begin{align*}
& 4\int_{\mathbb{R}}\int_{(\R^d)^3}(\Phi'(\rho^1)\Phi'(\rho^2))^\frac{1}{2}\nabla\rho^1\cdot \nabla\rho^2\overline{\kappa}^{\ve,\d}_{1,t}\overline{\kappa}^{\ve,\d}_{2,t}\zeta^M\varphi_R
\\ & \leq 2\int_{\mathbb{R}}\int_{\R^d}\left(p^1_t\kappa^{\ve,\d}_1\overline{\kappa}^{\ve,\d}_{2,t}+p^2_t\kappa^{\ve,\d}_2\overline{\kappa}^{\ve,\d}_{1,t}\right)\zeta^M\varphi_R,
\end{align*}
and it therefore follows from \eqref{su_25} that
\[\partial_t\mathbb{I}^{\ve,\d,M}_{t,\textrm{par}}\leq 2\int_{\mathbb{R}}\int_{\R^d}\big((\Phi'(\rho^1))^\frac{1}{2}-(\Phi'(\rho^2))^\frac{1}{2}\big)^2\nabla\rho^1\cdot \nabla\rho^2\overline{\kappa}^{\ve,\d}_{1,t}\overline{\kappa}^{\ve,\d}_{2,t}\zeta^M\varphi_R.\]
The local $\nicefrac{1}{2}$-H\"older continuity of $\Phi'$, the support of $\kappa^\d$ and of $\zeta^M$, and the local boundedness of $\Phi'$ and $\Phi$ away from zero and infinity on compact subsets of $(0,\infty)$ prove that, for some $c\in(0,\infty)$ depending on $M$,
\[\big((\Phi'(\rho^1))^\frac{1}{2}-(\Phi'(\rho^2))^\frac{1}{2}\big)^2\overline{\kappa}^{\ve,\d}_{1,t}\overline{\kappa}^{\ve,\d}_{2,t}\zeta^M\leq c\d\overline{\kappa}^{\ve,\d}_{1,t}\overline{\kappa}^{\ve,\d}_{2,t}\zeta^M\mathbf{1}_{\{0<\abs{\rho^1-\rho^2}<c\d\}}, \]
and therefore, for some $c\in(0,\infty)$ depending on $M$,
\[\partial_t\mathbb{I}^{\ve,\d,M}_{t,\textrm{par}}\leq c\delta \int_{\mathbb{R}}\int_{(\R^d)^3}\mathbf{1}_{\{0<\abs{\rho^1-\rho^2}<c\d\}}\abs{\nabla\Phi^\frac{1}{2}(\rho^1)}\abs{\nabla\Phi^\frac{1}{2}(\rho^2)}\overline{\kappa}^{\ve,\d}_{1,t}\overline{\kappa}^{\ve,\d}_{2,t}\zeta^M\varphi_R.\]
The boundedness of $\delta \kappa^\delta$ in $\d\in(0,1)$, H\"older's inequality, and Young's inequality prove that, for some $c\in(0,\infty)$ depending on $M$,
\[\limsup_{\ve\rightarrow 0}\left(\partial_t\mathbb{I}^{\ve,\d,M}_{t,\textrm{par}}\right)\leq c\int_{\left\{0<\abs{\rho^1-\rho^2}\leq c\delta\;\;\textrm{and}\;\;(2M)^{-1}<\rho^i<2M\;\;\textrm{for}\;\;i=1,2\right\}}\abs{\nabla\rho^1}^2+\abs{\nabla\rho^2}^2.\]
The local regularity of the $\rho^i$ and the dominated convergence theorem then prove that
\begin{equation}\label{su_27}\limsup_{\d\rightarrow 0}\Big(\limsup_{\ve\rightarrow 0}\Big(\partial_t\mathbb{I}^{\ve,\d,M}_{t,\textrm{par}}\Big)\Big)\leq 0,\end{equation}
which completes the analysis of the parabolic terms.

\textbf{The hyperbolic terms.}  A similar analysis relying on the supports of $\kappa^\d$ and $\zeta^M$, the local Lipschitz continuity of $\Phi$ and therefore $\Phi^\frac{1}{2}$ on $(0,\infty)$, and \eqref{exp_50} prove that the hyperbolic terms \eqref{su_22} satisfy, for some $c\in(0,\infty)$ depending on $M\in(0,\infty)$,
\begin{align*}
& \partial_t\mathbb{I}^{\ve,\d,M}_{t,\textrm{hyp}}
\\ & \leq c\delta\int_{\mathbb{R}}\int_{(\R^d)^3}\mathbf{1}_{\left\{0<\abs{\rho^1-\rho^2}\leq c\delta\;\;\textrm{and}\;\;(2M)^{-1}<\rho^i<2M\;\;\textrm{for}\;\;i=1,2\right\}}\big(\abs{g}\big(\abs{\nabla\rho^1}+\abs{\nabla\rho^2}\big)\big)\overline{\kappa}^{\ve,\d}_{1,t}\overline{\kappa}^{\ve,\d}_{2,t}\zeta^M\varphi_R.
\end{align*}
H\"older's inequality, Young's inequality, and the boundedness of $\d\kappa^\d$ in $\d\in(0,1)$ then prove that
\[\limsup_{\ve\rightarrow 0}\Big(\partial_t\mathbb{I}^{\ve,\d,M}_{t,\textrm{hyp}}\Big)\leq c\int_{\left\{0<\abs{\rho^1-\rho^2}\leq c\delta\;\;\textrm{and}\;\;(2M)^{-1}<\rho^i<2M\;\;\textrm{for}\;\;i=1,2\right\}}\abs{g}^2+\abs{\nabla\rho^1}^2+\abs{\nabla\rho^2}^2,\]
and the $L^2_{t,x}$-integrability of $g$, the local regularity of the $\rho^i$, and the dominated convergence theorem prove that
\begin{equation}\label{su_28}\limsup_{\delta\rightarrow 0}\Big(\limsup_{\ve\rightarrow 0}\Big(\partial_t\mathbb{I}^{\ve,\d,M}_{t,\textrm{hyp}}\Big)\Big)\leq 0,\end{equation}
which completes the analysis of the hyperbolic terms.

\textbf{The control terms.}  The support of $\kappa^\d$ and of $\zeta^M$ and the local boundedness of $\Phi$ on $(0,\infty)$ prove that the control terms \eqref{su_022} satisfy, for some $c\in(0,\infty)$ depending on $M$,
\[\partial_t\mathbb{I}^{\ve,\d,M}_{t,\textrm{con}} \leq c\int_\mathbb{R}\int_{(\R^d)^3}\abs{g(x',t)-g(x,t)}\left(\abs{\nabla\rho^1}+\abs{\nabla\rho^2}\right)\overline{\kappa}^{\ve,\d}_{1,t}\overline{\kappa}^{\ve,\d}_{2,t}\zeta^M\varphi_R. \]
It then follows similarly to the above from the $L^2_{t,x}$-integrability of $g$, the local regularity of the $\rho^i$, and the dominated convergence theorem that
\[\limsup_{\ve\rightarrow 0}\Big( \partial_t\mathbb{I}^{\ve,\d,M}_{t,\textrm{con}} \Big)\leq 0,\]
and, therefore,
\begin{equation}\label{su_29}\limsup_{\d\rightarrow 0}\Big(\limsup_{\ve\rightarrow 0}\Big( \partial_t\mathbb{I}^{\ve,\d,M}_{t,\textrm{con}}\Big) \Big)\leq 0,\end{equation}
which completes the analysis of the control terms.

\textbf{The cutoff terms.}  It follows from the definition of the parabolic defect measures, the computation leading to the treatment of the martingale term in \cite[Theorem~4.7]{FG21} which relies on the inequality $\abs{2\chi^{i,\ve,\d}-\sgn^\delta}\leq 1$ and the equality, for example,
\begin{align*}
& \lim_{\ve,\d\rightarrow 0} \int_\mathbb{R}\int_{\R^d}\big(\Phi^\frac{1}{2}(\rho^1)g\cdot \nabla\rho^1* \overline{\kappa}_{1,t}^{\ve,\d}\big)\big(\sgn^\delta-2\chi^{2,\ve,\d}_t\big)\partial_\eta\zeta^M\varphi_R
\\ & =\int_{\R}\int_{\R^d}\Phi^\frac{1}{2}(\rho^1)g\cdot\nabla\rho^1\sgn(\rho^1-\rho^2)\partial_\eta\zeta^M(\rho^1)\varphi_R,
\end{align*}
that, after passing to the limit $\ve,\d\rightarrow 0$ in \eqref{su_23}, for $\Phi_M$ the unique function satisfying $\Phi_M(0)=0$ and $\Phi_M'(\xi)=\zeta_M(\xi)\Phi'(\xi)$,
\begin{align}\label{rel_101}
& \limsup_{\d\rightarrow 0}\Big(\limsup_{\ve \rightarrow 0}\big(\partial_t\mathbb{I}^{\ve,\d,M}_{t,\textrm{cut}}\big)\Big)
\\ \nonumber & \leq  \int_{\R^d}\abs{g}\abs{\partial_\eta\zeta^M(\rho^1)\Phi^\frac{1}{2}(\rho^2)\nabla\rho^1-\partial_\eta\zeta^M(\rho^2)\Phi^\frac{1}{2}(\rho^2)\nabla\rho^2}\varphi_R + \int_{\mathbb{R}}\int_{\R^d} \abs{\partial_\eta\zeta^M}\varphi_R(p^1_t+p^2_t)
\\ \nonumber & \quad + \int_{\R^d}\abs{\Phi_M(\rho^1)-\Phi_M(\rho^2)}\Delta\varphi_R
\\ \nonumber & \quad +\int_{\R^d}\abs{\zeta_M(\rho^1)\Phi^\frac{1}{2}(\rho^1)-\zeta_M(\rho^2)\Phi^\frac{1}{2}(\rho^2)}\abs{g}\abs{\nabla\varphi_R},
\end{align}
where we observe that in this case we preserve the exact equality in the term involving $\Delta\varphi_R$.  We do this to treat the case $d=2$ below.  In the case $d\geq 3$, as in the proof of Theorem~\ref{thm_rel_unique} we first pass to the limit $R\rightarrow\infty$, for which we have using the definition of $\Phi_R$, Proposition~\ref{prop_rel_int}, H\"older's inequality, and the dominated convergence theorem that
\begin{align*}
& \limsup_{R\rightarrow\infty}\int_{\R^d}\abs{\Phi_M(\rho^1)-\Phi_M(\rho^2)}\abs{\Delta\varphi_R}
\\ & \leq \limsup_{R\rightarrow\infty} \Big(\int_{B_{2R}\setminus B_R} \abs{\Phi_M(\rho^1)-\Phi_M(\rho^2)}^\frac{2_*}{2}\Big)^\frac{2}{2_*}\Big(\int_{\R^d}\abs{\Delta\varphi_R}^\frac{d}{2}\Big)^\frac{2}{d}
\\ & \leq c\limsup_{R\rightarrow\infty} \Big(\int_{B_{2R}\setminus B_R} \abs{\Phi_M(\rho^1)-\Phi_M(\rho^2)}^\frac{2_*}{2}\Big)^\frac{2}{2_*}=0.
\end{align*}
In the case $d=1$ we use the boundedness of $\Phi_M$ to argue that
\[\limsup_{R\rightarrow\infty}\int_{\R^d}\abs{\Phi_M(\rho^1)-\Phi_M(\rho^2)}\abs{\Delta\varphi_R}\leq c\limsup_{R\rightarrow\infty}\int_{\R^d}\abs{\Delta\varphi_R}=\limsup_{R\rightarrow\infty} \frac{c}{R}=0,\]
and in the case $d=2$ we use the local regularity of the $\rho^i$ implied by the nondegeneracy of $\Phi$ and the relative entropy estimate to write
\[\nabla\Phi_M(\rho^i) = \Phi'_M(\xi)\nabla\rho^i = \frac{\Phi'_M(\xi)\Phi^\frac{1}{2}(\xi)}{\Phi'(\xi)}\nabla\Phi^\frac{1}{2}(\rho^i)=\zeta_M(\xi)\Phi^\frac{1}{2}(\xi)\nabla\Phi^\frac{1}{2}(\rho^i),\]
for which we have after integrating by parts and applying H\"older's inequality that, for some $c\in(0,\infty)$ depending on $M$ but independent of $R$,
\begin{align*}
& \int_{\R^d}\abs{\Phi_M(\rho^1)-\Phi_M(\rho^2)}\Delta\varphi_R  = \int_{\R^d}\sgn(\Phi_M(\rho^1)-\Phi_M(\rho^2))\big(\nabla\Phi_M(\rho^1)-\nabla\Phi_M(\rho^2)\big)\cdot \nabla\varphi_R
\\ & \leq c\int_{\R^d}\big(\abs{\Phi^\frac{1}{2}(\rho^1)}+\abs{\nabla\Phi^\frac{1}{2}(\rho^2)}\big)\abs{\nabla\varphi_R} \leq  c\Big(\int_{B_{2R}\setminus B_R}\big(\abs{\nabla\Phi^\frac{1}{2}(\rho^1)}^2+\abs{\nabla\Phi^\frac{1}{2}(\rho^2)}^2\big)\Big)^\frac{1}{2}\Big(\int_{\R^d}\abs{\nabla\varphi_R}^2\Big)
\\ & \leq c\Big(\int_{B_{2R}\setminus B_R}\big(\abs{\nabla\Phi^\frac{1}{2}(\rho^1)}^2+\abs{\nabla\Phi^\frac{1}{2}(\rho^2)}^2\big)\Big)^\frac{1}{2}.
\end{align*}
It then follows from the finiteness of the relative entropy in Definition~\ref{skel_sol_def_rel} and the dominated convergence theorem that this term vanishes in the limit $R\rightarrow\infty$.  For the final term of \eqref{rel_101}, we observe similarly using the $(L^2_{x,t})^d$-integrability of $g$ that, if $d\geq 3$,
\begin{align*}
& \limsup_{R\rightarrow\infty} \int_{\R^d}\abs{\zeta_M(\rho^1)\Phi^\frac{1}{2}(\rho^1)-\zeta_M(\rho^2)\Phi^\frac{1}{2}(\rho^2)}\abs{g}\abs{\nabla\varphi_R}
\\ & \leq \limsup_{R\rightarrow\infty}\Big(\int_{B_{2R}\setminus B_R}\abs{\zeta_M(\rho^1)\Phi^\frac{1}{2}(\rho^1)-\zeta_M(\rho^2)\Phi^\frac{1}{2}(\rho^2)}^{2_*}\Big)^\frac{1}{2_*}\Big(\int_{B_{2R}\setminus B_R}\abs{g}^2\Big)^\frac{1}{2}\Big(\int_{\R^d}\abs{\nabla\phi_R}^d\Big)^\frac{1}{d}
\\ & \leq c\limsup_{R\rightarrow\infty}\Big(\int_{B_{2R}\setminus B_R}\abs{\zeta_M(\rho^1)\Phi^\frac{1}{2}(\rho^1)-\zeta_M(\rho^2)\Phi^\frac{1}{2}(\rho^2)}^{2_*}\Big)^\frac{1}{2_*}\Big(\int_{B_{2R}\setminus B_R}\abs{g}^2\Big)^\frac{1}{2}=0,
\end{align*}
where the final inequality follows from a straightforward modification of Proposition~\ref{prop_rel_int}.  The cases $d=1$ and $d=2$ are handled similarly to the above, where in both cases it suffices to bound $\zeta_M(\rho^i)\Phi^\frac{1}{2}(\rho^i)$ in $L^\infty$ and then apply H\"older's inequality and the dominated convergence theorem. Therefore, by the monotone convergence theorem,
\begin{align*}
& \limsup_{R\rightarrow\infty}\Big(\limsup_{\d\rightarrow 0}\Big(\limsup_{\ve \rightarrow 0}\big(\partial_t\mathbb{I}^{\ve,\d,M}_{t,\textrm{cut}}\big)\Big)\Big)
\\ & \leq  \int_{\R^d}\abs{g}\abs{\partial_\eta\zeta^M(\rho^1)\Phi^\frac{1}{2}(\rho^2)\nabla\rho^1-\partial_\eta\zeta^M(\rho^2)\Phi^\frac{1}{2}(\rho^2)\nabla\rho^2} + \int_{\mathbb{R}}\int_{\R^d} \abs{\partial_\eta\zeta^M}(p^1_t+p^2_t).
\end{align*}
We now proceed using Proposition~\ref{prop_rel_unique} and the $(L^2_{x,t})^d$-integrability of $g$, which prove for the set $A_{M,i}=\{M^{-1}\leq \xi\leq 2M^{-1}\}\cup\{M\leq\xi\leq 2M\}$ that, for some $c\in(0,\infty)$,
\begin{align}\label{rel_102}
& \liminf_{M\rightarrow\infty}\int_{\R^d}\abs{g}\abs{\partial_\eta\zeta^M(\rho^1)\Phi^\frac{1}{2}(\rho^2)\nabla\rho^1-\partial_\eta\zeta^M(\rho^2)\Phi^\frac{1}{2}(\rho^2)\nabla\rho^2} + \int_{\mathbb{R}}\int_{\R^d} \abs{\partial_\eta\zeta^M}(p^1_t+p^2_t)
\\ \nonumber & \leq \liminf_{M\rightarrow\infty} c\sum_{i=1}^2\big(\norm{g\mathbf{1}_{A_{M,i}(\rho^i)}}^2_{L^2_{t,x}}+\int_{\mathbb{R}}\int_{\R^d} \frac{1}{\xi}\big(\mathbf{1}_{\{M^{-1}\leq \xi\leq 2M^{-1}\}}+\mathbf{1}_{\{M\leq\xi\leq 2M\}}\big)p^i_t\big)=0,
\end{align}
which completes the analysis of the cutoff term.

\textbf{The conclusion.}  Returning to \eqref{su_0}, we have from \eqref{su_27}, \eqref{su_28}, \eqref{su_29}, \eqref{rel_102}, from $\rho_0^1=\rho_0^2$, and from the monotone convergence theorem that
\[\max_{t\in[0,T]}\norm{\rho^1(\cdot,t)-\rho^2(\cdot,t)}_{L^1(\R^d)}=0,\]
which completes the proof.\end{proof}

We conclude this section with the proof that a renormalised kinetic solution with initial data in $\Ent_{\Phi,\gamma}(\R^d)$ admits a continuous representative mapping into $L^1_{\textrm{loc}}(\R^d)$ and into $\mathcal{M}(\R^d)$.

\begin{proposition}\label{rel_L1-continuity}  Let $T,\gamma\in(0,\infty)$, let $\Phi\in\textrm{C}([0,\infty))\cap\textrm{C}^1_{\textrm{loc}}((0,\infty))$ satisfy Assumption~\ref{as_unique}, let $g\in (L^2_{t,x})^d$, and let $\rho_0\in \Ent_{\Phi,\gamma}(\R^d)$.  Then, the renormalised kinetic solution $\rho$ of \eqref{intro_1} in the sense of Definition~\ref{skel_sol_def_rel} with initial data $\rho_0$ and control $g$ has a representative in $\C([0,T];L^1_{\textrm{loc}}(\R^d))\cap \mathcal{C}$.  \end{proposition}

\begin{proof}  The proof of the existence of a representative in $\C([0,T];L^1_{\textrm{loc}}(\R^d))$ is virtually identical to \cite[Theorem~5.29]{FG21}, and the existence of a continuous representative in  $\mathcal{C}$ is then a consequence of the strong Fr\'echet continuity of $L^1_{\textrm{loc}}(\R^d)$ due to the fact that the topology of $\mathcal{M}$ is generated by integration against compactly supported continuous functions.\end{proof}

\section{Equivalence of Kinetic and Weak Solutions}\label{section_equivalence}

In this section, we will prove the equivalence of weak solutions in the sense of Definition~\ref{def_sol_re} and renormalised kinetic solutions in the sense of Definition~\ref{skel_sol_def_rel} to the skeleton equation \eqref{intro_1} for initial data in the space $\Ent_{\Phi,\gamma}(\R^d)$ under the following assumption on $\Phi$.  We then use this fact to establish the existence of renormalised kinetic solutions for initial data in $\Ent_{\Phi,\gamma}(\R^d)$.

\begin{assumption}\label{as_equiv} Let $\Phi\in\textrm{C}([0,\infty))\cap\textrm{C}^1_{\textrm{loc}}((0,\infty))$ satisfy Assumption~\ref{as_unique}.  Assume that $\Phi$ satisfies one of the following two conditions.
\begin{enumerate}
\item There exist $0<a\le A<\infty$ such that, for all $\xi\in (0,\infty)$, \begin{equation}\label{eq: DH hyp on Phi} a\le \Phi'(\xi)\le A. \end{equation}
\item We have that $\Phi^\frac{1}{2}\colon[0,\infty)\rightarrow[0,\infty)$ is convex, there exists $c\in(0,\infty)$ such that
\begin{equation}\label{skel_continuity_6} \sup_{\{\xi\geq 1\}}\abs{\frac{\Phi(\xi+1)}{\Phi(\xi)}}\leq c,\end{equation}
and for every $M\in(0,1)$ there exists $c\in(0,\infty)$ depending on $M$ such that
\begin{equation}\label{skel_continuity_7} \sup_{\{\xi\geq M\}}\abs{\frac{\Phi^\frac{1}{2}(\xi)}{\Phi'(\xi)}}\leq c.\end{equation}
\end{enumerate}
\end{assumption}

\begin{remark} Condition (1) could be replaced in our arguments by a modification of condition (1)
 in \cite[Assumption~10]{FehGes19}. In this case, we would require that, for some $p\in[2,\infty)$, \begin{equation} \label{eq: old condition}
	\Big(\frac{\Phi^{1/2}(\xi)}{\Phi'(\xi)}\Big)^p \le c(1+\xi),
\end{equation} and that $\xi\mapsto \Phi^{1/2}(\xi)$ satisfies the defective concavity property in Lemma \ref{lemma: defective concavity} below. \end{remark}

\begin{theorem}\label{thm_equiv}  Let $T,\gamma\in(0,\infty)$, let $\Phi\in\textrm{C}([0,\infty))\cap\textrm{C}^1_{\textrm{loc}}((0,\infty))$ satisfy Assumptions~\ref{as_unique} and \ref{as_equiv}, let $g\in (L^2_{t,x})^d$, and let $\rho_0\in \Ent_{\Phi,\gamma}(\R^d)$.  Then, $\rho$ is a renormalised kinetic solution in the sense of Definition~\ref{skel_sol_def_rel} if and only if $\rho$ is a weak solution in the sense of Definition~\ref{def_sol_re}.
\end{theorem}
 
The proof of the equivalence is based on the following observations.  It is readily seen that a renormalised kinetic solution is a weak solution.  For the converse statement, fix $g\in (L^2_{t,x})^d$ and let $\rho$ be any weak solution in the sense of Definition~\ref{def_sol_re}. We fix a smooth, compactly supported function $\kappa$ with $\int \kappa(x)dx=1$, and for $\epsilon>0$ let $\kappa^\epsilon:=\epsilon^{-d}\kappa(\cdot/\epsilon)$. In order to show that $\rho$ is also a kinetic solution in the sense of Definition~\ref{skel_sol_def_rel}, we fix a compactly supported test function $\psi \in C^\infty_{\rm c}(\mathbb{R}^d\times(0,\infty))$, let $\Psi(x,\xi):=\int_0^\xi \psi(x,\zeta)d\zeta$, and fix $M$ such that $\text{supp}(\psi)\subset \mathbb{R}^d\times[0,M]$.  Writing $\rho^\epsilon$ for the convolution of $\rho$ with $\kappa^\epsilon$, we will study the evolution of $\int_{\mathbb{R}^d} \Psi(x,\rho^\epsilon(x,t))$ and take the limit $\epsilon\to 0$.  We have that the composition $\Psi(x,\rho^\ve)$ is a distributional solution of the equation
\begin{equation}\label{exp_1}\partial_t\Psi(x,\rho^\ve) = -\psi(x,\rho^\ve)(2\Phi^\frac{1}{2}(\rho)\nabla\Phi^\frac{1}{2}(\rho)*\nabla\kappa^\ve)+\psi(x,\rho^\ve)(\Phi^\frac{1}{2}(\rho)g*\nabla \kappa^\ve),\end{equation}
and we treat both terms on the righthand side of \eqref{exp_1} identically using only the $L^2_{t,x}$-integrability of $g$ and $\nabla\Phi^\frac{1}{2}(\rho)$ respectively.  We therefore consider only the second term, for which we have
\begin{align}\label{exp_2}
&  \nonumber \int_0^t\int_{\R^d}\psi(x,\rho^\ve)(x,\rho^\ve)(\Phi^\frac{1}{2}(\rho)g*\nabla\kappa^\ve)
\\  & = \int_0^t\int_{\R^d}(\nabla_x\psi)(x,\rho^\ve)\cdot (\Phi^\frac{1}{2}(\rho)g*\kappa^\ve)+ \int_0^t\int_{\R^d}(\partial_\xi\psi)(x,\rho^\ve)\nabla\rho^\ve\cdot(\Phi^\frac{1}{2}(\rho)g*\kappa^\ve).
\end{align}
The $L^1_{t,x}$-integrability of the product $\Phi^\frac{1}{2}(\rho)g$ and boundedness of $\nabla_x\psi$ make it relatively straightforward to pass to the $\ve\rightarrow 0$ limit in the first term on the righthand side of \eqref{exp_2}.  The second term is more difficult even if $\rho$ is $H^1_x$-valued, since the unboundedness of $\Phi^\frac{1}{2}(\rho)$ means that $(\Phi^\frac{1}{2}(\rho)g*\kappa^\ve)$ does not converge $L^2_{t,x}$-strongly to the limit $\Phi^\frac{1}{2}(\rho)g$.  In fact, since our analysis holds in an arbitrary dimension, the limit effectively holds in $L^1_{t,x}$-strongly but not better.

In \cite{FehGes19} this difficulty was handled using the following observation.  The compact support of the test function $\psi$ yields an $L^\infty$-bound for the convolution $\rho^\ve$, and the analysis worked to transfer the $L^\infty$-bound for the convolution to an $L^\infty$-bound for $\rho$ itself, and by doing so it established the $L^2$-strong convergence of the convolutions away from the vanishing set on which $\rho$ is large.  The argument is complicated by the fact that $\rho$ is itself irregular, and it is therefore necessary to exploit the regularity implied by the relative entropy estimate.

In the case of hypothesis (2) in Assumption \ref{as_equiv}, the argument is essentially the same as the corresponding proof in \cite{FehGes19}, using the compact support of test functions. However, it does not appear that the nonlinearity $\varphi$ produced by the zero-range process is covered by either of the hypotheses corresponding to \cite[Assumption~10]{FehGes19}. Precisely, thanks to the lower bound, the assumption (\ref{eq: old condition}) holds with $p=2$, but it does not appear to be possible, in general, to verify the concavity of $\xi\mapsto \Phi^{1/2}(\xi)$, which is required for the argument in \cite{FehGes19}. Instead, we first give a preliminary calculation, showing how (\ref{eq: DH hyp on Phi}) implies a `defective concavity' property. \begin{lemma}\label{lemma: defective concavity} Fix $\kappa^\epsilon$ as above. Under \eqref{eq: DH hyp on Phi}, there exists an absolute constant $\vartheta$ such that, for all functions $u\in L^1_{\mathrm{loc}}(\mathbb{R}^d)$, it holds that   any $x \in \mathbb{R}^d$, \begin{equation} \label{eq: defective concavity}
	\left(\Phi^{1/2}(u)\right)^\epsilon(x) \le \vartheta \Phi^{1/2}(u^\epsilon(x))
\end{equation} where $u^\epsilon, (\Phi^{1/2}(u))^\epsilon$ denote the convolutions of $u$, $\Phi^{1/2}(u)$ with the mollifier $\kappa^\epsilon$. 
\end{lemma}\begin{proof}
	Thanks to (\ref{eq: DH hyp on Phi}), for all $\xi\ge 0$,\begin{equation}
		\label{eq: upper lower bd} a\xi \le \Phi(\xi)\le A\xi.
	\end{equation} We now have, for any $x$, $$ (\Phi^{1/2}(u))^\epsilon(x) \le A^{1/2} (u^{1/2})^\epsilon(x)
 \le A^{1/2} (u^\epsilon)^{1/2}(x) \le A^{1/2}a^{-1/2} \Phi(u^\epsilon(x))^{1/2} $$ where the second inequality follows from Jensen's inequality, and the first and third inequalities follow from the upper and lower bounds of (\ref{eq: upper lower bd}) respectively. This proves the assertion with $\vartheta:=\sqrt{A/a}$.
\end{proof}	  We now return to the argument preceding \eqref{exp_1} to show how the argument of \cite[Theorem 4.3]{FehGes19} may be modified to allow only defective concavity. In preparation for taking the limit, we form the following decomposition of $\mathbb{R}^d\times[0,T]$. For $k\ge 1$, set \begin{equation}
	\label{eq: Ak def} A_k:=\left\{(y,t)\in \mathbb{R}^d\times [0,T]: \Phi^{1/2}(\rho(y,t))\ge \Phi^{1/2}(M)+k\right\}.
\end{equation}We write $A_k^\mathrm{c}$ for the complement $A_k^\mathrm{c}:=(\mathbb{R}^d\times[0,T])\setminus A_k$ and $1_{k+1, k}$ for the indicator function of $A_k\setminus A_{k+1}$. Given any $k$, we may integrate \eqref{exp_1} and use the decomposition \begin{equation} \begin{split} \label{eq: big decomp} &\left.\int_{\mathbb{R}^d}\Psi(x,\rho^\epsilon(x,r))dx\right|_{r=0}^t\\ & =\int_0^t \int_{(\mathbb{R}^d)^2}2\Phi^{1/2}(\rho(y,s))\nabla \Phi^{1/2}(\rho(y,s))\cdot \nabla_x \kappa^\epsilon(y-x)\psi(x, \rho^\epsilon(x,s))1_{A_k}(y,t)dx dy ds \\ &-\int_0^t \int_{(\mathbb{R}^d)^2}g(y,t)\Phi^{1/2}(\rho(y,t))\cdot \nabla_x \kappa^\epsilon(y-x)\psi(x, \rho^\epsilon(x,s))1_{A_k}(y,t)dx dy ds\\& +\int_0^t \int_{(\mathbb{R}^d)^2}2\Phi^{1/2}(\rho(y,s))\nabla \Phi^{1/2}(\rho(y,s))\cdot \nabla_x \kappa^\epsilon(y-x)\psi(x,\rho^\epsilon(x,s))1_{A_k^\mathrm{c}}(y,s) dx dy ds \\ & -\int_0^t \int_{(\mathbb{R}^d)^2}g(y,t)\Phi^{1/2}(\rho(y,t))\cdot \nabla_x \kappa^\epsilon(y-x)\psi(x, \rho^\epsilon(x,s))1_{A_k^\mathrm{c}}(y,t)dx dy ds.\end{split}\end{equation} We wish to take the limit $\epsilon\to 0$. An identical argument to \cite[Lemma 4.5]{FehGes19} shows that, for any choice of $k\ge 1$, the final two terms converge to \begin{equation}  \label{eq: terms same as FG}\begin{split} \int_0^t \int_{\mathbb{R}^d} (2\nabla \Phi^{1/2}(\rho)-g)\Phi^{1/2}(\rho)\Big(\partial_x \psi(x,\rho)+\frac{2\Phi(\rho)}{\Phi'(\rho)}\nabla \Phi^{1/2}(\rho)\partial_\xi \psi(x,\rho)\Big)dx ds.\end{split} \end{equation} In order to complete the same argument as \cite{FehGes19}, one must show that the first two terms on the right-hand side of (\ref{eq: big decomp}) vanish as $\epsilon\to 0$, which will require a careful choice of $k$. The following replaces \cite[Lemma 4.4]{FehGes19}.   Once this is proven, we apply \eqref{eq: terms same as FG} with this choice of $k$ to take the limits of every term on the right-hand side of \eqref{eq: big decomp} and complete the argument as in \cite{FehGes19}.  \begin{lemma}\label{lemma: 4.4'} Suppose that $\Phi\in  C([0,\infty))\cap C^1((0,\infty))$ satisfies \eqref{eq: DH hyp on Phi}. Then, for every $F\in L^2(\mathbb{R}^d\times[0,T]; \mathbb{R}^d)$ and $\psi\in C^\infty_c(\mathbb{R}^d\times \mathbb{R})$, for sufficiently large $k=K(\vartheta, \psi)$, \begin{equation}\label{eq: new commutator conclusion} \limsup_{\epsilon\to 0}\sup_{t\le T}\Big|\int_0^t \int_{(\mathbb{R}^d)^2} F(y,s)\Phi^{1/2}(\rho(y,s))\cdot \nabla_x \kappa^\epsilon(y-x)\psi(x, \rho^\epsilon(x,t))1_{{A_k}}(y,s) dy dx ds \Big| =0 .\end{equation} \end{lemma}
	\begin{proof}
		Following the cited lemma and using the compactness of the support $\text{supp}(\psi)$, it is sufficient to show that we may choose $k=K(\vartheta, \psi)$ such that \begin{equation} \label{eq: aim to show comm est}
			\limsup_{\epsilon\to 0} \epsilon^{-2}\int_0^T \int_{(\mathbb{R}^d)^2}\Phi(\rho(y,t))|\epsilon \nabla \kappa^\epsilon(y-x)||\psi(x, \rho^\epsilon(x,t))|1_{A_k}(y,t)dxdydt<\infty.
		\end{equation} We now fix $k\ge 1$ to be chosen later. Repeating the arguments leading to \cite[Equation (4.14)]{FehGes19}, \begin{equation}
			\begin{split} &\int_0^T \int_{(\mathbb{R}^d)^2} \Phi(\rho(y,t))|\epsilon\nabla_x \kappa^\epsilon(y-x)||\psi(x, \rho^\epsilon(x,t))|1_{A_k}(y,t)dydxdt \\& \hspace{1cm} \le \sum_{n\ge k} \int_0^T \int_{(\mathbb{R}^d)^2}(\Phi^{1/2}(M)+n+1)^2|\epsilon \nabla_x \kappa^\epsilon(x)||\psi(y+x, \rho^\epsilon(y+x,t))|1_{k,k+1}(y,t) dy dx dt. \end{split}
		\end{equation} The subsequent stages of the argument in \cite{FehGes19} use (standard) concavity, and we must now modify the argument. Letting $\vartheta\ge 1$ be the constant in Lemma \ref{lemma: defective concavity}, we modify the definition of the set $B^\epsilon_{n,x}$ from \cite[Equation (4.15)]{FehGes19}:\begin{equation}
B^\epsilon_{n,x}:=\Big\{(y,t)\in \mathbb{R}^d\times [0,T]: \big|\Phi^{1/2}(\rho(y,t))-\vartheta^{-1}\big(\Phi^{1/2}(\rho)\big)^\epsilon(y+x,t)\big|1_{n,n+1}(y,t)\ge n\Big\}.
\end{equation} First, we check that this preserves the relationship between $A_n$ and $B^\epsilon_{n,x}$ from the corresponding objects in \cite{FehGes19}, namely that \begin{equation}
	\label{eq: A and B} (A_n\setminus A_{n+1})\cap \text{supp}(\psi(\cdot +x, \rho^\epsilon(\cdot +x, \cdot)) \subset B^\epsilon_{n,x}.
\end{equation} For every $(y,t)\in (A_n\setminus A_{n+1})\cap \text{supp}(\psi(\cdot +x, \rho^\epsilon(\cdot +x, \cdot))$,  (\ref{eq: defective concavity}) implies that $$ \vartheta^{-1}(\Phi^{1/2}(\rho))^\epsilon(y+x,t) \le \Phi^{1/2}(\rho^\epsilon(y+x,t)) $$ whence $$ \Phi^{1/2}(\rho(y,t))-\vartheta^{-1}\left(\Phi^{1/2}(\rho)\right)^\epsilon(y+x,t) \ge \Phi^{1/2}(\rho(y,t))-\Phi^{1/2}(\rho^\epsilon(x+y,t))\ge n $$ and we conclude that (\ref{eq: A and B}) holds.   It now follows that, as in the original proof, for some $c>0$, \begin{align} \label{eq: introduce BEN}
	& \int_0^T \int_{(\mathbb{R}^d)^2} \Phi(\rho(y,t))|\epsilon \nabla_x \kappa^\epsilon(y-x)||\psi(x,\rho^\epsilon(x,t))|1_{A_k}(y,t) dy dx dt \\ \nonumber & \le c\int_{\mathbb{R}^d}\sum_{n\ge k} \big(\Phi^{1/2}(M)+n+1\big)^2|B^\epsilon_{n,x}||\epsilon \nabla \kappa^\epsilon(x)| dx.
\end{align} Next, we must estimate the Lebesgue measures of $B^\epsilon_{n,x}$; as a result of the new definition, the argument again diverges from \cite{FehGes19}, and this will ultimately force a choice of $k$. If $(y,t)\in B^\epsilon_{n,x}$, it holds that \begin{equation}\begin{split}|\vartheta^{-1}\Phi^{1/2}(\rho(y,t)) - \vartheta^{-1}(\Phi^{1/2}(\rho))^\epsilon(y+x,t)| & \ge n-(1-\vartheta^{-1})\Phi^{1/2}(\rho(y,t)) \\ & \ge n-(1-\vartheta^{-1})(\Phi^{1/2}(M)+n+1) \end{split} \end{equation} whence \begin{equation}
	|\Phi^{1/2}(\rho(y,t))-(\Phi^{1/2}(\rho))^\epsilon(y+x,t)| \ge n-(\theta-1)(\Phi^{1/2}(M)+1).
\end{equation} We now choose $k$, depending only on $\vartheta, M$ so that, for all $n\ge k$, the right-hand side of the previous display is at least $\frac{n}{2}$. For this choice of $k$ and for $n\ge k$, we have the inclusion \begin{equation}
	B^\epsilon_{n,x}\subset \Big\{(y,t)\in \mathbb{R}^d\times [0,T]: \big|\Phi^{1/2}(\rho(y,t))-\big(\Phi^{1/2}(\rho)\big)^\epsilon(y+x,t)\big|1_{n,n+1}(y,t)\ge \frac{n}{2}\Big\}.
\end{equation} The measure of the sets appearing on the right-hand side can be estimated by repeating the arguments of \cite[Equations (4.19-4.21)]{FehGes19} and, assembling everything and varying the constant $c$, \begin{equation}
	\label{eq: commutator conclusion new} \sum_{n\ge k} \big(\Phi^{1/2}(M)+k+1\big)^2|B^\epsilon_{n,x}| \le c(\epsilon^2+|x|^2)\big\|\nabla \Phi^{1/2}\big\|^2_{L^2(\mathbb{R}^d\times [0,T])}.
\end{equation} Returning to \eqref{eq: introduce BEN}, the integration of the $x$-dependent prefactor against the kernel $\epsilon |\nabla \kappa^\epsilon|$ produces $ \int_{\mathbb{R}^d} (\epsilon^2+|x|^2)\epsilon|\nabla \kappa^\epsilon(x)| dx \le c\epsilon^2$, and (\ref{eq: aim to show comm est}) is proven.   \end{proof}

\begin{proposition}\label{thm_rel_exist}  Let $T,\gamma\in(0,\infty)$, let $\Phi\in\textrm{C}([0,\infty))\cap\textrm{C}^1_{\textrm{loc}}((0,\infty))$ satisfy Assumptions~\ref{as_unique} and \ref{as_equiv}, let $g\in (L^2_{t,x})^d$, and let $\rho_0\in \Ent_{\Phi,\gamma}(\R^d)$.  Then, there exists a a renormalised kinetic solution of \eqref{intro_1} with control $g$ and initial data $\rho_0$.\end{proposition}

\begin{proof}  The proof is an immediate consequence of Proposition~\ref{rel_exist} and Theorem~\ref{thm_equiv}.\end{proof}

We conclude this section with the proof of weak-strong continuity.  Precisely, we will prove that a weakly convergent sequence of controls and initial data induces a strongly convergent sequence of solutions.

\begin{proposition}\label{rel_weak_strong}  Let $T,\gamma\in(0,\infty)$ and let $\Phi\in\textrm{C}([0,\infty))\cap\textrm{C}^1_{\textrm{loc}}((0,\infty))$ satisfy Assumptions~\ref{ap_assume}, \ref{as_unique}, and \ref{as_equiv}.  Then, for every $\rho_{0,n},\rho_0\in \Ent_{\Phi,\gamma}(\R^d)$ and $g_n,g\in (L^2_{t,x})^d$ that satisfy, as $n\rightarrow\infty$,
\[\sup_{n\in\N}\int_{\R^d}\Psi_{\Phi,\gamma}(\rho_{0,n})<\infty\;\;\textrm{and}\;\;\rho_{0,n}\rightharpoonup \rho_0\;\; \textrm{weakly-* in}\;\;L^1_x\;\;\textrm{and}\;\;g_n\rightharpoonup g\;\;\textrm{weakly in}\;\;(L^2_{t,x})^d,\]
the renormalised kinetic solutions $\rho_n$ with controls $g_n$ and with initial data $\rho_{0,n}$ satisfy that, as $n\rightarrow\infty$, for every $R\in(0,\infty)$,
\[\rho_n\rightarrow \rho\;\;\textrm{strongly in}\;\;L^1([0,T];L^1(B_R)),\]
for $\rho$ the unique weak solution of the skeleton equation with control $g$ and initial data $\rho_0$.
\end{proposition}

\begin{proof}  The proof is a consequence of the compactness argument used in Proposition~\ref{rel_exist}, and the uniqueness of the limit implied by Theorem~\ref{thm_rel_unique}.\end{proof}

\section{Equivalence of the Rate Functions}\label{sec_lsc_envelope}

 We now use the analysis of the skeleton equation in the previous sections in order to deduce the full LDP Theorem \ref{thrm: main result} from the partial LDP in Theorem \ref{thrm: partial LDP}.  Henceforth we fix a constant density $\gamma\in(0,\infty)$.  In Theorem \ref{thrm: partial LDP}, we proved upper and lower bounds where the rate functions are the particular instances of the following with $\Phi=\varphi$. For the upper bound, (\ref{eq: Ior}) is a special case of setting
 \begin{align}\label{lsc_0111}
 {I}^{\rm up}(\rho) & :=\int_{\mathbb{R}^d}\Psi_{\Phi, \gamma}(\rho_0)+\sup_{H\in C^{1,3}_{\rm c}([0,T]\times \mathbb{R}^d)} \langle H_T,\pi_T\rangle - \langle H_0, \pi_0\rangle - \int_0^T\langle \partial_t H_t, \pi_t\rangle dt 
 \\ \nonumber & \quad - \frac12 \int_0^T \int_{\mathbb{R}^d}\Phi(\rho_t(x))\left(\Delta H_t(x)+|\nabla H_t(x)|^2\right) dt dx
 \end{align}

if $\rho \in \DD_{\rm a.c.}$ and $\nabla \Phi^{1/2}(\rho)\in (L^2_{t,x})^d$, and $I=\infty$ on $\DD\setminus \DD_{\rm a.c.}$. In the lower bound, (\ref{eq: Ilo}) is the particular case of setting
 \begin{equation}\label{lsc_0112}  I^{\textrm{lo}}(\pi):= \overline{(I^0)_{|\mcS}}(\pi)   \end{equation}  where $I^0$ is the rate function defining the $\liminf$ in \eqref{eq: Ilo} and $\overline{(\cdot)_{|\mathcal{S}}}$ denotes the lower semicontinuous envelope of the restriction to a set $\mathcal{S}$ of regular fluctuations recalled in Definition \ref{smooth_fluctuation} below. Finally, in the previous expression,  $I(\pi)$ is the full rate function analagous to (\ref{lsc_01111'}):%
\begin{equation}\label{lsc_01111}
I(\rho):=\int_{\mathbb{R}^d}\Psi_{\Phi,\gamma}(\rho_0(x))dx+\frac{1}{2}\inf\left\{\norm{g}^2_{L^2}\colon \partial_t\rho = \Delta\Phi(\rho)-\nabla\cdot(\Phi^\frac{1}{2}(\rho)g)\right\}
\end{equation}
when $\rho \in \DD_{\rm a.c.}$ has finite entropy dissipation, and $I=\infty$ on $ \DD\setminus \DD_{\rm a.c.}$. In this section, we will show that, under the Assumptions~\ref{ap_assume}, \ref{as_unique}, and \ref{as_equiv} on $\Phi$, all three rate functions coincide. The deduction of Theorem \ref{thrm: main result} from Theorem \ref{thrm: partial LDP} follows, since the hypotheses (A1-A2) of the theorem, and the consequence \eqref{eq: DH hyp on Phi'}, show that these conditions are satisfied when $\Phi=\varphi$ is the nonlinearity (\ref{eq: def varphi}) arising from the zero-range process.

It is a straightforward consequence of \cite[Lemma~36, Proposition~37]{FehGes19} that, for the rate functions appearing in \eqref{lsc_0111}, \eqref{lsc_0112}, and \eqref{lsc_01111},
\begin{align}\label{lsc_11}
& I(\rho) = I^0(\rho) = I^{\textrm{up}}(\rho) \nonumber
 \\&= \frac{1}{2}\sup_{\{\norm{\psi}_{H^1_{\Phi(\rho)}}\leq 1\}}\Big(-\int_0^T\int_{\R^d}\Phi(\rho)\Delta\psi- \int_0^T\int_{\R^d}\rho\partial_t\psi+\int_{\R^d}\psi(x,s)\rho(x,s)\big|_{s=0}^{s=T}\Big)^2,
\end{align}
where the supremum is taken over smooth functions $\psi\in\textrm{C}^\infty_c(\R^d\times[0,T])$, and the space  $H^1_{\Phi(\rho)}$ is defined in Definition~\ref{weighted_H1} below.  It then remains only to prove that $I^{\textrm{\emph{lo}}}=I$, which is the content of Theorem~\ref{thm:envelope} below.

\begin{definition}\label{weighted_H1} Let $\Phi\in\textrm{C}([0,\infty))\cap\textrm{C}^1_{\textrm{loc}}((0,\infty))$ and $\rho\in L^\infty([0,T];L^1_{\textrm{loc}}(\R^d))$ satisfy $\nabla\Phi^\frac{1}{2}(\rho)\in (L^2_{t,x})^d$.   We first introduce the equivalence $\sim$ on $\textrm{C}^\infty(\R^d\times[0,T])$ by
\[\psi\sim\phi\;\;\textrm{if and only if}\;\;\int_0^T\int_{\R^d}\Phi(\rho)\abs{\nabla\psi-\nabla\phi}^2=0.\]
Let $H^1_{\Phi(\rho)}$ be the Hilbert space completion of the set of equivalence classes $\textrm{C}^\infty(\R^d\times[0,T])/\sim$ with respect to the positive definite inner product, for every $\phi,\psi\in \textrm{C}^\infty_c(\R^d\times[0,T])/\sim$,
\[\langle \phi,\psi \rangle_{H^1_{\Phi(\rho)}}=\int_0^T\int_{\R^d}\Phi(\rho)\nabla\phi\cdot \nabla\psi\;\;\textrm{and}\;\;\norm{\phi}_{H^1_{\Phi(\rho)}}=\langle \phi,\phi\rangle_{H^1_{\Phi(\rho)}}^\frac{1}{2}.\]
\end{definition}

\begin{definition}\label{smooth_fluctuation}  Let $\Phi\in\textrm{C}([0,\infty))\cap\textrm{C}^1_{\textrm{loc}}((0,\infty))$ satisfy Assumptions~\ref{ap_assume}, \ref{as_unique}, and \ref{as_equiv}, let $\gamma\in(0,\infty)$, and let $\C^\infty_{c,\gamma}(\R^d)=\{\psi+\gamma\colon\psi\in\C^\infty_c(\R^d)\}$.   Let $\mathcal{S}\subseteq \mathcal{C}$ denote the subset of all $\rho \in \DD_{\rm a.c.}$ such that there exist a strictly positive $\rho_0\in \textrm{C}^\infty_{c,\gamma}(\R^d)$ and an $H\in C^{3,1}_c(\R^d\times(0,T))$ such that
\[\partial_t\rho  = \Delta\Phi(\rho)-\nabla\cdot\left(\Phi(\rho)\nabla H\right)\;\;\textrm{in}\;\;\R^d\times(0,T)\;\;\textrm{with}\;\;\rho(\cdot,0)=\rho_0,\]
in the sense of Definition~\ref{def_sol_re} with control $g=\Phi^\frac{1}{2}(\rho)\nabla H$.
\end{definition}

\begin{theorem}\label{thm:envelope}  Let $\Phi\in\textrm{C}([0,\infty))\cap\textrm{C}^3_{\textrm{loc}}((0,\infty))$ satisfy Assumptions~\ref{ap_assume}, \ref{as_unique}, and \ref{as_equiv}.  Then, the rate functions defined in \eqref{lsc_0112} and \eqref{lsc_01111} satisfy
\[I^{\textrm{\emph{lo}}}=I.\]
That is, $I$ is the l.s.c.\ envelope of $I^0$ restricted to smooth fluctuations $\mcS$ and $I=I^{\textrm{\emph{lo}}}=I^{\textrm{\emph{up}}}$.
\end{theorem}

\begin{proof} It was shown in \eqref{lsc_11} that $I=I^0=I^{\textrm{up}}$ on $\mcS$.  By definition $I^{\textrm{lo}}$ is then the l.s.c.\ envelope of $I$ restricted to $\mcS$, and it suffices to show that $I$ equals the l.s.c.\ envelope of its restriction to $\mcS$.

Since $I$ is l.s.c.\ we have $I\leq I^{\textrm{lo}}$.  It remains to show that, for any $\rho\in \mathcal{C}$ satisfying $I(\rho)<\infty$, there exists a sequence $\rho_n\in\mathcal{S}$ such that, as $n\rightarrow \infty$,
\[\rho_n\rightarrow\rho\;\;\textrm{strongly in $L^1_{t,x}$}\;\;\textrm{and}\;\;I(\rho_n)\rightarrow I(\rho).\]
The existence of a unique minimiser $g\in (L^2_{t,x})^d$ of the rate function is a consequence of the convexity of the set of controls defining the rate function and the weak lower semicontinuity of the Sobolev norm.  We let $g\in (L^2_{t,x})^d$ be the unique
 control satisfying
\[\partial_t\rho = \Delta \Phi(\rho)-\nabla\cdot(\Phi^\frac{1}{2}(\rho)g)\;\;\textrm{with}\;\;I(\rho)=\frac{1}{2}\norm{g}^2_{L^2}.\]

The construction of the smooth approximation occurs in four steps.  First, we introduce a smoothing of the control $g$ and then a smoothing and perturbation of the initial data $\rho(\cdot,0)\in\Ent_{\Phi,\gamma}(\R^d)$ so that it becomes a strictly positive and compactly supported perturbation of the constant density $\gamma$.  We then introduce a cutoff that ``turns off'' the control in regions where the solution approaches zero or infinity.  This guarantees that the approximate solutions remain strictly bounded away from zero and infinity---and therefore that solutions remain in a region where $\Phi$ is uniformly elliptic---and allows for the application of standard parabolic and elliptic regularity estimates to establish the regularity of the control in Definition~\ref{smooth_fluctuation}.  Finally, we introduce a spatial cutoff in order to obtain a compactly supported control, in accordance with Definition~\ref{smooth_fluctuation}.\

\textit{Step 1:  The smooth approximation.}  Extend $g$ to $\R^d\times\R$ by taking $g=0$ on the complement of $\R^d\times[0,T]$, for every $\ve\in(0,1)$ let $\kappa_{d+1}^\ve$ be a standard convolution kernel on $\R^{d+1}$ of scale $\ve$, and for every $n\in\N$ let $g_n$ be defined by $g_n = (g*\kappa_{d+1}^{\nicefrac{1}{n}})$.  For every $\ve\in(0,1)$ let $\kappa^\ve_d$ be a smooth convolution kernel on $\R^d$, let $\varphi$ be a smooth cutoff of $\overline{B}_1$ in $B_2$ and for every $R\in(0,\infty)$ let $\varphi_R(x)=\varphi(\nicefrac{x}{R})$, and for every $n\in\N$ large enough so that $\nicefrac{1}{n}<\gamma$ let $\rho_{0,n,R}$ be the smooth, bounded, and strictly positive initial data defined by
\[\rho_{0,n} = \big(((\rho(\cdot,0))\vee\nicefrac{1}{n})\wedge n)*\kappa^{\nicefrac{1}{n}}\big)\varphi_R+(1-\varphi_R)\gamma.\]
For every $n\in\N$ let $\psi_n\colon[0,\infty)\rightarrow[0,1]$ be a smooth function satisfying
\[\psi_n = 0\;\;\textrm{on}\;\;[0,\nicefrac{1}{2n}]\cup[2n,\infty)\;\;\textrm{and}\;\;\psi_n = 1\;\;\textrm{on}\;\;[\nicefrac{1}{n},n].\]
A small adaptation of the proofs of Theorem~\ref{thm_rel_unique}, Proposition~\ref{rel_L1-continuity}, and Theorem~\ref{thm_equiv} proves that, for every $n\in\N$, there exists a unique solution $\rho_{n,R}\in \mathcal{C}$ of the equation
\[\partial_t\rho_{n,R} = \Delta\Phi(\rho_{n,R})-\nabla\cdot(\Phi^\frac{1}{2}(\rho_{n,R})\psi_n(\rho_{n,R})g_{n,R})\;\;\textrm{in}\;\;\R^d\times(0,T)\;\;\textrm{with}\;\;\rho_{n,R}(\cdot,0)=\rho_{0,n,R},\]
in the sense of Definition~\ref{def_sol_re} with control $\psi_n(\rho_{n,R})g_n$.  We have using the comparison principle, the definition of $\psi_n$, the choice $\nicefrac{1}{n}<\gamma$, and the definition of $\rho_{0,n,R}$ that
\[\nicefrac{1}{2n}\leq \rho_{n,R}\leq 2n\;\;\textrm{on}\;\;\R^d\times[0,T].\]
Since $\psi_n$, $g_n$, and $\rho_{0,n,R}$ are smooth and bounded, $\Phi\in\textrm{C}^3_{\textrm{loc}}((0,\infty))$ with $\Phi'$ strictly positive on $(0,\infty)$, and since $\rho_{n,R}$ is bounded and bounded away from zero, interior Schauder estimates (see, for example, Lady\u{z}henskaya, Solonnikov, and Ural'ceva \cite{Lad1967}) prove that $\rho_{n,R}\in\textrm{C}^{3,2}(\R^d\times[0,T])$.  We view $\rho_{n,R}$ as satisfying
\[\partial_t\rho_{n,R}  = \Delta\Phi(\rho_{n,R}n)-\nabla\cdot(\Phi^\frac{1}{2}(\rho_{n,R})\tilde{g}_n)\;\;\textrm{in}\;\;\R^d\times(0,T)\;\;\textrm{with}\;\;\rho_{n,R}(\cdot,0)=\rho_{0,n,R},\]
for the control $\tilde{g}_n = \psi_n(\rho_{n,R})g_n$.  The positivity and boundedness of $\rho_{n,R}$ and the positivity of $\Phi$ on $(0,\infty)$ prove that elements of $H^1_{\Phi(\rho_{n,R})}$ can be identified with a unique element of $L^2_tH^1_x$ with zero spatial mean on almost every time slice.  A repetition of the argument in \cite[Proposition~37]{FehGes19}, which amounts essentially to the Riesz representation theorem in the space $H^1_{\Phi(\rho_{n,R})}$, then shows that for every $n\in\N$ sufficiently large and $R\in(0,\infty)$ there exists $H_{n,R}\in L^2_tH^1_x$ satisfying
\begin{equation}\label{lsc_299}\int_0^T\int_{\R^d}\Phi(\rho_{n,R})\abs{\nabla H_{n,R}}^2\leq \int_0^T\int_{\R^d}\abs{\tilde{g}_n}^2\leq \int_0^T\int_{\R^d}\abs{g_n}^2,\end{equation}
such that
\[\partial_t\rho_{n,R}  = \Delta\Phi(\rho_{n,R})-\nabla\cdot(\Phi(\rho_{n,R})\nabla H_{n,R})\;\;\textrm{in}\;\;\R^d\times(0,T)\;\;\textrm{with}\;\;\rho_{n,R}(\cdot,0)=\rho_{0,n,R}.\]
Therefore, for every $n\in\N$ sufficiently large and $R\in(0,\infty)$, $H_{n,R}$ is a solution to the uniformly elliptic equation
\[-\nabla\cdot (\Phi(\rho_{n,R})\nabla H_{n,R}) = \partial_t\rho_{n,R}-\Delta\Phi(\rho_{n,R})\;\;\textrm{in}\;\;\R^d\times(0,T),\]
from which $\nicefrac{1}{2n}\leq \rho_{n,R}\leq 2n$, the local $\textrm{C}^3$-regularity and positivity of $\Phi$ on $(0,\infty)$, the $\textrm{C}^{3,2}$-regularity of $\rho_{n,R}$, and interior elliptic regularity estimates (see, for example, \cite{Lad1967}) prove that $H_{n,R}\in \textrm{C}^{3,1}(\R^d\times(0,T))$.  At this stage, we have constructed a control $H_{n,R}$ satisfying the smoothness criterion of Definition~\ref{smooth_fluctuation}.  However, the $H_{n,R}$ are not necessarily compactly supported in space.  It is for this reason that, in the final step, for every $\tilde{R}\in[1,\infty)$, we let $\rho_{n,R,\tilde{R}}$ be the solution

\[\partial_t\rho_{n,R,\tilde{R}}  = \Delta\Phi(\rho_{n,R,\tilde{R}})-\nabla\cdot(\Phi(\rho_{n,R,\tilde{R}})\nabla (H_{n,R}\varphi_{\tilde{R}}))\;\;\textrm{in}\;\;\R^d\times(0,T)\;\;\textrm{with}\;\;\rho_{n,R,\tilde{R}}(\cdot,0)=\rho_{0,n,R}.\]
Since $H_{n,R}\in\C^{3,1}(\R^d\times(0,T))$ we have using the smoothness of $\varphi$ that $H_{n,R}\varphi_{\tilde{R}}\in\C^{3,1}(\R^d\times(0,T))$ uniformly in $\tilde{R}\in[1,\infty)$.  This completes the proof that $\rho_{n,R,\tilde{R}}\in\mathcal{S}$ for every $n\in\N$ sufficiently large, $R\in(0,\infty)$, and $\tilde{R}\in[1,\infty)$.

\textit{Step 2: The strong convergence of the smooth approximations.}  We will first observe that, due to the fact that the $\nabla(H_{n,R}\varphi_{\tilde{R}})$ are uniformly bounded in $(L^\infty_{t,x})^d\cap (L^2_{t,x})^d$, and since $\rho_{0,n}$ is uniformly bounded away from zero and infinity, the comparison principle proves that the $\rho_{n,R,\tilde{R}}$ are uniformly bounded away from zero and infinity in $\tilde{R}$, for every $n$ and $R$.  Furthermore, the $\rho_{n,R,\tilde{R}}$ satisfy a uniform relative entropy estimate in $\tilde{R}$, for every $n$ and $R$.  More precisely, following Proposition~\ref{rel_prop0}, we have that, for some $c\in(0,\infty)$ independent of $n$, $R$, and $\tilde{R}$,
\begin{align*}
& \sup_{t\in[0,T]}\int_{\R^d}\Psi_{\Phi,\gamma}(\rho_{n,R,\tilde{R}})+\int_0^T\int_{\R^d}\abs{\nabla\Phi^\frac{1}{2}(\rho_{n,R,\tilde{R}})}^2
\\ & \leq c\Big(\int_{\R^d}\Psi_{\Phi,\gamma}(\rho_{0,n,R})+\int_0^T\int_{\R^d}\Phi(\rho_{n,R,\tilde{R}})\abs{\nabla(H_{n,R}\varphi_{\tilde{R}})}^2\Big).
\end{align*}
For the final term on the righthand side,
\begin{align}\label{lsc_0001}
\int_0^T\int_{\R^d}\Phi(\rho_{n,R,\tilde{R}})\abs{\nabla(H_{n,R}\varphi_{\tilde{R}})}^2 & \leq 2\int_0^T\int_{\R^d}\Phi(\rho_{n,R,\tilde{R}})\abs{\nabla H_{n,R}}^2\varphi_{\tilde{R}}^2
\\ \nonumber & \quad + 2\int_0^T\int_{\R^d}\Phi(\rho_{n,R,\tilde{R}})\abs{\nabla\varphi_{\tilde{R}}}^2 H_{n,R}^2.
\end{align}
It follows from the dominated convergence theorem and \eqref{lsc_299} that the first term on the righthand side of \eqref{lsc_0001} converges, as $\tilde{R}\rightarrow\infty$, to
\[\int_0^T\int_{\R^d}\Phi(\rho_{n,R,\tilde{R}})\abs{\nabla H_{n,R}}^2.\]
For the second term on the righthand side of \eqref{lsc_0001}, we have using the uniform boundedness of $\rho_{n,R,\tilde{R}}$ in $\tilde{R}$, for every $n$ and $R$, that, if $d=1$ or $d=2$,
\[\int_0^T\int_{\R^d}\Phi(\rho_{n,R,\tilde{R}})\abs{\nabla\varphi_{\tilde{R}}}^2 H_{n,R}^2\leq cR^{-2}\int_0^T\int_{\textrm{Supp}(\nabla\varphi_{\tilde{R}})} H_{n,R}^2,\]
which, since $H_{n,R}\in\C^{3,1}(\R^d\times(0,T))$, converges to zero if $d=1$ and remains uniformly bounded if $d=2$.  If $d\geq 3$, for $\nicefrac{1}{2_*} = \nicefrac{1}{2}-\nicefrac{1}{d}$,
\[\int_0^T\int_{\R^d}\Phi(\rho_{n,R,\tilde{R}})\abs{\nabla\varphi_{\tilde{R}}}^2 H_{n,R}^2\leq c\int_0^T\Big(\int_{\textrm{Supp}(\nabla\varphi_{\tilde{R}})} H_{n,R}^{2_*}\Big)^\frac{2}{2_*},\]
which, by the dominated convergence theorem, vanishes in the limit $\tilde{R}\rightarrow\infty$ using the standard $H^1$ Sobolev inequality and $\C^{3,1}(\R^d\times(0,T))$ regularity of $H_{n,R}$.  Similarly, the fact that the analogous relative entropy estimate for the $\rho_{n,R}$ holds uniformly in $n$ and $R$ is immediate from Proposition~\ref{rel_prop0} due to the fact that the $H_{n,R}$ are uniformly bounded in $H^1_{\Phi(\rho_{n,R})}$, that uniformly in $n$ and $R$ we have estimate \eqref{lsc_299}, and that, using the convexity of $\Psi_{\Phi,\gamma}$,

\[\sup_{n\in\N, R\in[1,\infty)}\int_{\R^d}\Psi_{\Phi,\gamma}(\rho_{0,n,R})\leq \int_{\R^d}\Psi_{\Phi,\gamma}(\rho_0).\]
The above estimates and a repetition of the methods of Proposition~\ref{rel_exist} prove that, along a subsequence $\tilde{R}\rightarrow\infty$,
\[\lim_{\tilde{R}\rightarrow\infty}\rho_{n,R,\tilde{R}}=\rho_{n,R}\;\;\textrm{in}\;\;L^1([0,T];L^1_{\textrm{loc}}(\R^d)),\]
which, since the limit is unique, proves the convergence along the full sequence.  And, similarly, along the full sequence $n,R\rightarrow\infty$,
\[\lim_{n,R\rightarrow\infty}\rho_{n,R} = \rho\;\;\textrm{in}\;\;L^1([0,T];L^1_{\textrm{loc}}(\R^d)).\]
It remains to prove the convergence in $\mathcal{C}$.  Let $\psi\in \C^\infty_c(\R^d)$ and for every $\d\in(0,1)$ let $\kappa^\d$ be a standard $d$-dimensional convolution kernel on $\R^d$ of scale $\d$.  It follows from the local $L^1$-boundedness of the $\rho_{n,R,\tilde{R}}$ implied by the relative entropy estimate that the paths
\begin{equation}\label{lsc_300}t\rightarrow \int_{\R^d}\rho_{n,R,\tilde{R}}(x,t)\psi(x)\;\;\textrm{are uniformly bounded in $n$, $R$, and $\tilde{R}$.}\end{equation}
Furthermore, following the methods of Proposition~\ref{prop_rel_int} locally in space, for $K\in(0,\infty)$ satisfying $\Supp(\psi)\subseteq B_K$, it follows from the equation that for every $s,t\in[0,T]$, for $\nicefrac{1}{2_*}=\nicefrac{1}{2}-\nicefrac{1}{d}$ if $d\geq 3$ and $2_*=6$ if $d=1$ or $d=2$, for $c\in(0,\infty)$ depending on $\psi$ but independent of $n$, $R$, and $\tilde{R}$,
\begin{align*}
& \abs{\int_{\R^d}\rho_{n,R,\tilde{R}}(x,t)\psi(x)-\int_{\R^d}\rho_{n,R,\tilde{R}}(x,s)\psi(x)}
\\ & \leq c\abs{s-t}^{1-\frac{2}{2_*}}\norm{\Phi^\frac{1}{2}(\rho_{n,R,\tilde{R}})}_{L^\frac{2_*}{2}_tL^2(B_K)}\big(\norm{\nabla\Phi^\frac{1}{2}(\rho_{n,R,\tilde{R}})}_{L^2_{t,x}}+\norm{\nabla(H_{n,R}\varphi_{\tilde{R}})}_{L^2_{t,x}}\big).
\end{align*}
By the Arzel\`a--Ascoli theorem, this implies that the paths \eqref{lsc_300} are tight and converge as $\tilde{R}\rightarrow\infty$ in $\C([0,T])$ to the path
\begin{equation}\label{lsc_302}t\rightarrow \int_{\R^d}\rho_{n,R}(x,t)\psi(x),\end{equation}
where the limit is uniquely identified by the strong Fr\'echet $L^1_{\textrm{loc}}(\R^d)$-convergence above.  Since the limiting paths \eqref{lsc_302} are also uniformly bounded and satisfy using \eqref{lsc_299} that
\begin{align*}
& \abs{\int_{\R^d}\rho_{n,R,\tilde{R}}(x,t)\psi(x)-\int_{\R^d}\rho_{n,R}(x,s)\psi(x,s)}
\\ & \leq c\abs{s-t}^{1-\frac{2}{2_*}}\norm{\Phi^\frac{1}{2}(\rho_{n,R})}_{L^\frac{2_*}{2}_tL^2(B_K)}\big(\norm{\nabla\Phi^\frac{1}{2}(\rho_{n,R})}_{L^2_{t,x}}+\norm{g}_{L^2_{t,x}}\big),
\end{align*}
we again have using the Arzel\`a--Ascoli theorem that these paths are relatively compact $\C([0,T])$, for which the strong Fr\'echet $L^1_{\textrm{loc}}(\R^d)$-convergence proves the uniqueness of the limit and therefore the convergence of the paths \eqref{lsc_302} as $n,R\rightarrow\infty$ in $\C([0,T])$ to the path
\[t\rightarrow \int_{\R^d}\rho(x,t)\psi(x).\]
The density of compactly supported smooth functions in $\C_c(\R^d)$ and the triangle inequality then complete the proof that
\[\lim_{n,R\rightarrow\infty}\Big(\lim_{\tilde{R}\rightarrow\infty}\rho_{n,R,\tilde{R}}\Big)=\rho\;\;\textrm{in}\;\;\mathcal{C}.\]

\textit{Step 3: Convergence of the rate functions.}  It remains only to show that, along a subsequence $n,R,\tilde{R}\rightarrow\infty$, we have that
\[I(\rho_{n,R,\tilde{R}})\rightarrow I(\rho).\]
We first pass to the limit $\tilde{R}\rightarrow\infty$, for which it is a consequence of estimate \eqref{lsc_299}, the analogous interpolation estimates of Proposition~\ref{prop_rel_int}, and the strong Fr\'echet $L^1_{\textrm{loc}}(\R^d)$-convergence of $\rho_{n,R,\tilde{R}}$ to $\rho_{n,R}$ that
\[\lim_{\tilde{R}\rightarrow\infty}\frac{1}{2}\int_0^T\int_{\R^d}\Phi(\rho_{n,R,\tilde{R}})\abs{\nabla(H_{n,R}\varphi_{\tilde{R}})}^2 = \frac{1}{2}\int_0^T\int_{\R^d}\Phi(\rho_{n,R})\abs{\nabla H_{n,R}}^2.\]
It remains to pass to the limit $n,R\rightarrow\infty$.  For this, we know that
\begin{equation}\label{lsc_45}I(\rho_{n,R})=\frac{1}{2}\int_0^T\int_{\R^d}\Phi(\rho_n)\abs{\nabla H_{n,R}}^2 \leq \frac{1}{2}\int_0^T\int_{\R^d}\abs{\psi_n(\rho_{n,R})g_n}^2\leq\frac{1}{2}\norm{g}^2_{L^2_{t,x}}.\end{equation}
After passing to a subsequence $n,R\rightarrow\infty$, for some $\tilde{g}\in (L^2_{t,x})^d$,
\[\Phi^\frac{1}{2}(\rho_{n,R})\nabla H_{n,R}\rightharpoonup \tilde{g}\;\;\textrm{weakly in}\;\;(L^2_{t,x})^d,\]
from which it follows from \eqref{lsc_45} and the weak lower semicontinuity of the Sobolev norm that $\norm{\tilde{g}}_{L^2}\leq\norm{g}_{L^2}$.  Since it follows from Theorem~\ref{thm_rel_unique} and Theorem~\ref{thm_equiv} that $\rho$ solves
\[\partial_t\rho = \Delta \Phi(\rho)-\nabla\cdot(\Phi^\frac{1}{2}(\rho)\tilde{g})\;\;\textrm{in}\;\;\R^d\times(0,\infty)\;\;\textrm{with}\;\;\rho(\cdot,0)=\rho_0,\]
we have by definition of the rate function that $\norm{\tilde{g}}_{L^2}\geq \norm{g}_{L^2}$ and therefore that $\norm{\tilde{g}}_{L^2}= \norm{g}_{L^2}$.  Due to the fact that the minimizer is unique, this proves that $\tilde{g}=g$ and that, along a subsequence,
\[\frac{1}{2}\norm{g}_{L^2}=\frac{1}{2}\norm{\tilde{g}}_{L^2}\leq\liminf_{n,R\rightarrow\infty}I(\rho_{n,R})\leq\limsup_{n,R\rightarrow\infty}I(\rho_{n,R})\leq\frac{1}{2}\limsup_{n\rightarrow\infty}\norm{g_n}_{L^2}= \norm{g}_{L^2},\]
and therefore, along a subsequence,
\[\lim_{n,R\rightarrow\infty}I(\rho_{n,R})= \frac{1}{2}\norm{g}^2_{L^2}=I(\rho),\]
which completes the proof that $I=I^{\textrm{lo}}$, which completes the proof.\end{proof}

\noindent{\bf  Acknowledgements}.  The first author acknowledges support from the National Science Foundation DMS-Probability
Standard Grant 2348650, the Simons Foundation Travel Grant MPS-TSM-00007753, and the
Louisiana Board of Regents RCS Grant 20130014386. The second author acknowledges support by
the Max Planck Society through the Research Group “Stochastic Analysis in the Sciences.” This
work was funded by the European Union (ERC, FluCo, grant agreement No. 101088488). Views
and opinions expressed are however those of the author(s) only and do not necessarily reflect those
of the European Union or of the European Research Council. Neither the European Union nor the
granting authority can be held responsible for them. The third author is  supported by the Royal Commission for the Exhibition of 1851.

\bibliography{references}
\bibliographystyle{alpha}

\end{document}